\documentclass[11pt]{article}

\usepackage{amsmath}
\usepackage{amsfonts}
\usepackage{amsthm}
\usepackage{amssymb}
\usepackage{mathrsfs}
\usepackage{verbatim}
\usepackage{enumerate}
\usepackage{graphicx}
\usepackage{longtable}
\usepackage[all,cmtip]{xy}
\usepackage{upref}
\usepackage[hang,small,bf]{caption}
\usepackage{marginnote}
\usepackage{a4wide}
\usepackage{mathtools,tensor}
\usepackage{xfrac}
\usepackage{color}
\usepackage{bm}
\usepackage{hyperref}
\usepackage{nomencl}
\usepackage{xtab,cuted}

\newtheorem{thm}{Theorem}[section]
\newtheorem{prp}[thm]{Proposition}
\newtheorem{lem}[thm]{Lemma}
\newtheorem{cor}[thm]{Corollary}
\newtheorem{con}{Conjecture}
\newtheorem*{con*}{Conjecture 2}
\theoremstyle{definition}
\newtheorem{dfn}[thm]{Definition}
\newtheorem{rmk}[thm]{Remark}
\newtheorem{exa}[thm]{Example}

\numberwithin{equation}{section}

\newcommand{\T}{{\mathbb T}}
\newcommand{\Z}{{\mathbb Z}}

\newcommand{\R}{{\mathbb R}}

\newcommand{\C}{{\mathbb C}}

\newcommand{\inj}{{\mathrm {inj}}}
\newcommand{\op}[1]{\operatorname{#1}}

\newcommand{\ZZ}{{\mathcal Z}}
\newcommand{\pt}{{\mathrm{pt}}}

\newcommand{\veee}{{\text{\tiny$\vee$}}}

\newcommand{\dist}{{\mathrm{dist}}}

\newcommand{\om}{\mathrm{\omega}}
\newcommand{\vol}{\mathrm{vol}}
\newcommand{\Vol}{\mathrm{Vol}}
\newcommand{\Fix}{\mathrm{Fix\,}}

\newcommand{\Crit}{\mathrm{Crit\,}}
\newcommand{\CAL}{\mathrm{CAL}}
\newcommand{\id}{\mathrm{id}}

\newcommand{\di}{{\mathrm d}}

\newcommand{\Fvol}{{\mathfrak{Vol}}}
\newcommand{\fvol}{{\mathfrak{vol}}}
\newcommand{\ta}{{\mathrm T}}
\newcommand{\dR}{{\mathrm{dR}}}

\newcommand{\ev}{{\mathrm{ev}}}

\newcommand{\p}{\partial}
\newcommand{\into}{\hookrightarrow}
\newcommand{\x}{\times}
\newcommand{\wh}{\widehat}

\newcommand{\beq}{\begin{equation}}
\newcommand{\beqn}{\begin{equation}\nonumber}
\newcommand{\eeq}{\end{equation}}
\newcommand{\bea}{\begin{equation}\begin{aligned}}
\newcommand{\bean}{\begin{equation}\begin{aligned}\nonumber}
\newcommand{\eea}{\end{aligned}\end{equation}}

\AtEndDocument{\bigskip{\footnotesize
  \textsc{Ruprecht-Karls-Universit\"at Heidelberg, Mathematisches Institut, Im Neuenheimer Feld 205, 69120 Heidelberg, Germany} \par  
  \textit{E-mail address}: \texttt{\href{mailto:gbenedetti@mathi.uni-heidelberg.de}{gbenedetti@mathi.uni-heidelberg.de}} \par
  \addvspace{\medskipamount}
\noindent\textsc{Seoul National University, Department of Mathematical Sciences, Research Institute in Mathematics, Gwanak-Gu, 
	Seoul 08826, South Korea} \par  
  \textit{E-mail address}: \texttt{\href{mailto:jungsoo.kang@snu.ac.kr}{jungsoo.kang@snu.ac.kr}} \par
}}

\title{On a local systolic inequality for odd-symplectic forms}
\author{Gabriele Benedetti and Jungsoo Kang}
\begin{document}
\maketitle
\begin{abstract}

The aim of this paper is to formulate a local systolic inequality for odd-symplectic forms (also known as Hamiltonian structures) and to establish it in some basic cases.

Let $\Omega$ be an odd-symplectic form on an oriented closed manifold $\Sigma$ of odd dimension. We say that $\Omega$ is Zoll if the trajectories of the flow given by $\Omega$ are the orbits of a free $S^1$-action. After defining the volume of $\Omega$ and the action of its periodic orbits, we prove that the volume and the action satisfy a polynomial equation, provided $\Omega$ is Zoll. This builds the equality case of a conjectural systolic inequality for odd-symplectic forms close to a Zoll one. We prove the conjecture when the $S^1$-action yields a flat $S^1$-bundle or $\Omega$ is quasi-autonomous.  In particular the conjecture is established in dimension three.

This new inequality recovers the contact systolic inequality as well as the inequality between the minimal action and the Calabi invariant for Hamiltonian isotopies $C^1$-close to the identity on a closed symplectic manifold. Applications to the study of periodic magnetic geodesics on closed orientable surfaces is given in the companion paper \cite{BK19b}.\\[-1.5ex]

\end{abstract}
\tableofcontents
\section{Introduction}
In this article we formulate a local systolic-diastolic inequality for odd-symplectic forms and establish it in some situations.  Before discussing the precise set-up, we recall two inequalities which motivate us to study the local systolic-diastolic inequality for odd-symplectic forms and which arise as extreme cases of this inequality. Throughout the paper, we normalise the circle $S^1$ to have length one, namely $S^1:=\R/\Z$.
\subsection{Two known inequalities}
\subsubsection*{Contact systolic-diastolic inequality}
Let $\Sigma$ be a connected closed manifold of dimension $2n+1$. Let $\alpha$ be a contact form on $\Sigma$, namely $\alpha$ is a one-form such that $\alpha\wedge(\di\alpha)^n$ is nowhere vanishing. Orienting $\Sigma$ so that $\alpha\wedge(\di\alpha)^n$ is positive, we define the total volume of $\Sigma$ with respect to $\alpha$ by
\begin{equation}\label{eq:contact_volume}
\mathrm{Volume}(\alpha):=\int_{\Sigma}\alpha\wedge(\di\alpha)^n>0.
\end{equation}
We consider the Reeb vector field $R_\alpha$ on $\Sigma$ characterised by the relations $\di\alpha(R_\alpha,\cdot)=0$ and $\alpha(R_\alpha)=1$ and denote by $\Phi_{\alpha}$ the generated flow.

A contact form $\alpha_*$ on $\Sigma$ is called \textit{Zoll} if all orbits of $\Phi_{\alpha_*}$ are periodic with the same prime period $T(\alpha_*)$. This is equivalent to saying that $\Phi_{\alpha_*}$ induces a free circle action on $\Sigma$ with period $T(\alpha_*)$. In this case, the quotient is a closed symplectic manifold $(M,\omega)$ and the quotient map $\mathfrak p:\Sigma\to M$ is a non-trivial circle bundle. More precisely, we have $\mathfrak p^*\omega=\di\alpha_*$ so that $-\tfrac{1}{T(\alpha_*)}\omega$ represent the Euler class of the bundle. Vice versa, for every closed symplectic manifold $(M,\omega)$ such that $\frac1T\omega$ represents an integer class for some $T>0$, there exists a manifold $\Sigma$ and a Zoll contact form with period $T$ yielding $(M,\omega)$ according to the construction above, see \cite{BW58}.

By \cite{Bot80}, if $\mathfrak h$ denotes the free homotopy class of prime periodic orbits of $\Phi_{\alpha_*}$, then $\Phi_\alpha$ carries a periodic orbit in the class $\mathfrak h$, if $\alpha$ is sufficiently close to $\alpha_*$ .
We denote by $T_{\min}(\alpha,\mathfrak h)$, respectively $T_{\max}(\alpha,\mathfrak h)$ the minimal, respectively maximal, period of prime periodic orbits of $\Phi_\alpha$ in the class $\mathfrak h$. If $\Sigma$ has dimension three, $T_{\min}(\alpha,\mathfrak h)$ and $T_{\max}(\alpha,\mathfrak h)$ satisfy the following contact systolic-diastolic inequality, which was originally conjectured by \'Alvarez-Paiva and Balacheff \cite{APB14}.

\begin{thm}\label{thm:contact_sys}\cite[Theorem 1.4]{BK19a}
Let $\Sigma$ be a closed manifold of dimension three, and let $|H_1^\mathrm{tor}(\Sigma;\Z)|$ be the order of the torsion part of $H_1(\Sigma,\Z)$. If $\alpha_*$ is a Zoll contact form on $\Sigma$, there exists a $C^2$-neighbourhood $\mathcal U$ of $\di\alpha_*$ such that for every contact form $\alpha$ on $\Sigma$ with $\di\alpha\in\mathcal U$, there holds
\[
T_{\min}(\alpha,\mathfrak h)^2\leq \frac{1}{|H_1^\mathrm{tor}(\Sigma;\Z)|}\mathrm{Volume}(\alpha)\leq T_{\max}(\alpha,\mathfrak h)^2,
\]
where any of the two equalities holds if and only if $\alpha$ is Zoll.
\end{thm}
\begin{rmk}
When $\Sigma=S^3$ the result above is due to Abbondandolo, Bramham, Hryniewicz and Salom\~ao, see \cite{ABHS15}. In that paper, forms $\alpha$ are also constructed, which are $C^0$-close to $\alpha_*$ and violate the systolic inequality. Furthermore, the same authors showed in \cite{ABHSnew} that, if $(\Sigma,\xi)$ is an arbitrary contact closed three-manifold, then, for every $C>0$, there exists a contact form $\alpha_C$ with $\ker\alpha_C=\xi$ such that $T_{\min}(\alpha_C)^2>C\cdot\mathrm{Volume}(\alpha_C)$.
\end{rmk}

\subsubsection*{Symplectic systolic-diastolic inequality}
Let $(M,\om)$ be a connected closed symplectic manifold of dimension $2n$, and let $\varphi:M\to M$ be a Hamiltonian diffeomorphism generated by a Hamiltonian $H:M\times[0,1]\to\R$. We orient $M$ so that $\om^n$ is positive. One can define the Hamiltonian action $\mathcal A_H:\Lambda_{\mathrm{short}}(M)\to\R$ on the space of ``short'' one-periodic curves $q:S^1\to M$  by
\begin{equation}\label{e:hamact}
\mathcal A_H(\gamma):=\int_{D^2}{\widehat q}^{\hspace{2pt}*}\om+\int_0^1 H(q(t),t)\di t,
\end{equation}
where $\widehat q:D^2\to M$ is a ``small'' capping disc for $q\in\Lambda_{\mathrm{short}}(M)$ (see Section \ref{sec:conjsys} for a precise definition). The minimal and maximal Hamiltonian actions of fixed points of $\varphi$, whose associated curve is short, are given by 
\[
\min\mathcal A_H:=\inf_{\gamma\in\Crit\mathcal A_H}\mathcal A_H(\gamma),\qquad \max\mathcal A_H:=\sup_{\gamma\in\Crit\mathcal A_H}\mathcal A_H(\gamma).
\]
The Calabi invariant of $H$ with respect to $\om$ is defined by
\begin{equation}\label{eq:calabi_intro}
\CAL_{\om}(H):=\int_{M\times[0,1]}H\,\omega^n\wedge\di t.
\end{equation}
One can easily see that the minimal/maximal actions and the Calabi invariant are related in the following way. It is one particular case of Proposition \ref{p:qa1}.
\begin{prp}\label{p:hamsys}
Assume that a Hamiltonian diffeomorphism $\varphi:(M,\om)\to (M,\om)$ is generated by a quasi-autonomous Hamiltonian $H:M\times[0,1]\to\R$. For instance, $\varphi$ is the time-one map of a Hamiltonian isotopy $C^1$-close to the identity. Then, there holds
\begin{equation*}
\min\mathcal A_H\leq \frac{\CAL_\om(H)}{\int_M\om^n}\leq\max\mathcal A_H,
\end{equation*}
and any of the two equalities holds if and only if $H(q,t)=h(t)$ for all $(q,t)\in M\times[0,1]$ and some function $h:[0,1]\to\R$.\qed
\end{prp}
\begin{rmk}
Proposition \ref{p:hamsys} is false for general Hamiltonian diffeomorphisms. Indeed, adapting \cite[Proposition 2.28]{ABHS15}, one can construct a time-one map $\varphi$ of a Hamiltonian isotopy $C^0$-close to the identity not satisfying the inequality above. A proper way to go beyond the quasi-autonomous case might be to take into account not only the fixed points but all the periodic points of $\varphi$ as explored by Hutchings in \cite[Theorem~1.2]{Hut16}.
\end{rmk}
 
\begin{rmk}
In \cite{BK19a} we saw how the use of a global surface of section reduces Theorem \ref{thm:contact_sys} to a generalised version of Proposition \ref{p:hamsys} for surfaces with boundary, where $\omega$ is symplectic in the interior and vanishes of order one at the boundary.
\end{rmk}

To give a unified model for the above two situations, we notice that the Hamiltonian diffeomorphism $\varphi:M\to M$ can be seen as the return map of a flow $\Phi^X$ on the manifold $M\times S^1$. If $H:M\times S^1\to\R$ is a Hamiltonian one-periodic in time generating $\varphi$ as time one-map, then $X(q,t):=\p_t+X_{H_t}(q)$, for all $(q,t)\in M\times S^1$. Here, $X_{H_t}$ is the Hamiltonian vector field tangent to $M\times\{t\}$ associated with the function $H_t:=H(\cdot,t)$. Thus, we view Proposition \ref{p:hamsys} as giving a systolic-diastolic inequality for those flows on the trivial $S^1$-bundle $\mathfrak p:M\times S^1\to M$, which are obtained by lifting a Hamiltonian isotopy of $M$. On the other hand, Theorem \ref{thm:contact_sys} yields a systolic-diastolic inequality for Reeb flows on the non-trivial $S^1$-bundles obtained from the Boothby-Wang construction. We will show that these two types of flows are particular incarnations of the flow induced by an odd-symplectic structure (also known as Hamiltonian structure, see \cite{CM05}) on an oriented odd-dimensional closed manifold which is the total space of an \emph{arbitrary} oriented circle bundle over a symplectic manifold. Moreover, we will formulate a conjectural systolic-diastolic inequality in this setting which recovers the contact (Theorem \ref{thm:contact_sys}) and the symplectic one (Proposition \ref{p:hamsys}) as two extreme cases.

\subsection{Odd-symplectic systolic geometry}

Let $(\Sigma,\mathfrak o_\Sigma)$ be a connected oriented closed manifold of dimension $2n+1$, and let $C\in H^2_{\mathrm{dR}}(\Sigma)$ be a class in its de Rham cohomology. We write $\Omega^1(\Sigma)$ for the set of one-forms on $\Sigma$ and $\Xi^2_C(\Sigma)$ for the set of closed two-forms on $\Sigma$ representing the class $C$. We pick an auxiliary element $\Omega_0\in \Xi^2_C(\Sigma)$ and get a surjective map
\begin{equation}\label{e:omega1xi2}
\Omega^1(\Sigma)\to \Xi^2_C(\Sigma),\qquad \alpha\mapsto \Omega_\alpha:=\Omega_0+\di\alpha,
\end{equation}
whose kernel is the space of closed one-forms on $\Sigma$. Using this map, we can associate to $\alpha\in\Omega^1(\Sigma)$ a real number $\Vol(\alpha)$, which generalises both the contact volume \eqref{eq:contact_volume} and the Calabi invariant \eqref{eq:calabi_intro}. Namely, we define a functional $\Vol:\Omega^1(\Sigma)\to \R$ by
\[
\Vol(0)=0,\qquad \di_{\alpha}\Vol\cdot \beta=\int_\Sigma\beta\wedge\Omega_\alpha^{n},\quad \forall\,\beta\in\Omega^1(\Sigma).
\]
For instance, when $n=1$, we have
\begin{equation*}
\Vol(\alpha)=\frac12\int_\Sigma \alpha\wedge\di\alpha+\int_\Sigma\alpha\wedge\Omega_0.
\end{equation*}
We also remark that, when $C=0$ and $\Omega_0=0$, the function $\Vol$ recovers the Chern-Simons action for principal $S^1$-bundles over $\Sigma$ (see \cite{CS74} and Remark \ref{r:cs}). For $n=1$, $\Omega_0=0$ and $\di\alpha=\iota_F\mu$ for some vector field $F$ on $\Sigma$, whose flow preserves a global volume form $\mu$ on $\Sigma$, $\Vol(\alpha)$ recovers the helicity of $F$, see \cite{Pat09}.

If $C^{n}=0$, then we can use the map \eqref{e:omega1xi2} to push forward $\Vol$ to a \textbf{volume functional}
\[
\Fvol:\Xi^2_C(\Sigma)\to\R.
\]
If $C^{n}\neq 0$, this procedure does not work, since there exists $\tau\in H^1_\dR(\Sigma)$ such that $\tau\cup C^n\neq 0$. In this case, we say that $\alpha\in\Omega^1(\Sigma)$ is \textbf{normalised}, if $\Vol(\alpha)=0$. For every $\Omega\in\Xi^2_C(\Sigma)$, there exists $\alpha\in\Omega^1(\Sigma)$ normalised such that $\Omega=\Omega_\alpha$. Therefore, in this situation we can just work with normalised forms and declare $\Fvol:\Xi^2_C(\Sigma)\to\R$ to be identically zero. In both cases, the volume functional is invariant under diffeomorphisms $\Psi:\Sigma\to\Sigma$ isotopic to the identity (see Proposition \ref{prp:vol}):
\[
\Fvol(\Psi^*\Omega)=\Fvol(\Omega).
\]

Having identified what the volume should be, we want to introduce the set $\mathcal X(\Omega)$ of closed characteristics of $\Omega$ generalising periodic Reeb and Hamiltonian orbits. To this purpose, we consider the possibly singular distribution $\ker\Omega\to \Sigma$ called the \textbf{characteristic distribution} and define
\[
\mathcal X(\Omega):=\Big\{\gamma:S^1\hookrightarrow \Sigma\ \Big|\ \dot\gamma\in \ker\Omega\Big\}\Big/\!\sim\,,
\]
where $\gamma_1\sim\gamma_2$ if and only if $\gamma_1$ and $\gamma_2$ coincide up to an orientation-preserving reparametrisation of $S^1$. The distribution $\ker\Omega$ is co-oriented by $\Omega^n$ and we orient it using the given orientation $\mathfrak o_\Sigma$ on $\Sigma$.

We single out the forms in $\Xi^2_C(\Sigma)$, whose characteristic distribution is one-dimensional. They are of special importance, as their closed characteristics are the periodic orbits of a flow. 
\begin{dfn}
A two-form $\Omega$ on $\Sigma$ is said to be \textbf{odd-symplectic}, if it is closed and maximally non-degenerate, namely if its characteristic distribution is an (oriented) real line bundle over $\Sigma$. We write $\mathcal S_C(\Sigma)$ for the subset of all odd-symplectic forms in $\Xi^2_C(\Sigma)$.  
\end{dfn}
In $\mathcal S_C(\Sigma)$ we hope to find elements whose characteristic distribution is as simple as possible. To this purpose we introduce the set of oriented $S^1$-bundles with total space $\Sigma$:
\[
\mathfrak P(\Sigma):=\big\{\mathfrak p:\Sigma\to M\ \big|\ \mathfrak p\ \text{is an oriented $S^1$-bundle}\big\}.
\]
\begin{dfn}
An odd-symplectic form $\Omega$ is said to be \textbf{Zoll} if the oriented leaves of its characteristic distribution are the fibres of some $\mathfrak p_\Omega:\Sigma\to M_\Omega$ in $\mathfrak P(\Sigma)$. We write $\ZZ_C(\Sigma)$ for the set of Zoll forms in $\mathcal S_C(\Sigma)$.
\end{dfn}
\begin{exa}
Let $(N_1,\sigma_1)$ and $(N_2,\sigma_2)$ be two connected closed symplectic manifolds and suppose that $[\sigma_1]\in H^2_\dR(N_1)$ is an integral cohomology class. We build the product symplectic manifold $(N_1\times N_2,\sigma_1\oplus\sigma_2)$ and consider the oriented $S^1$-bundle $\mathfrak p:\Sigma\to N_1\times N_2$ whose Euler class is $-[\sigma_1\oplus 0]$. Then, $\Omega:=\mathfrak p^*(\sigma_1\oplus\sigma_2)$ is a Zoll form with cohomology class $C:=\mathfrak p^*[0\oplus\sigma_2]$.
\end{exa}

Odd-symplectic Zoll forms $\Omega$ with vanishing cohomology class are exactly the differentials of Zoll contact forms $\alpha$. We can use this observation to extend the classification of Zoll contact forms on three-manifolds contained in \cite[Proposition 1.2]{BK19a} to Zoll odd-symplectic forms.
\begin{prp}\label{p:cla}
Let $\Sigma$ be a connected oriented closed three-manifold. There is a Zoll odd-symplectic form on $\Sigma$ if and only if $\Sigma$ is the total space of an oriented $S^1$-bundle over a connected oriented closed surface $M$. When the bundle is non-trivial, then every Zoll odd-symplectic form is the differential of a Zoll contact form. When the bundle is trivial, then the group $H_1(\Sigma;\Z)$ is free with rank equal to $2\, \mathrm{genus}(M)+1$, and if $\Omega$ and $\Omega'$ are Zoll odd-symplectic forms on $\Sigma$, there exists a real number $T>0$ and a diffeomorphism $\Psi:\Sigma\to\Sigma$ such that $\Psi^*\Omega'=T\Omega$. Moreover,
\begin{itemize}
	\item if $\Sigma=S^2\times S^1$, the set $\ZZ_C(\Sigma)$ has exactly two connected components, $\forall\, C\in H^2_{\dR}(\Sigma)\setminus 0$;
	\item if $\Sigma=\T^2\times S^1$, the set $\ZZ_C(\Sigma)$ is non-empty and connected, $\forall\, C\in H^2_{\dR}(\Sigma)\setminus 0$;
	\item if $\Sigma=M\times S^1$ with $\chi(M)<0$, there is a one-dimensional subspace $L\subset H^2_{\dR}(\Sigma)$ such that $\ZZ_C(\Sigma)$ is non-empty if and only if $C\in L\setminus 0$; in this case the set $\ZZ_C(\Sigma)$ is connected.
\end{itemize}
\end{prp}
\begin{rmk}
The classification of Zoll odd-symplectic forms up to diffeomorphism on a three-manifold is equivalent to the classification of bundles in the set
$\{\mathfrak p_\Omega\ |\ \Omega\in\ZZ(\Sigma)\}$ up to isomorphism. Analogously, the connected components of $\ZZ_C(\Sigma)$ on a three-manifold are in bijection with the connected components of $\{\mathfrak p_\Omega\ |\ \Omega\in\ZZ_C(\Sigma)\}$. This is due to the fact that the map $\ZZ_C(\Sigma)\to\mathfrak P(\Sigma)$, given by $\Omega\mapsto\mathfrak p_\Omega$, has contractible fibres.
\end{rmk}
Let us assume that $\ZZ_C(\Sigma)$ is not empty and take $\Omega_*\in\ZZ_C(\Sigma)$ with associated $S^1$-bundle 
\[
\mathfrak p_{\Omega_*}:\Sigma\to M_*.
\] 
This implies that there exists a positive symplectic form $\om_*$ on $M_*$ such that $\Omega_*=\mathfrak p_{\Omega_*}^*\om_*$. We set $c_*:=[\om_*]\in H^2_\dR(M_*)$. We write $\mathfrak h\in[S^1,\Sigma]$ for the free-homotopy class of the oriented $\mathfrak p_{\Omega_*}$-fibres and $e_*\in H^2_\dR(M_*)$ for minus the real Euler class of $\mathfrak p_{\Omega_*}$. Let $\mathfrak P^0(\Sigma)$ be the connected component of $\mathfrak p_{\Omega_*}$ inside $\mathfrak P(\Sigma)$. As we did for the volume, the action will be computed with respect to some reference object, which we now define.
\begin{dfn}
A \textbf{weakly Zoll pair} is a couple $(\mathfrak p,c)$, where $\mathfrak p:\Sigma\to M$ is an element in $\mathfrak P(\Sigma)$ and $c\in H^2_\dR(M)$ is a cohomology class. We write $\mathfrak Z(\Sigma)$ for the set of weakly Zoll pairs and $\mathfrak Z_C(\Sigma)$ for the subset of those pairs $(\mathfrak p,c)$ such that $C=\mathfrak p^*c\in H_\dR^2(\Sigma)$. We denote by $\mathfrak Z_C^0(\Sigma)$ the set of those $(\mathfrak p,c)\in\mathfrak Z_C(\Sigma)$ such that $\mathfrak p\in\mathfrak P^0(\Sigma)$.
\end{dfn}
\begin{rmk}
If $\Omega\in\mathcal S_C(\Sigma)$ is Zoll, then $(\mathfrak p_\Omega,[\om_\Omega])\in\mathfrak Z_C(\Sigma)$ is a weakly Zoll pair, where $\om_\Omega$ is the closed two-form on $M_\Omega$ such that $\Omega=\mathfrak p_\Omega^*\om_\Omega$. Conversely, if $(\mathfrak p,c)\in\mathfrak Z_C(\Sigma)$ is a weakly Zoll pair, we can consider any closed two-form $\om$ on $M$ such that $c=[\om]$ and build the two-form $\mathfrak p^*\om\in\Xi^2_C(\Sigma)$, which is Zoll exactly when $\om$ is symplectic. 
\end{rmk}
Let us now fix a {\bf reference} weakly Zoll pair 
\[
(\mathfrak p_0,c_0)\in\mathfrak Z_C^0(\Sigma),\qquad \mathfrak p_0:\Sigma\to M_0.
\] 
We write $e_0\in H^2_{\dR}(M_0)$ for minus the real Euler class of $\mathfrak p_0$ and we have the equivalence
\[
C^n\neq0\qquad \Longleftrightarrow\qquad e_0=0.
\]
Let $\Lambda_{\mathfrak h}(\Sigma)$ be the space of one-periodic curves in the class $\mathfrak h$, and let $\widetilde{\Lambda}_{\mathfrak h}(\Sigma)$ be the space of homotopies of paths $\{\gamma_r\}_{r\in[0,1]}$ inside $\Lambda_{\mathfrak h}(\Sigma)$ such that $\gamma_0$ is some oriented $\mathfrak p_0$-fibre. The admissible homotopies are allowed to move $\gamma_0$ inside the set of oriented $\mathfrak p_0$-fibres but have to fix the periodic curve $\gamma_1$, so that the natural projection $\widetilde{\Lambda}_{\mathfrak h}(\Sigma)\to \Lambda_{\mathfrak h}(\Sigma)$, $[\gamma_r]\mapsto \gamma_1$ is a covering map.

We pick a closed two-form $\om_0\in\Xi^2_{c_0}(M_0)$ and choose $\Omega_0:=\mathfrak p_0^*\om_0\in\Xi^2_C(\Sigma)$ as our auxiliary element for the computation of the volume. One sees that $\Fvol:\Xi^2_C(\Sigma)\to\R$ depends only on $(\mathfrak p_0,c_0)$ and not on $\om_0$. We associate to $\Omega=\Omega_0+\di\alpha$ the \textbf{action functional} 
\begin{equation*}
\widetilde {\mathcal A}_{\Omega}:\widetilde\Lambda_{\mathfrak h}(\Sigma)\to\R,\qquad\widetilde{\mathcal A}_{\Omega}\big([\gamma_r]\big):=\int_{[0,1]\times S^1}\Gamma^*\Omega + \int_{S^1}\gamma_0^*\alpha,
\end{equation*}
where $\alpha$ is chosen to be normalised if $C^n\neq0$, and $\Gamma:[0,1]\times S^1$ is the cylinder traced by the path $\{\gamma_r\}$. Like the volume, the functional $\widetilde{\mathcal A}_{\Omega}$ depends only on $(\mathfrak p_0,c_0)$ and not on $\om_0$. If $[\gamma_r]$ is a critical point of $\widetilde{\mathcal A}_{\Omega}$, then $\gamma_1\in\mathcal X(\Omega)$, provided $\gamma_1$ is embedded.
Furthermore, the action is invariant under isotopies $\{\Psi_r:\Sigma\to\Sigma\}$ starting at the identity (Proposition \ref{prp:action_inv_isotopy}):
\[
\widetilde {\mathcal A}_{\Psi_1^*\Omega}\big([\Psi^{-1}_r\circ\gamma_r]\big)=\widetilde{\mathcal A}_{\Omega}\big([\gamma_r]\big),\qquad\forall\,[\gamma_r]\in\widetilde\Lambda_{\mathfrak h}(\Sigma).
\]
In general, $\widetilde{\mathcal  A}_\Omega$ does not descend to a functional on $\Lambda_{\mathfrak h}(\Sigma)$. More precisely, by Lemma \ref{l:exact} this happens if and only if $c_0|_{\pi_2(M_0)}=a e_0|_{\pi_2(M_0)}$ for some $a\in\R$. However, as we see now, we can define an action on the set of Zoll forms $\Omega$ with $(\mathfrak p_\Omega,[\om_\Omega])\in\mathfrak Z_C^0(\Sigma)$. In this case, the volume of $\Omega$ can be expressed as a polynomial function of the action. 
\begin{dfn}
The \textbf{Zoll polynomial} $P:\R\to\R$ associated with $(\mathfrak p_0,c_0)$ is given by 
\[
P(0)=0,\qquad \frac{\di P}{\di A}(A)=\langle (Ae_0+c_0)^{n},[M_0]\rangle.
\]
For instance, when $n=1$, the polynomial reads
\begin{equation*}
P(A)=\frac12 \langle e_0,[M_0]\rangle A^2+\langle c_0,[M_0]\rangle A.
\end{equation*}
\end{dfn}
\begin{thm}\label{t:oszoll}
There is a well-defined volume function
\[
\Fvol:\mathfrak Z_C(\Sigma)\to\R,\qquad \Fvol(\mathfrak p,c):=\Fvol(\mathfrak p^*\om)
\]
and well-defined action functional
\[
\mathcal A:\mathfrak Z_C^0(\Sigma)\to\R,\qquad \mathcal A(\mathfrak p,c):=\widetilde{\mathcal A}_{\mathfrak p^*\om}\big([\delta_r]\big).
\]
Here, $\om$ is any closed two-form on $M$ in the class $c$, $\{\delta_r\}$ is any path of periodic curves such that $\delta_r$ is an oriented $\mathfrak p_r$-fibre, where $\{\mathfrak p_r\}$ is any path of oriented $S^1$-bundles from $\mathfrak p_0$ to $\mathfrak p_1=\mathfrak p$.
Moreover, there holds
\begin{equation*}
P\big(\mathcal A(\mathfrak p,c)\big)=\Fvol(\mathfrak p,c),\qquad\forall\,(\mathfrak p,c)\in\mathfrak Z_C^0(\Sigma). 
\end{equation*}
If $A_*:=\mathcal A(\mathfrak p_{\Omega_*},c_*)$, then $\tfrac{\di P}{\di A}(A_*)>0$. In particular, the polynomial $P$ is non-zero.
\end{thm}
From this result, we can generalise the equality cases in Theorem \ref{thm:contact_sys} and Proposition \ref{p:hamsys}.
\begin{cor}\label{c:zollequality}
Let $\Omega\in\ZZ_C(\Sigma)$ be a Zoll odd-symplectic form such that $\mathfrak p_\Omega\in\mathfrak P^0(\Sigma)$. If we set $\mathcal A(\Omega):=\mathcal A(\mathfrak p_\Omega,[\om_\Omega])$, then
\[
P(\mathcal A(\Omega))=\Fvol(\Omega).\eqno\qed
\]
\end{cor}
In what follows, we describe a conjectural systolic-diastolic inequality for odd-symplectic forms close to $\Omega_*$ and with class $C\in H^2_{\dR}(\Sigma)$. To this end, we fix a finite open covering $\{B_i\}$ of $M_*$ by balls so that all their pairwise intersections are also contractible. Let $\Lambda(\mathfrak p_{\Omega_*})$ be the space of periodic curves $\gamma\in\Lambda_{\mathfrak h}(\Sigma)$ with the property that $\mathfrak p_{\Omega_*}(\gamma)$ is contained in some $B_i$ and there is a path $\{\gamma_r^\mathrm{short}\}\in\widetilde\Lambda_{\mathfrak h}(\Sigma)$ entirely contained in $\mathfrak p_{\Omega_*}^{-1}(B_i)$ with $\gamma^\mathrm{short}_1=\gamma$. If $\Omega\in\Xi^2_C(\Sigma)$, we set
\begin{equation*}
\mathcal A_{\Omega}:\Lambda(\mathfrak p_{\Omega_*})\to\R,\qquad\mathcal A_\Omega(\gamma):=\widetilde{\mathcal A}_{\Omega}\big([\delta_r\#\gamma^\mathrm{short}_r]\big),\qquad \gamma^\mathrm{short}_1=\gamma.
\end{equation*}
Here, the symbol $\#$ denotes the concatenation of paths, $\{\delta_r\}$ is any path of periodic curves such that $\delta_1=\gamma_0^\mathrm{short}$, and for every $r\in[0,1]$, $\delta_r$ is an oriented $\mathfrak p_r$-fibre, where $\{\mathfrak p_r\}$ is a path of oriented $S^1$-bundles connecting $\mathfrak p_0$ with $\mathfrak p_1=\mathfrak p_{\Omega_*}$. We define
\[
\mathcal A_{\min}(\Omega):=\inf_{\gamma\in\mathcal X(\Omega)\cap\Lambda(\mathfrak p_{\Omega_*})}\mathcal A_{\Omega}(\gamma),\qquad \mathcal A_{\max}(\Omega):=\sup_{\gamma\in\mathcal X(\Omega)\cap\Lambda(\mathfrak p_{\Omega_*})}\mathcal A_{\Omega}(\gamma).
\]
By \cite[Section III]{Gin87} or \cite[Section 3.2]{APB14}, if $\Omega\in \mathcal S_C(\Sigma)$ is $C^1$-close to $\Omega_*$, the set $\mathcal X(\Omega)\cap\Lambda(\mathfrak p_{\Omega_*})$ is compact and non-empty. Furthermore, the numbers $\mathcal A_{\min}(\Omega)$ and $\mathcal A_{\max}(\Omega)$ are finite and vary $C^1$-continuously with $\Omega$.

\begin{con}[Local systolic-diastolic inequality for odd-symplectic forms]\label{con:lsios}
Let $\Omega_*$ be a Zoll odd-symplectic form with cohomology class $C\in H^2_\dR(\Sigma)$. There is a $C^{k-1}$-neighbourhood $\mathcal U$ of $\Omega_*$ in $\mathcal S_C(\Sigma)$ with $k\geq 2$ such that
\begin{equation*}
P(\mathcal A_{\min}(\Omega))\leq\Fvol(\Omega)\leq P(\mathcal A_{\max}(\Omega)),\qquad \forall\,\Omega\in \mathcal U.
\end{equation*}
The equality holds in any of the two inequalities above, if and only if $\Omega$ is Zoll.
\end{con}
\begin{rmk}
If the real Euler class of the bundle associated with $\Omega_*$ vanishes, then the inequality in the conjecture is equivalent to
\[
\mathcal A_{\min}(\Omega)\leq0\leq \mathcal A_{\max}(\Omega),\qquad \forall\,\Omega\in \mathcal U,
\]
with any of the equalities holding if and only if $\Omega$ is Zoll.
\end{rmk}
\begin{rmk}
The volume, the action and the Zoll polynomial depend on the choice of reference pair $(\mathfrak p_0,c_0)\in\mathfrak Z_C^0(\Sigma)$. However, thanks to Theorem \ref{t:oszoll} we will observe in Remark \ref{r:indepconj} that the functional
\[
\big(P\circ \mathcal A_\Omega-\Fvol(\Omega)\big):\Lambda(\mathfrak p_{\Omega_*})\to\R
\]
is independent of such a choice. Hence, Conjecture \ref{con:lsios} remains the same when we use another reference pair in $\mathfrak Z_C^0(\Sigma)$.
\end{rmk}
\begin{rmk}
As mentioned at the beginning of this subsection, Conjecture \ref{con:lsios} recovers the contact and symplectic systolic-diastolic inequality. In the {\it contact} case, we have $\Omega_*=\di \alpha_*$ for a contact form $\alpha_*$ so that $C=0$. Up to multiplication by a positive constant, we can assume that $T(\alpha_*)=1$. We take $\mathfrak o_\Sigma=\mathfrak o_{\alpha_*}$ and $c_0=0$. There holds
\[
P(A)=\frac{1}{n+1}\langle e_0^{n},[M_0]\rangle A^{n+1},
\]
where $\langle e_0^{n},[M_0]\rangle=\langle [\om_*^n],[M_*]\rangle>0$. When $n=1$, we have $\langle e_0,[M_0]\rangle=|H_1^\mathrm{tor}(\Sigma;\Z)|$. If $\Omega=\di\alpha$ is $C^{0}$-close to $\Omega_*$, we can take $\alpha$ to be a contact form. We have
\[
\Fvol(\di \alpha)=\frac1n\mathrm{Volume}(\alpha),\qquad \mathcal A_{\di\alpha}(\gamma)=\int_{S^1}\gamma^*\alpha=T(\gamma),\quad \forall\,\gamma\in\Lambda_{\mathfrak h}(\Sigma).
\]
Thus, there holds $T_{\min}(\alpha,\mathfrak h)\leq \mathcal A_{\min}(\di\alpha)\leq \mathcal A_{\max}(\di\alpha)\leq T_{\max}(\alpha,\mathfrak h)$, and Conjecture \ref{con:lsios} implies
\[
T_{\min}(\alpha,\mathfrak h)^n\leq \frac{1}{\langle e_0^{n},[M_0]\rangle}\mathrm{Volume}(\alpha)\leq T_{\max}(\alpha,\mathfrak h)^n
\]
with equality cases exactly when the contact form $\alpha$ is Zoll.

In the {\it symplectic} case, $\mathfrak p_{\Omega_*}$ is trivial so that $\Sigma=M_*\times S^1$ and we call $t\in S^1$ the global angular coordinate. We take $\Omega=\Omega_*+\di(H\di t)$ for some function $H:M_*\times S^1\to \R$. The characteristic distribution of $\Omega$ is generated by the vector field $\partial_t+X_{H_t}$, where $X_{H_t}$ is the $\omega_*$-Hamiltonian vector field of $H_t:=H(\cdot,t)$ tangent to $M_*\times\{t\}$. Then, $\mathcal A_\Omega$ is the Hamiltonian action defined in \eqref{e:hamact}, $\langle c_0^n,[M_0]\rangle=\langle [\om_*^n],[M_*]\rangle>0$, and
\[
P(A)=\langle c_0^n,[M_0]\rangle A,\qquad \Vol(H\di t)=\int_{M_*\times S^1} (H\di t)\wedge\omega_*^{n}=\CAL_{\om_*}(H).
\] 
Conjecture \ref{con:lsios} generalises Proposition \ref{p:hamsys}, since $\mathcal A_{\min}(\Omega)=\min\mathcal A_H$ and $\mathcal A_{\max}(\Omega)=\max\mathcal A_H$, if we take the Hamiltonian $H$ to be normalised, namely $\CAL_{\om_*}(H)=0$.
\end{rmk}
Inspired by the Hamiltonian case, we will look first at a special class of $\Omega\in\Xi^2_C(\Sigma)$ when studying Conjecture \ref{con:lsios}. Here, we assume without loss of generality that $\mathfrak p_0=\mathfrak p_{\Omega_*}$. We fix an $S^1$-connection form $\eta$ for the bundle $\mathfrak p_0$. This means that $\eta$ restricts to the angular form on each fibre and $\di\eta=\mathfrak p_0^*\kappa$ for some $\kappa\in\Xi^2_{e_0}(M_0)$. We consequently choose the form $\om_0$ so that $\om_*=\om_0+A_*\kappa$, where $A_*=\mathcal A(\Omega_*)$. We say that a form $\Omega\in\Xi^2_C(\Sigma)$ is an \textbf{H-form}, if there exists a function $H:\Sigma\to\R$ such that 
\[
\Omega=\Omega_0+\di(H\eta).
\] 
We call $H$ a defining Hamiltonian for $\Omega$. If $H$ is $C^0$-close to $A_*$, then $\Omega$ is odd-symplectic, and if $H$ is $C^k$-close to $A_*$, then $\Omega$ is $C^{k-1}$-close to $\Omega_*$. An H-form is called \textbf{quasi-autonomous} if there exist $q_{\min},q_{\max}\in M_0$ and defining Hamiltonians $H_{\min}$ and $H_{\max}$ such that
\[
\min_{\Sigma} H_{\min}=H_{\min}(z),\quad\forall\,z\in\mathfrak p^{-1}_0(q_{\min}),\qquad \max_{\Sigma} H_{\max}=H_{\max}(z),\quad\forall\,z\in\mathfrak p^{-1}_0(q_{\max}).
\]
We say that $H$ is quasi-autonomous, if the corresponding $\Omega$ is quasi-autonomous. We establish Conjecture \ref{con:lsios} for quasi-autonomous forms close to a Zoll odd-symplectic form.
\begin{prp}\label{p:qaos}
There exists a $C^2$-neighbourhood $\mathcal H$ of the constant $A_*$ in the space of quasi-autonomous functions over $\Sigma$ such that if $H\in\mathcal H$ and $\Omega=\Omega_*+\di(H\eta)$ then
\[
P\big(\mathcal A_{\min}(\Omega)\big)\leq\Fvol(\Omega)\leq P\big(\mathcal A_{\max}(\Omega)\big),
\]
and any of the two equalities holds if and only if $\Omega$ is Zoll. In the Zoll case, $H$ is invariant under the holonomy of $\eta$. In particular, it is constant if $e_0\neq0$. 
\end{prp}
Let $\mathfrak p_0:M_0\times S^1\to M_0$ be the trivial bundle and take $\eta=\di t$, where $t \in S^1$ is the angular coordinate. In this case, we show in Proposition \ref{p:Moser} that if $\Omega\in\Xi^2_C(\Sigma)$ is $C^2$-close to $\Omega_*$, then $\Omega$ is an H-form with a $C^2$-small defining Hamiltonian $H$, after applying a diffeomorphism of $\Sigma$ isotopic to the identity. Therefore, in this setting the conjecture follows from Proposition \ref{p:hamsys} (or Proposition \ref{p:qaos}) and the invariance of the volume and the action under diffeomorphisms. From this result and the topological Lemma \ref{l:3state1}, Conjecture \ref{con:lsios} for bundles with $e_0=0$ can readily be proven.
\begin{thm}\label{t:triv}
The local systolic-diastolic inequality for odd-symplectic forms, whose associated bundle has vanishing real Euler class, holds true in the $C^2$-topology, .
\end{thm}
If $\Omega_*$ is any Zoll odd-symplectic form on a closed three-manifold $\Sigma$, then either the Euler class of its associated bundle vanishes or $\Omega_*=\di\alpha_*$ for some Zoll contact form $\alpha_*$. Hence, Theorem \ref{t:triv} together with \cite[Theorem 1.4]{BK19a} establish Conjecture \ref{con:lsios} in dimension three.
\begin{cor}\label{c:mainsys}
The local systolic-diastolic inequality for odd-symplectic forms holds true in the $C^2$-topology on closed three-manifolds.\qed
\end{cor}
\begin{rmk} 
More generally, one could try to formulate a systolic-diastolic inequality in a neighbourhood of an odd-symplectic form $\Omega_*$, whose closed characteristics are tangent to a (not necessarily free) $S^1$-action on $\Sigma$, cf.~\cite{Tho76} and \cite{BR94}.
\end{rmk}
\subsubsection*{The example of magnetic geodesics}

Our main motivation to study the systolic-diastolic inequality for odd-symplectic forms comes from twisted cotangent bundles, where the magnetic form is symplectic.
	Let $(N,\sigma)$ be a connected closed symplectic manifold of dimension $2m$. Let $\mathfrak p_{\ta^*N}:\ta^*N\to N$ be the cotangent bundle map and consider the twisted symplectic form
	\[
	\di\lambda_\mathrm{can}+\mathfrak p_{\ta^*N}^*\sigma\in\Omega^2(\ta^*N),
	\]
	where $\lambda_\mathrm{can}$ is the canonical one-form on $\ta^*N$. We fix a $\sigma$-compatible almost complex structure $J$ on $N$ with associated metric $g_J$. The structure $J$ turns $\mathfrak p_{\ta^*N}$ into a complex vector bundle and we denote by $\mathfrak p_J:\mathbb P_\C(\ta^*N)\to N$ its projectivisation. For all Riemannian metrics $g$ on $N$ in the same conformal class of $g_J$, let $S_g^*N$ be the unit co-sphere bundle of $g$, and write $\mathfrak p_g:S^*_gN\to N$ and $\mathfrak i_g:S_g^*N\hookrightarrow\ta^*N$ for the natural projection and inclusion. The one-parameter group $t\mapsto e^{tJ}$ acts fibrewise on $\ta^*N$ and yields a free $S^1$-action on $S^*_gN$, since $J$ is $g$-orthogonal. The quotient is naturally identified with $\mathbb P_\C(\ta^*N)$, so that we have an oriented $S^1$-bundle $\mathfrak p$ making the following diagram commute
	\[
	\xymatrix{
		\ta^*N\ar[dr]_{\mathfrak p_{\ta^*N}}&S^*_gN\ar[r]^-{\mathfrak p}\ar[d]^{\mathfrak p_g}\ar[l]_{\ \, \mathfrak i_g}&\mathbb P_\C(\ta^*N)\ar[dl]^{\mathfrak p_J}.\\
		&N&}
	\]
	Therefore, $\Omega_{g,\sigma}:=\mathfrak i_g^*(\di\lambda_\mathrm{can}+\mathfrak p_{\ta^*N}^*\sigma)$ is an odd-symplectic two-form on $S^*_gN$ with cohomology class $C:=\mathfrak p_g^*[\sigma]$. The class $C$ vanishes if and only if $N$ is a surface different from the two-torus. In general, it is well-known that there exists a Zoll form $\mathfrak p^*\om_\sigma$ in the class $C$. More specifically, $\om_\sigma$ is the symplectic form on $\mathbb P_\C(\ta^*N)$ defined as
	\[
	\om_\sigma:=a\,\om_\mathrm{FS}+ \mathfrak p_J^*\sigma,
	\]
	for some $a>0$ small enough (see \cite[Proposition 3.18]{Voi} when $(N,\sigma,J)$ is K\"ahler). Here, $\mathfrak p^*\om_\mathrm{FS}=\di\eta$, where $\eta$ is a connection form for $\mathfrak p$, which, for all $x\in N$, restricts on the fibre $\mathfrak p_g^{-1}(x)\cong S^{2m-1}$ to the standard contact form on the sphere. The $\mathfrak p$-fibres are almost tangent to the characteristic distribution of $\Omega_{g,\sigma}$, if $\sigma$ is very big. However, the form $\Omega_{g,\sigma}$ and $\mathfrak p^*\om_\sigma$ are remarkably not close to each other, if $m>1$.
	
	The relevance of this example stems from the fact that the characteristics of $\Omega_{g,\sigma}$ are the tangent lifts of the magnetic geodesics of the pair $(g,\sigma)$. A curve $c:\R\to N$ is called a magnetic geodesic if $g(\dot c,\dot c)\equiv 1$ and there holds
	\[
	\nabla^g_{\dot c}\dot c=-fJ\dot c,
	\]
	where $\nabla^g$ is the Levi-Civita connection of $g$ and $f:N\to(0,\infty)$ is the conformal factor given by $f\cdot g=g_J$. We refer to the work of Kerman in \cite{Ker99} for results on the existence of closed magnetic geodesics in this setting. 
	
In the companion paper \cite{BK19b}, we use Corollary \ref{c:mainsys} and Proposition \ref{p:cla} to establish a systolic-diastolic inequality for magnetic geodesics when $N$ is a surface, namely $m=1$.
\subsection{Structure of the paper}
\begin{itemize}
\item In Section \ref{c:vol}, we introduce the volume of a closed two-form on an odd-dimensional oriented closed manifold and explore its invariance properties.
\item We define odd-symplectic forms and H-forms in Section \ref{c:odd}. Under certain conditions, we prove a stability result for H-forms in the set of odd-symplectic forms.

\item We devote Section \ref{sec:oriented} to review some basic facts about oriented $S^1$-bundles, which will be crucially used in the following discussion.

\item In Section \ref{sec:weak_Zoll}, we define weakly Zoll pairs and Zoll odd-symplectic forms. We also prove Proposition \ref{p:cla}, which classifies Zoll odd-symplectic forms on oriented three-manifolds.

\item In Section \ref{s:acfun}, we introduce the action of a closed two-form on an odd-dimensional manifold which is the total space of some oriented $S^1$-bundle. We prove Theorem \ref{thm:mainaction}, which is just a reformulation of Theorem \ref{t:oszoll} showing that the action and the volume of a weakly Zoll pair are related through the Zoll polynomial.

\item We formulate the local systolic-diastolic inequality for odd-symplectic forms in Section \ref{sec:conjsys}. We prove it for quasi-autonomous H-forms (Proposition \ref{p:qa1} corresponding to Proposition \ref{p:qaos}) or when the real Euler class of the bundle vanishes (end of Subsection \ref{ss:triv} corresponding to Theorem \ref{t:triv} and Proposition \ref{p:hamsys}).
\end{itemize}

\bigskip

\noindent\textbf{Acknowledgements.} This work is part of a project in the Collaborative Research Center \textit{TRR 191 - Symplectic Structures in Geometry, Algebra and Dynamics} funded by the DFG. It was initiated when the authors worked together at the University of M\"unster and partially carried out while J.K.~was affiliated with the Ruhr-University Bochum. We thank Peter Albers, Kai Zehmisch, and the University of M\"unster for having provided an inspiring academic environment. We are also grateful to Michael Kapovich for the characterisation contained in Lemma \ref{l:3state1}. G.B.~would like to express his gratitude to Hans-Bert Rademacher and the whole Differential Geometry group at the University of Leipzig. J.K.~was supported by Samsung Science and Technology Foundation under Project Number SSTF-BA1801-01.

\subsection*{Notations}
 In the following discussion $(\Sigma,\mathfrak o_\Sigma)$ is a connected oriented closed manifold of dimension $2n+1$. We endow $\Sigma$  with an auxiliary Riemannian metric in order to measure distances and norms of various objects defined on $\Sigma$. We denote by $r$ a real number in $[0,1]$. We will often consider paths $r\mapsto\mathbf o_r$ with values in some target space $\mathbf O$. We use the short-hand $\{\mathbf o_r\}\subset\mathbf O$ for such a path and $[\mathbf o_r]$ for a homotopy class of paths with conveniently chosen boundary conditions. If $\{\mathbf o_r\}$ and $\{\mathbf o'_r\}$ are two paths with $\mathbf o_1=\mathbf o'_0$, we write $\{\mathbf o_r\}\#\{\mathbf o_r'\}$ for the concatenated path. Finally, if a mathematical object is expressed by a certain symbol, we add a \emph{prime} to it in order to denote another object of the same kind.

\section{The volume of a closed two-form}\label{c:vol}

\subsection{Definition of the volume}
Let $\Xi^k(\Sigma)$ be  the set of closed $k$-forms on $\Sigma$ and $\Xi_C^2(\Sigma)$ the set of elements in $\Xi^2(\Sigma)$ representing the de Rham cohomology class $C\in H_\dR^2(\Sigma)$. The set $\Xi^2_C(\Sigma)$ is affine with underlying vector space
\begin{equation*}\label{def:baromega}
\bar{\Omega}^1(\Sigma):=\frac{\Omega^1(\Sigma)}{\Xi^1(\Sigma)}.
\end{equation*}
Indeed, for every $\Omega\in\Xi^2_C(\Sigma)$, we have the surjective map
\[\label{eq:B_Omega}
B_{\Omega}:\Omega^1(\Sigma)\longrightarrow\Xi^2_C(\Sigma),\qquad \alpha\mapsto\Omega+\di\alpha.
\]
This map passes to the quotient under $\Xi^1(\Sigma)$ and yields a bijection 
\begin{equation*}
\bar{B}_{\Omega}:\bar{\Omega}^1(\Sigma)\longrightarrow\Xi^2_C(\Sigma)
\end{equation*}
so that we have an identification
\[
\mathrm T_\Omega\Xi^2_C(\Sigma)=\bar{\Omega}^1(\Sigma).
\]
We define now an exact one-form $\vol\in\Omega^1(\Omega^1(\Sigma))$. To this purpose, we fix a {\bf reference two-form} $\Omega_0\in\Xi^2_C(\Sigma)$, and for every $\alpha\in\Omega^1(\Sigma)$, we use  the short-hand 
\[
\Omega_\alpha:=B_{\Omega_0}(\alpha)=\Omega_0+\di\alpha.
\]
For every $\beta\in \Omega^1(\Sigma)\cong\ta_\alpha\Omega^1(\Sigma)$, we set
\[\label{def:vol}
\vol_{\alpha}\cdot\beta:=\int_\Sigma\beta\wedge \Omega_\alpha^{n}.
\]
In the lemma below, we show that the one-form $\vol$ admits the primitive functional
\begin{equation}\label{e:defvol}
\Vol:\Omega^1(\Sigma) \to \R\,,\qquad \Vol(\alpha):= \int_0^1\left(\int_\Sigma\alpha\wedge\Omega_{r\alpha}^{n}\right)\di r.
\end{equation}
\begin{rmk}\label{r:cs}
Exchanging the order of integration, we can rewrite
\[
\Vol(\alpha)=\int_\Sigma \mathrm{CS}(\alpha),\qquad \mathrm{CS}(\alpha):=\int_0^1\alpha\wedge \Omega^n_{r\alpha}\,\di r.
\]
When $C=0$ and $\Omega_0=0$, the form $(n+1)\mathrm{CS}(\alpha)$ reduces to the Chern-Simons form for principal $S^1$-bundles with invariant polynomial $P(x_1,\ldots,x_{n+1})=x_1\cdot\ldots\cdot x_{n+1}$ \cite[equation (3.1)]{CS74}. Furthermore, $(n+1)\Vol(\alpha)$ corresponds to the cohomology class in \cite[Theorem 3.9, case $2l-1=n$]{CS74} (where $M$ therein is our $\Sigma$ and $l$ is our $n+1$).
\end{rmk}
\begin{lem}\label{l:volform}
The functional $\Vol$ associated with $\Omega_0$ is uniquely characterised by the properties
\begin{equation*}
\Vol(0)=0,\qquad \di \Vol=\vol.
\end{equation*}
The following formula holds (see also \cite[equation (3.5)]{CS74}): 
\begin{equation*}
\Vol(\alpha)\ =\ \sum_{j=0}^{n}\frac{1}{j+1}\binom{n}{j}\int_\Sigma\alpha\wedge (\di \alpha)^j\wedge\Omega_0^{n-j}.
\end{equation*}
\end{lem}
\begin{proof}
Differentiating under the integral sign we write $\di_\alpha\Vol(\beta)=\int_0^1V(r)\di r$ and compute
\begin{align*}
V(r)&=\int_\Sigma\Big(\beta\wedge\Omega_{r\alpha}^{n}+\alpha\wedge \Omega_{r\alpha}^{n-1}\wedge nr\di\beta\Big)\\
&=\int_\Sigma\Big(\beta\wedge\Omega_{r\alpha}^{n}+nr \Omega_{r\alpha}^{n-1}\wedge\big(\beta\wedge \di\alpha-\di(\alpha\wedge\beta)\big)\Big)\\
&=\int_\Sigma\Big(\beta\wedge\Omega_{r\alpha}^{n}+r \beta\wedge\Omega_{r\alpha}^{n-1}\wedge\big(n\di\alpha\big)\Big)\\
&=\frac{\di }{\di r}\left(\int_\Sigma r\beta\wedge\Omega_{r\alpha}^{n}\right),
\end{align*}
where in the third equality we used Stokes' Theorem. We conclude that
\begin{equation*}
\di_\alpha\Vol\cdot\beta=\int_0^1\frac{\di}{\di r}\left(\int_\Sigma r\beta\wedge\Omega_{r\alpha}^{n}\right)\di r=\int_\Sigma 1\cdot\beta\wedge\Omega_{1\cdot\alpha}^{n}-\int_\Sigma 0\cdot\beta\wedge\Omega_{0\cdot\alpha}^{n}=\vol_{\alpha}\cdot\beta.
\end{equation*}
The formula for $\Vol$ follows by expanding the binomial $\Omega_{r\alpha}^{n}= \big(\Omega_0+r\di \alpha\big)^{n}$ in \eqref{e:defvol} and integrating each term.
\end{proof}
\begin{rmk}\label{r:vol3}
If $\dim\Sigma=3$, namely $n=1$, the formula for $\Vol$ turns into
\begin{equation*}
\Vol(\alpha)= \int_\Sigma\alpha\wedge\big(\Omega_0+\tfrac{1}{2}\di\alpha\big).
\end{equation*}
\end{rmk}
Let us specify the dependence of $\Vol$ on the reference form $\Omega_0$.

\begin{lem}\label{l:volcomp}
Let $\Vol':\Omega^1(\Sigma)\to\R$ be the volume functional associated with another reference two-form $\Omega_0'\in\Xi^2_C(\Sigma)$. If $\alpha'\in\Omega^1(\Sigma)$ is such that $\Omega_0'=\Omega_0+\di\alpha'$, then
\[
\Vol(\alpha)=\Vol(\alpha')+\Vol'(\alpha-\alpha'),\qquad\forall\,\alpha\in\Omega^1(\Sigma).
\]
\end{lem}
\begin{proof}
For every $\alpha,\beta\in\Omega^1(\Sigma)$, we have
\[
\vol'_{\alpha-\alpha_1}\cdot\beta=\int_\Sigma\beta\wedge (\Omega_0'+\di\alpha-\di\alpha')^{n}=\int_\Sigma\beta\wedge (\Omega_0+\di\alpha)^{n}=\vol_{\alpha}\cdot\beta,
\]
where $\vol':=\di\Vol'$. Therefore, from the fundamental theorem of calculus, we get
\begin{align*}
\Vol(\alpha)-\Vol(\alpha')=\int_0^1\vol_{r(\alpha-\alpha')+\alpha'}\cdot(\alpha-\alpha')\di r&=\int_0^1\vol'_{r(\alpha-\alpha')}\cdot(\alpha-\alpha')\di r\\
&=\Vol'(\alpha-\alpha').\qedhere
\end{align*}
\end{proof}

We now wish to use the map $B_{\Omega_0}$ to push the volume form and functional to the space of closed two-forms in the class $C$. To this purpose, we observe that $\Vol$ behaves linearly under the addition of closed one-forms. More precisely, for all $\alpha\in\Omega^1(\Sigma)$ we have
\begin{equation}\label{e:wd2}
\begin{split}
\Vol(\alpha+\tau)-\Vol(\alpha)=\Vol(\tau)=\langle[\tau]\cup C^{n},[\Sigma]\rangle,\quad \forall\,\tau\in\Xi^1(\Sigma),
\end{split}
\end{equation}
and similarly,
\begin{equation}\label{e:wd3}
\begin{split}
\vol_\alpha\cdot\tau=\langle[\tau]\cup C^{n},[\Sigma]\rangle,\quad \forall\,\tau\in\Xi^1(\Sigma).
\end{split}
\end{equation}
These formulae suggest to treat two separate cases.
\medskip

\subsubsection*{Case 1:~$C^{n}=0$.}

The volume functional $\Vol:\Omega^1(\Sigma)\to\R$ passes to the quotient by $\Xi^1(\Sigma)$ according to \eqref{e:wd2}:
\begin{equation}\label{e:defvolcn0}
\Fvol:\Xi^2_C(\Sigma)\to\R,\qquad\Fvol(\Omega_0)=0.
\end{equation}
Following Lemma \ref{l:volcomp}, if $\Fvol'$ denotes the functional obtained from $\Omega_0'$, then there holds
\[
\Fvol(\Omega)=\Fvol(\Omega_0')+\Fvol'(\Omega),\qquad\forall\,\Omega\in\Xi^2_C(\Sigma).
\]
This means that $\vol\in\Omega^1(\Omega^1(\Sigma))$ descends to a one-form
\begin{equation*}
\fvol\in \Omega^1(\Xi^2_C(\Sigma)),\qquad \di\Fvol=\fvol,
\end{equation*}
which is \emph{independent} of the reference two-form $\Omega_0$.

\subsubsection*{Case 2:~$C^{n}\neq0$.}

The functional $\Vol$ does not descend to $\Xi^2_C(\Sigma)$. Indeed, by Poincar\'e duality, there exists a form $\tau\in\Xi^1(\Sigma)$ such that 
\begin{equation}\label{eq:Poincare_tau}
[\tau]\cup C^{n}\neq0.
\end{equation}
However, we can use the volume function to define a distinguished class of one-forms.
\begin{dfn}\label{d:normalized}
We say that $\alpha\in\Omega^1(\Sigma)$ is a \textbf{normalised one-form}, if the implication
\[
C^{n}\neq 0\quad\Longrightarrow\quad \Vol(\alpha)=0
\]
holds true (in particular, when $C^{n}=0$ as in Case 1, all one-forms are normalised).
\end{dfn}
The inclusion $\Vol^{-1}(0)\subset \Omega^1(\Sigma)$ induces a bijection
\begin{equation}\label{e:normal1}
\frac{\Vol^{-1}(0)}{\sim}\longrightarrow \bar{\Omega}^1(\Sigma)\stackrel{\bar B_{\Omega_0}}{\cong} \Xi^2_C(\Sigma),
\end{equation}
where $\alpha'\sim\alpha''$ if and only if $\alpha''-\alpha'\in\Xi^1(\Sigma)$. Indeed, by \eqref{e:wd2} any class $\bar\alpha:=\alpha+\Xi^1(\Sigma)\in\bar\Omega^1(\Sigma)$ has a normalised representative $\alpha'=\alpha+s\tau$, for some $s\in\R$. We also observe that, if $\alpha''$ is another normalised representative of $\bar\alpha$, then, again by \eqref{e:wd2}, we have
\begin{equation}\label{e:vanishingcup}
[\alpha''-\alpha']\cup C^{n}=0.
\end{equation}
In view of \eqref{e:normal1}, we define the volume form and functional on $\Xi^2_C(\Sigma)$ trivially by setting
\[\label{e:defvolcnneq0}
\Fvol:\Xi^2_C(\Sigma)\to\R,\quad \Fvol:=0,\qquad\qquad \fvol\in\Omega^1(\Xi^2_C(\Sigma)),\quad \fvol:=0.
\]
If $\alpha$ is a normalised one-form representing $\bar\alpha\in\bar\Omega^1(\Sigma)$, then the inclusion $\ker\vol_\alpha\subset \Omega^1(\Sigma)$ induces the surjection $\ker\vol_\alpha\to\bar\Omega^1(\Sigma)  \cong \ta_{\bar\alpha}\bar\Omega^1(\Sigma)$ due to the existence of $\tau$ satisfying \eqref{eq:Poincare_tau}, and in turn the isomorphism
\[
\frac{\ker\vol_\alpha}{\ker \Big(\Xi^1(\Sigma)\xrightarrow{[\,\cdot\,]\cup C^{n}}H^{2n+1}_\dR(\Sigma)\Big)}\longrightarrow \ta_{\bar\alpha}\bar{\Omega}^1(\Sigma)\cong \ta_{\Omega_\alpha}\Xi^2_C(\Sigma),
\]
thanks to \eqref{e:wd3} and \eqref{e:vanishingcup}.

\begin{rmk}\label{rmk:contact_hamiltonian}
For contact forms and Hamiltonian systems the volume functional recovers the following well-known formulae.
\begin{itemize}
\item(Contact forms) This is an instance of Case 1. Let $\Omega_0=0$ and $\alpha\in\Omega^{1}(\Sigma)$ be a contact form. If we endow $\Sigma$ with the orientation $\mathfrak o_\alpha$, we recover the contact volume up to a constant factor:
\[
\Fvol(\di\alpha)=\Vol(\alpha)=\frac{1}{n+1}\int_\Sigma\alpha\wedge (\di\alpha)^{n}=\frac{1}{n+1}\mathrm{Volume}(\alpha).
\]
\item(Hamiltonian systems) This is an instance of Case 2. Let $\Sigma=M\x S^1$ and $\Omega_0=\mathfrak p^*\om_0$ for some symplectic form $\om_0$ on $M$, where $\mathfrak p:M\x S^1\to M$ is the projection on the first factor. Let $\alpha=H\di\phi$, where $H:M\x S^1\to \R$ and $\phi$ is the coordinate on $S^1$. Then, the volume functional recovers the {\bf Calabi invariant} of $H$
\[
\Vol(H\di\phi)=\int_{M\times S^1} (H\di\phi)\wedge\om_0^{n}=\CAL_{\om_0}(H)
\]
and $\alpha$ is normalised if and only if the Calabi invariant of $H$ vanishes.
\end{itemize}
\end{rmk}

\subsection{The volume is invariant under pull-back and isotopies}\label{s:voliso}
In this subsection, we prove two invariance results for the volume. Recall that $\Omega_0$ is the fixed reference form in $\Xi_C^2(\Sigma)$ for $C\in H^2_\dR(\Sigma)$.
\begin{prp}\label{p:volinvcov}
Let $\Sigma$ and $\Sigma^\veee$ be two closed oriented manifolds of dimension $2n+1$ and $\Pi:\Sigma^\veee\to\Sigma$ a map of degree $\deg \Pi\in\Z$. Let $C^\veee:=\Pi^*C\in H^2_\dR(\Sigma^\veee)$ and $\Omega_0^\veee:=\Pi^*\Omega_0\in \Xi^2_{C^\veee}(\Sigma^\veee)$. If $\Vol,\, \Fvol$ are the volumes associated with $\Omega_0$ and $\Vol^\veee,\, \Fvol^\veee$ the volumes associated with $\Omega^\veee_0$, there holds
\[
\Vol^\veee\circ \Pi^*=(\deg \Pi)\cdot\Vol,\qquad \Fvol^\veee\circ \Pi^*=(\deg \Pi)\cdot\Fvol.
\]
\end{prp}
\begin{proof}
The statement follows immediately from the definition of the volume and the observation that for all top dimensional forms $\mu$ on $\Sigma$ there holds
\[
\int_{\Sigma^\veee}\Pi^*\mu=\deg \Pi\cdot\int_{\Sigma}\mu.\qedhere
\]
\end{proof}
For the second result, we need a little preparation. Let $\Psi:\Sigma\to\Sigma$ be a diffeomorphism isotopic to the identity $\id_\Sigma$. Since the pull-back $\Psi^*$ of $\Psi$ acts as the identity in cohomology, there is a normalised one-form $\theta\in\Omega^1(\Sigma)$ (determined up to a normalised closed one-form) such that
\begin{equation}\label{eq:alpha_F}
\di\theta=\Psi^*\Omega_0-\Omega_0.
\end{equation}
We define
\begin{equation*}\label{def:psihat}
\widehat \Psi^*_\theta:\Omega^1(\Sigma)\to\Omega^1(\Sigma),\qquad \widehat \Psi^*_\theta(\alpha):=\theta+\Psi^*\alpha,
\end{equation*}
so that the following diagram commutes:
\begin{equation}\label{eq:commutative_diagram1}
\xymatrix{ \Omega^1(\Sigma)\ar[r]^-{\widehat \Psi_{\theta}^*} \ar[d]_{B_{\Omega_0}}& \Omega^1(\Sigma)\ar[d]^{B_{\Omega_0}}\,.\\
 \Xi^2_C(\Sigma)\ar[r]^{\Psi^*} & \Xi^2_C(\Sigma)}
\end{equation}
Let $\{\Psi_r\}$ be an isotopy with $\Psi_0=\id_\Sigma$ and $\Psi_1=\Psi$. This gives rise to a smooth family of one-forms $\{\theta_r\}$ satisfying
\begin{equation*}
\theta_0=0,\qquad\di\theta_r=\Psi^*_r\Omega_0-\Omega_0,\quad \forall\, r\in[0,1].
\end{equation*}
To construct such a family, we just take the time-dependent vector field $X_r$ generating the isotopy $\{\Psi_r\}$ and let $\{\theta_r\}$ be the unique path such that
\begin{equation*}
\theta_0=0,\qquad\dot\theta_r=\Psi_r^*(\iota_{X_r}\Omega_0).
\end{equation*}
To ease the notation, if $\alpha\in\Omega^1(\Sigma)$, we use the short-hand
\begin{equation}\label{e:shorthand}
\wh \Psi^*_r(\alpha):=\wh \Psi_{{\theta_r}}^*(\alpha)=\theta_r+\Psi^*_r\alpha.
\end{equation}
Every $\theta_r$ is normalised, i.e.~$\Vol(\theta_r)=0$. Indeed, $\Vol(\theta_0)=\Vol(0)=0$ and
\[
\frac{\di}{\di r}\Vol(\theta_r)=\vol_{\theta_r}(\dot\theta_r)=\int_\Sigma \Psi_r^*(\iota_{X_r}\Omega_0)\wedge (\Psi^*_r\Omega_0)^{n}=\int_\Sigma (\iota_{X_r}\Omega_0)\wedge \Omega_0^{n}=0.
\]
In particular, the form $\tau_1:=\theta-\theta_1$ is closed and normalised and as observed in \eqref{e:vanishingcup}, there holds
\begin{equation*}
[\tau_1]\cup C^{n}=0
\end{equation*}
(note that this condition is automatically satisfied if $C^{n}=0$).
\begin{prp}\label{prp:vol}
If $\Psi$ is a diffeomorphism isotopic to $\id_\Sigma$ and $\theta$ is any normalised one-form satisfying \eqref{eq:alpha_F}, there holds
\[
\Vol\circ \widehat \Psi^*_\theta=\Vol.
\]
As a consequence, the set of normalised one-forms is $\wh \Psi_\theta^*$-invariant and we have
\[
\Fvol\circ \Psi^*=\Fvol.
\]
\end{prp}
\begin{proof}
For any $\alpha\in\Omega^1(\Sigma)$, we have
\[
\frac{\di}{\di r}\widehat \Psi^*_{r}(\alpha)=\Psi_r^*(\iota_{X_r}\Omega_0)+\Psi_r^*\Big(\iota_{X_r}\di\alpha+\di\big(\alpha(X_r)\big)\Big)=\Psi_r^*\Big(\iota_{X_r}\Omega_\alpha+\di\big(\alpha(X_r)\big)\Big).
\]
Using this relation, we compute
\begin{align*}
\frac{\di}{\di r}\Vol\big(\widehat \Psi^*_r(\alpha)\big)&=\vol_{\widehat \Psi^*_{r}(\alpha)}\cdot\frac{\di}{\di r}\widehat \Psi^*_r(\alpha)\\
&=\int_\Sigma \Psi_r^*\Big(\iota_{X_r}\Omega_\alpha+\di\big(\alpha(X_r)\big)\Big)\wedge (\Psi_r^*\Omega_\alpha)^{n}\\
&= \int_\Sigma\iota_{X_r}\Omega_\alpha\wedge\Omega_\alpha^{n}+\int_\Sigma \di\big(\alpha(X_r)\big)\wedge \Omega_\alpha^{n}\\
&=0.
\end{align*}
We conclude from \eqref{e:wd2} and the relation $[\tau_1]\cup C^{n}=0$ that
\[
\Vol\big(\widehat \Psi^*_\theta(\alpha)\big)=\Vol\big(\widehat \Psi^*_1(\alpha)+\tau_1\big)=\Vol\big(\widehat \Psi^*_1(\alpha)\big)+\langle[\tau_1]\cup C^{n},[\Sigma]\rangle=\Vol\big(\widehat \Psi^*_0(\alpha)\big)=\Vol(\alpha).
\]
This proves the first identity. The second identity follows from the first one and the commutation relation $B_{\Omega_0}\circ\widehat \Psi_{\theta}^*=\Psi^*\circ B_{\Omega_0}$ in \eqref{eq:commutative_diagram1}.
\end{proof}

\section{Odd-symplectic forms}\label{c:odd}
\subsection{A couple of definitions}\label{s:odddef}
For a closed two-form $\Omega\in\Xi^2(\Sigma)$ on $\Sigma$, we consider the (possibly singular) distribution $\ker\Omega\to\Sigma$ defined by
\[
\ker\Omega:=\big\{(z,u)\in\ta\Sigma\ \big|\ \Omega_z(u,v)=0,\;\forall\, v\in \ta_z\Sigma\big\}.
\]
The distribution $\ker\Omega$ is naturally co-oriented by $\Omega^n$ and we orient it combining the orientation on $\Sigma$ with such a co-orientation.
\begin{dfn}\label{dfn:odd_symplectic}
We say that $\Omega\in\Xi^2(\Sigma)$ is {\bf odd-symplectic} if $\ker\Omega\to\Sigma$ is a one-dimensional distribution and denote the set of odd-symplectic forms by $\mathcal S(\Sigma)\subset\Xi^2(\Sigma)$. If $C\in H^2_\dR(\Sigma)$, we write $\mathcal S_C(\Sigma):=\mathcal S(\Sigma)\cap\Xi^2_C(\Sigma)$.
\end{dfn}
\begin{rmk}\label{rmk:equivalent_def_odd_symplectic}
Any of the following conditions gives an equivalent definition of $\Omega\in\Xi^2(\Sigma)$ being odd-symplectic:
\begin{itemize}
\item $\Omega^{n}$ is nowhere vanishing.
\item There exists a one-form $\sigma$ on $\Sigma$ such that $\sigma\wedge\Omega^{n}$ is a volume form.
\item  There exists a nowhere vanishing one-form $\sigma$ such that $\Omega|_{\ker\sigma}$ is a non-degenerate bilinear form on $\ker\sigma$.
\end{itemize} 
\end{rmk}
\begin{dfn}\label{dfn:H_forms}
Given a reference form $\Omega_0\in\Xi^2_C(\Sigma)$ and a pair of one-forms $\alpha_0,\sigma_0\in\Omega^1(\Sigma)$, we can build an affine map
\begin{align*}
C^{\infty}(\Sigma)&\longrightarrow \Xi^2_C(\Sigma),\\ 
H&\longmapsto \Omega_{\alpha_0+H\sigma_0}=\Omega_0+\di(\alpha_0+H\sigma_0).
\end{align*}
We refer to the image of this map, as the set of \textbf{H-forms} (with respect to $\Omega_0$, $\alpha_0$, and $\sigma_0$).
\end{dfn}
\subsection{A stability result}

As we will see in Section \ref{sec:conjsys}, given $\Omega_0$, $\alpha_0$ and $\sigma_0$ as in Definition \eqref{dfn:H_forms}, it is extremely helpful to see, if an odd-symplectic form $\Omega_1\in\mathcal S_C(\Sigma)$ admits a diffeomorphism $\Psi:\Sigma\to\Sigma$ isotopic to $\id_\Sigma$ such that $\Psi^*\Omega_1$ is an H-form. We provide a criterion for such a diffeomorphism to exist.

\begin{prp}\label{p:Moser}
Let $\{\Omega_r=\Omega_0+\di\alpha_r\}$ be a path in $\mathcal S_C(\Sigma)$. Let $\{\sigma_r\}$ be a path in $\Omega^{1}(\Sigma)$ such that $\sigma_r\wedge \Omega_r^n$ is a volume form on $\Sigma$, which exists according to Remark \ref{rmk:equivalent_def_odd_symplectic}. Let us denote by $I_r:(\ker\sigma_r)^*\to\ker\sigma_r$ the inverse of the map $v\mapsto\iota_v\Omega_{\alpha_r}|_{\ker\sigma_r}$. If
\begin{equation}\label{e:dotsigmar}
\big(\dot\sigma_r-\iota_{I_r(\dot\alpha_r)}\di\sigma_r\big)\big|_{\ker\sigma_r}=0,
\end{equation}
then there exist an isotopy $\{\Psi_r:\Sigma\to\Sigma\}$ with $\Psi_0=\id_\Sigma$ generated by a time-dependent vector field  $X_r\in\ker\sigma_r$ and a path of functions $\{H_r:\Sigma\to\R\}$ such that, for every $r\in[0,1]$,
\begin{equation*}
\wh{\Psi}_r^*\alpha_r=\alpha_0+H_r\sigma_0+\di K_r,\qquad K_r:=\int_0^r\alpha_{r'}(X_{r'})\circ \Psi_{r'}\di r'.
\end{equation*}
Here $\wh{\Psi}_r^*$ is the map in \eqref{e:shorthand}. In particular, we have
\begin{equation*}
\Psi^*_1\Omega_1=\Omega_0+\di(\alpha_0+H_1\sigma_0),\qquad \Vol(\alpha_1)=\Vol(\alpha_0+H_1\sigma_0).
\end{equation*}
\end{prp}
\begin{proof}
The argument is a refinement of the Gray stability theorem, see e.g.~\cite[Theorem 2.2.2]{Gei08}. We introduce a vector field $V_r$ on $\Sigma$ uniquely determined by the relations
\[
\iota_{V_r}\Omega_{\alpha_r}=0,\qquad \sigma_r(V_r)=1.
\]
We define a path of vector fields $r\mapsto X_r$ on $\Sigma$ satisfying the relations
\begin{equation}\label{e:stab1}
X_r\in\ker\sigma_r,\qquad \big(\iota_{X_r}\Omega_{\alpha_r}+\dot\alpha_r\big)\big|_{\ker\sigma_r}=0,
\end{equation}
which exists and is unique by (i). Applying (ii), we also have
\begin{equation}\label{e:stabalpha}
\big(\iota_{X_r}\di\sigma_r+\dot\sigma_r\big)\big|_{\ker\sigma_r}=0.
\end{equation}
Let $r\mapsto \Psi_r$ be the isotopy on $\Sigma$ obtained by integrating $X_r$ and setting $\Psi_0=\id_\Sigma$. We construct an auxiliary path of functions $r\mapsto J_r$ through
\begin{equation*}
J_0=0,\qquad \dot J_r=\big(\di\sigma_r(X_r,V_r)+\dot\sigma_r(V_r)\big)\circ \Psi_r.
\end{equation*}
Combining the last equation with \eqref{e:stabalpha}, we arrive at
\[
\iota_{X_r}\di\sigma_r+\dot\sigma_r=(\dot J_r\circ \Psi_r^{-1})\sigma_r.
\]
With this piece of information, we can compute
\[
\frac{\di}{\di r}\big(\Psi_r^*\sigma_r\big)=\Psi_r^*\big(\iota_{X_r}\di\sigma_r+\dot\sigma_r\big)=\Psi_r^*\big(\dot J_r\circ \Psi_r^{-1}\sigma_r\big)=\dot J_r \big(\Psi_r^*\sigma_r\big).
\]
Since $r\mapsto e^{J_r}\sigma_0$ also satisfies the same ordinary differential equation and coincides with $\Psi_r^*\sigma_r$ at $r=0$, we see that
\begin{equation}\label{e:fralpha}
\Psi_r^*\sigma_r=e^{J_r}\sigma_0.
\end{equation}

Next we define the path of functions $\{H_r:\Sigma\to\R\}$ by
\begin{equation}\label{e:stab2}
H_0= 0,\qquad \dot H_r=\big(\dot\alpha_r(V_r)\circ \Psi_r\big) e^{J_r}.
\end{equation}
Conditions \eqref{e:stab1} and \eqref{e:stab2} together can be rephrased as
\begin{equation}\label{e:stab3}
\iota_{X_r}\Omega_{\alpha_r}+\dot\alpha_r=\big((\dot H_r e^{-J_r})\circ \Psi_r^{-1}\big)\sigma_r.
\end{equation}

We are ready to prove the formula for $\wh{\Psi}^*_r\alpha_r$ in the statement of the proposition. The identity clearly holds for $r=0$ and we just need to show that the $r$-derivatives of both sides coincide for every $r$:
\begin{align*}
\frac{\di}{\di r}\big(\wh{\Psi}_r^*\alpha_r\big)=\Psi_r^*\Big(\di\big(\alpha_r(X_r)\big)\iota_{X_r}\di\alpha_r+ \dot\alpha_r\Big)+\dot\theta_r&=\di\big(\alpha_r(X_r)\circ \Psi_r\big)+\Psi_r^*\big(\iota_{X_r}\Omega_{\alpha_r}+\dot\alpha_r\big)\\
&=\di\dot K_r+\dot H_re^{-J_r} \Psi_r^*\sigma_r\\
&=\di\dot K_r+\dot H_r\sigma_0,
\end{align*}
where we used that $\dot\theta_r=\Psi_r^*\iota_{X_r}\Omega_0$ and equations \eqref{e:fralpha} and  \eqref{e:stab3}. Finally, Proposition \ref{prp:vol} yields the identity $\Vol(\alpha_1)=\Vol(\alpha_0+H_1\sigma_0)$.
\end{proof}
\begin{rmk}\label{r:stabil}
The form $\sigma_r\wedge \Omega_r^n$ is a volume form and condition \eqref{e:dotsigmar} is satisfied in the following two extreme cases:
\begin{enumerate}[(a)]
\item The form $\Omega_0$ vanishes and $\alpha_r=T_r\sigma_r$ is a contact form for every $r$, where $\{T_r\}$ is some path of real numbers. In this case, $K_r$ and $\theta_r$ vanish so that we have the usual Gray stability theorem
\[
\Psi^*_r\alpha_r=(1+\tfrac{1}{T_0}H_r)\alpha_0.
\]
\item The form $\sigma_0$ is closed, there holds $\sigma_r=\sigma_0$ for all $r$, and $\Omega_r$ non-degenerate on $\ker\sigma_0$.   
\end{enumerate}
\end{rmk}
\begin{rmk}
It would be interesting to find a condition not involving $r$-derivatives implying \eqref{e:dotsigmar}. Such a condition can indeed be found in cases (a) or (b) in the remark above, where the odd-symplectic forms $\Omega_r$ are actually \textit{stable} according to \cite[Section 2.1]{CM05}.
\end{rmk}

\section{Oriented $S^1$-bundles}\label{sec:oriented}

\subsection{$S^1$-bundles and free $S^1$-actions}\label{ss:deffirst}
Let \label{def:mathfrakPsigma}$\mathfrak P(\Sigma)$ be the space of all oriented $S^1$-bundles having $\Sigma$ as total space, up to equivalence. Here we say that two bundles are {\bf equivalent} if they have the same oriented fibres. If $\mathfrak p:\Sigma\to M$ is an element of $\mathfrak P(\Sigma)$, then $M$ is a closed manifold of dimension $2n$. We write \label{def:xi2M}$\Xi^2(M)$ for the space of closed two-forms on $M$ and $\Xi^2_c(M)$ for the set of those $\om\in\Xi^2(M)$ with $[\om]=c$. We orient $M$ combining the orientation on $\Sigma$ with the orientation of the $\mathfrak p$-fibres. We denote by $-\mathfrak p$ the bundle obtained from $\mathfrak p$ by reversing the orientation.
\begin{dfn}
\label{def:p-1}If $\mathfrak p$ belongs to $\mathfrak P(\Sigma)$, then $\mathfrak{p}^{-1}(\pt)\subset\Sigma$ denotes an oriented $\mathfrak p$-fibre and $[\mathfrak{p}^{-1}(\pt)]_\Z\in H_1(\Sigma;\Z)$ its integral homology class.
\end{dfn}
\begin{dfn}\label{def:bundle}
A \textbf{bundle map} between oriented $S^1$-bundles $\mathfrak p:\Sigma\to M$ and $\mathfrak p^\veee:\Sigma^\veee\to M^\veee$ is a map $\Pi:\Sigma^\veee\to\Sigma$ diffeomorphically sending every oriented fibre of $\mathfrak p^\veee$ to an oriented fibre of $\mathfrak p$. In this case, we write \label{def:pullback}$\mathfrak p^\veee=\Pi^*\mathfrak p$. A bundle map yields a quotient map $\pi:M^\veee\to M$ between the base manifolds so that $\mathfrak p\circ\Pi=\pi\circ \mathfrak p^\veee$. If $\Pi$ is also a diffeomorphism, we say that $\Pi$ is a \textbf{bundle isomorphism}. When $\Sigma^\veee=\Sigma$, we write $\mathfrak p^\veee=\mathfrak p'$, where $\mathfrak p':\Sigma\to M'$, and denote by $\Psi:\Sigma\to\Sigma$ a bundle isomorphism between $\mathfrak p$ and $\mathfrak p'$ with quotient map $\psi:M'\to M$. We summarise the properties of $\Pi$ and $\Psi$ in two commutative diagrams.
\begin{equation}\label{eq:commutative_diagram2}
\xymatrix{
\Sigma^\veee\ar[d]_{\mathfrak p^\veee}\ar[r]^ \Pi&\Sigma\ar[d]^{\mathfrak p}\,,\\ M^\veee\ar[r]^{\pi}& M}\qquad\qquad \xymatrix{
\Sigma\ar[d]_{\mathfrak p'}\ar[r]^ \Psi_{\sim}&\Sigma\ar[d]^{\mathfrak p}\,.\\ M'\ar[r]^{\psi}_\sim& M}
\end{equation}
\end{dfn}
Let $\mathfrak U(\Sigma)$ be the space of free $S^1$-actions on $\Sigma$. For $\mathfrak{u}\in \mathfrak{U}(\Sigma)$, we denote by \label{def:V}$V=V_\mathfrak{u}$ the fundamental vector field on $\Sigma$ generated by $\mathfrak{u}$. We have a natural map 
\begin{equation}\label{eq:S1-action_bundle}
\mathfrak{U}(\Sigma)\to\mathfrak P(\Sigma),
\end{equation}
which associates to a free $S^1$-action the quotient map onto its orbit space. It is a classical fact that this map is surjective. Indeed, let us fix $\mathfrak p\in\mathfrak P(\Sigma)$. Using a partition of unity, we find a vector field $X$ on $\Sigma$, positively tangent to the $\mathfrak p$-fibres at every point. If $T(z)>0$ is the period of the periodic orbit of $X$ starting at $z\in\Sigma$, then the rescaled vector field $V:=TX$ yields the desired $S^1$-action. The set $\mathfrak{U}(\Sigma)$ carries a natural $C^1$-topology as a closed subset of the space $\mathrm{Maps}(S^1\times\Sigma,\Sigma)$ and we endow $\mathfrak P(\Sigma)$  with the quotient topology brought by the map \eqref{eq:S1-action_bundle}. We write $\mathfrak{U}(\mathfrak p)$ for the fibre above $\mathfrak p\in\mathfrak P(\Sigma)$. This is a convex space in the following sense. If $V_0$ and $V_1=T_1V_0$ are the fundamental vector fields of $\mathfrak u_0,\mathfrak u_1\in\mathfrak U(\mathfrak p)$, then
\[
V_r:=\frac{V_1}{r+(1-r)T_1},\qquad \forall r\in[0,1]
\]
is the fundamental vector field of some $\mathfrak u_r\in \mathfrak U(\mathfrak p)$.

Finally, if $\Lambda(\Sigma)$ denotes the space one-periodic curves in $\Sigma$, we have a map
\begin{equation}\label{e:jmathpsi}
\jmath_\mathfrak{p}:\Sigma\to\Lambda(\Sigma)
\end{equation}
associating to a point the $\mathfrak{u}$-orbit through it for some $\mathfrak u\in\mathfrak U(\mathfrak p)$. Up to an orientation-preserving change of parametrisation of the elements of $\Lambda(\Sigma)$, this map depends only on $\mathfrak p$.

\begin{dfn}
If $\mathfrak p:\Sigma\to M$ belongs to $\mathfrak P(\Sigma)$ and $\mathfrak u$ belongs to $\mathfrak U(\mathfrak p)$, we write
\[
e_\Z\in H^2(M;\Z)
\] for \textit{minus} the Euler class of $\mathfrak u$, as defined, for example in \cite{Ch77}. This class is independent of $\mathfrak u\in\mathfrak U(\mathfrak p)$, and therefore, we refer to it as \textbf{minus the Euler class} of $\mathfrak p$.
\end{dfn}
\begin{rmk}\label{r:psiee}
The bundle $\mathfrak p$ is trivial (namely admits a global section) if and only if $e_\Z=0$. If $\mathfrak p^\veee:\Sigma^\veee\to M^\veee$ is another bundle with minus Euler class $e^\veee_\Z\in H^2(M^\veee;\Z)$, then
\begin{equation*}
\mathfrak p^\veee=\Pi^*\mathfrak p\quad\Longrightarrow\quad e^\veee_\Z=\pi^*e_\Z,
\end{equation*}
where $\Pi$ is a bundle map as in \eqref{eq:commutative_diagram2}. 
\end{rmk}
The inclusion map $\Z\into\R$ induces a map on the level of homology and cohomology and we write $e$ and $[\mathfrak{p}^{-1}(\pt)]$ for the images of $e_\Z$ and $[\mathfrak{p}^{-1}(\pt)]_\Z$, respectively:  
\begin{equation*}\label{def:p-1r}
\begin{aligned}
H^2(M;\Z)&\longrightarrow H^2(M;\R)\cong H^2_\dR(M),\qquad & H_1(\Sigma;\Z)&\longrightarrow H_1(\Sigma;\R).\\
 e_\Z &\longmapsto e & [\mathfrak{p}^{-1}(\pt)]_\Z&\longmapsto [\mathfrak{p}^{-1}(\pt)]
\end{aligned}
\end{equation*}
From the piece of the Gysin exact sequence $H^{0}_{\op{dR}}(M)\xrightarrow{\cup e}  H_{\op{dR}}^{2}(M)\xrightarrow{\mathfrak{p}^*} H^{2}_{\op{dR}}(\Sigma)$,
we know that
\begin{equation}\label{e:gysin}
\R\cdot e=\ker \Big(\mathfrak p^*:H^2_{\op{dR}}(M)\to H^2_{\op{dR}}(\Sigma)\Big).
\end{equation}
Moreover, by the universal coefficient theorem, we get
\[
e_\Z \text{ is torsion}\ \ \iff\ \ e=0,\qquad\qquad [\mathfrak{p}^{-1}(\pt)]_\Z \text{ is torsion}\ \ \iff\ \ [\mathfrak{p}^{-1}(\pt)]=0.
\]
The Chern-Weil theory yields a convenient description of the real cohomology class $e$. Let $\mathcal K(\mathfrak{u})$ denote the space of connection one-forms associated with $\mathfrak{u}\in\mathfrak U(\mathfrak p)$:
\[\label{eq:connection_forms}
\mathcal K(\mathfrak{u}):=\big\{\eta\in\Omega^1(\Sigma)\ \big|\ \eta(V)=1,\;\di\eta=\mathfrak p^*\kappa_\eta \text{ for some }\kappa_\eta\in\Omega^2(M)\big\}.
\]
This space is non-empty, and for every $\eta\in\mathcal K(\mathfrak u)$, the form $\kappa_\eta$ is closed and we have
\[
e=[\kappa_\eta]\in H^2_{\dR}(M).
\] 
Conversely, for any $\kappa\in\Xi^2(M)$ with $e=[\kappa]$, there is $\eta_\kappa\in\mathcal K(\mathfrak{u})$ with $\kappa_{\eta_\kappa}=\kappa$. We denote by
\[
\mathcal K(\mathfrak p):=\bigcup_{\mathfrak u\in\mathfrak U(\mathfrak p)}\mathcal K(\mathfrak u)
\]
the set of \textbf{connection one-forms} for $\mathfrak p$.

\subsection{Flat $S^1$-bundles and the local structure of $\mathfrak P(\Sigma)$}
This subsection is devoted to the proof of three lemmas. In the first two, we study flat bundles, i.e.\ with $e=0$. In the last one, we prove a theorem of Weinstein \cite{Wei74} showing that the space $\mathfrak P(\Sigma)$ is locally trivial.
\begin{lem}\label{l:ecc}
Let $\mathfrak p:\Sigma\to M$ be an oriented $S^1$-bundle. We have an equivalence
\begin{equation}\label{e:eppt}
e=0\qquad\Longleftrightarrow\qquad[\mathfrak{p}^{-1}(\pt)]\neq 0.
\end{equation}
If $c\in H^2_{\mathrm{dR}}(M)$ and we define $C:=\mathfrak p^*c\in H^2_{\mathrm{dR}}(\Sigma)$, there holds
\begin{equation}\label{e:pd}
\mathrm{PD}(C^{n})=\langle c^{n},[M]\rangle\cdot [\mathfrak{p}^{-1}(\pt)]\,\in\, H_1(\Sigma;\R).
\end{equation}
In particular, if $c^{n}\neq 0$, then
\begin{equation}\label{e:cne}
\begin{aligned}
(i)&\ \ \ker \Big( H^1_{\op{dR}}(\Sigma)\xrightarrow{\langle\,\cdot\,,[\mathfrak{p}^{-1}(\pt)]\rangle}\R\Big)=\ker \Big(H^1_{\op{dR}}(\Sigma)\xrightarrow{(\,\cdot\,)\cup C^{n}}H^{2n+1}_\dR(\Sigma)\Big),\\
(ii)&\ \ C^{n}\neq 0\quad \Longleftrightarrow\quad e=0.
\end{aligned}
\end{equation}
\end{lem}
\begin{proof}
The proof of \eqref{e:eppt} relies on the Gysin sequence and of its Poincar\'e dual
\[
\xymatrix{ H^{2n-2}_{\op{dR}}(M)\ar[r]^-{\cup e} \ar[d]_{\mathrm{PD}}& H_{\op{dR}}^{2n}(M)\ar[d]^{\mathrm{PD}} \ar[r]^{\mathfrak{p}^*}\ar[d]^{\mathrm{PD}} & H^{2n}_{\op{dR}}(\Sigma) \ar[r]^{\mathfrak{p}_*}\ar[d]^{\mathrm{PD}} & H^{2n-1}(M) \ar[d]^{\mathrm{PD}},\\
 H_2(M;\R)\ar[r]^{\cap e} & H_0(M;\R) \ar[r]^{\mathrm{PD}(\mathfrak{p}^*)} & H_1(\Sigma;\R)\ar[r]& H_1(M;\R)}
\]
where $\mathrm{PD}$ denotes Poincar\'e duality and $\mathfrak{p}_*$ stands for the integration along fibres.  We claim that $\mathrm{PD}(\mathfrak{p}^*)$ sends the class of a point $[\pt]$ to $[\mathfrak{p}^{-1}(\pt)]$. To this purpose, let $\mu\in\Xi^{2n}(M)$ be such that $\mathrm{PD}([\mu])=[\pt]$, so that $\int_M\mu=1$. Going around the central square in the diagram above, the claim follows, if we can show that $\mathrm{PD}(\mathfrak p^*[\mu])=[\mathfrak{p}^{-1}(\pt)]$. This last equality is true since for all $\tau\in \Xi^1(\Sigma)$ we have
\[
\langle [\tau],[\mathfrak{p}^{-1}(\pt)]\rangle=\mathfrak p_*[\tau]=\mathfrak p_*[\tau]\cdot\int_M\mu=\int_M\mathfrak p_*[\tau]\cdot\mu\stackrel{(\star)}{=}\int_\Sigma\tau\wedge \mathfrak p^*\mu=\langle[\tau]\cup\mathfrak p^*[\mu],[\Sigma]\rangle,
\]
where in $(\star)$ we have used Fubini's Theorem. Therefore, $[\mathfrak p^{-1}(\pt)]\neq0$ if and only if the map $\cap e$ is equal to zero. This happens if and only if $\langle e, H_2(M;\R)\rangle=0$ as one sees identifying $H_0(M;\R)$ with $\R$. Finally, from the universal coefficient theorem, $\langle e, H_2(M;\R)\rangle=0$ if and only if $e=0$.

Let us show \eqref{e:pd}. If $\om\in \Xi^2_c(M)$ and $\tau\in \Xi^1(\Sigma)$ is arbitrary, we have
\begin{equation*}
\langle c^{n},[M]\rangle\cdot\langle[\tau] ,[\mathfrak{p}^{-1}(\pt)]\rangle=\int_M\mathfrak{p}_*[\tau]\cdot \om^{n}=\int_\Sigma\tau\wedge \mathfrak p^*\om^{n}=\langle [\tau]\cup C^{n},[\Sigma]\rangle.
\end{equation*}
Let us assume that $\langle c^{n},[M]\rangle\neq0$. Looking again at the last equation, we readily see that (i) in \eqref{e:cne} holds. The equivalence in (ii) stems from a combination of \eqref{e:pd} and \eqref{e:eppt}
\[
C^{n}\neq0\quad\iff\quad \mathrm{PD}(C^{n})\neq0\quad\iff\quad [\mathfrak{p}^{-1}(\pt)]\neq0\quad\iff\quad e=0.\qedhere
\]
\end{proof}
We now show that a flat bundle can always be pulled back to a trivial one.
\begin{lem}\label{l:3state1}
Let $\mathfrak p:\Sigma\to M$ be an oriented $S^1$-bundle. The class $e\in H^2_\dR(M)$ vanishes if and only if there exists a finite cyclic cover $\pi:M^\veee\to M$ such that the pull-back bundle $\mathfrak p^\veee:\Sigma^\veee\to M^\veee$ of $\mathfrak p$ by $\pi$ is trivial. In this case, the order of $e_\Z$ in $H^2(M;\Z)$ equals the minimum degree of a cover with the properties above.
\end{lem}
\begin{proof}
We preliminarily observe that if $\pi:M^\veee\to M$ is a finite cover of degree $k$ and the bundle $\mathfrak p^\veee:\Sigma^\veee\to M^\veee$ is the pull-back of $\mathfrak p$ through $\pi$ with minus the Euler class $e_\Z^\veee$, then by Remark \ref{r:psiee} there holds
\[
\pi_*e_\Z^\veee=k\cdot e_\Z,
\]
where $\pi_*:H^2(M^\veee;\Z)\to H^2(M;\Z)$ is the transfer map (see \cite[Chapter 3.G]{Hat02}). Therefore, if we suppose that $\mathfrak p^\veee$ is trivial, we deduce that $e_\Z^\veee=0$, and hence, $k\cdot e_\Z=0$. In particular, the order of $e_\Z$ divides $k$.
 
Conversely, let $k$ be a positive integer such that $k\cdot e_\Z=0$ and take $\mathfrak{u}\in\mathfrak{U}(\mathfrak p)$. Let $\tfrac{\Z}{k\Z}\to S^1$ be the canonical homomorphism $j\to j/k$, and we denote by $\mathfrak{u}_k$ the free $\tfrac{\Z}{k\Z}$-action on $\Sigma$ obtained combining this homomorphism with $\mathfrak{u}$. Let $\Sigma_k=\Sigma/\mathfrak{u}_k$ be the quotient by this action and denote by $\Pi_k:\Sigma\to \Sigma_k$ the associated quotient map. The natural map $\mathfrak p_k:\Sigma_k\to M$ is an oriented $S^1$-bundle. Let $(e_\Z)_k\in H^2(M;\Z)$ denote minus the Euler class of $\mathfrak p_k$. Since $S^1/\tfrac{\Z}{k\Z}\cong S^1$ through the map $\phi\mapsto k\cdot \phi$, we see that $(e_\Z)_k=k\cdot e_\Z=0$. Therefore, we have a section $\mathfrak s_k:M\to \Sigma_k$ for $\mathfrak p_k$ and we take a connected component $M^\veee\subset \Sigma$ of $\Pi_k^{-1}(\mathfrak s_k(M))$. If $\mathfrak i:M^\veee\to\Sigma$ is the inclusion map, we define $\pi:M^\veee\to M$ as $\pi:=\mathfrak p\circ \mathfrak i$. The map $\pi$ is a covering map, whose deck transformation group is given by the sub-action of $\mathfrak{u}_k$ that leaves $M^\veee$ invariant. The deck transformation group is isomorphic to $\tfrac{\Z}{j\Z}$, where $j$ divides $k$, since it is a subgroup of the cyclic group $\tfrac{\Z}{j\Z}$. Let $\Pi:\Sigma^\veee\to\Sigma$ be a covering map lifting $\pi$. We sum up the construction above in the commutative diagram
\[
\xymatrix{
\Sigma^\veee\ar[d]_{\mathfrak p^\veee}\ar[r]^{\Pi}&\Sigma\ar[d]_{\mathfrak p}\ar[r]^{\Pi_k}&\Sigma_k\ar[dl]_{\mathfrak p_k}.\\
M^\veee\ar@{-->}@/^2pc/[u]^{\mathfrak s^\veee} \ar[ur]^{\mathfrak i}\ar[r]^{\pi}&M\ar@{-->}@/_1pc/[ur]_{\mathfrak s_k}
}
\]
We construct a section $\mathfrak s^\veee:M^\veee\to\Sigma^\veee$ tautologically, as follows. Let us consider an arbitrary $q^\veee\in M^\veee$. By definition of pull-back bundle, the $\mathfrak p^\veee$-fibre over $q^\veee$ is identified through $\Pi$ with the $\mathfrak p$-fibre over $\pi(q^\veee)$. Since $\mathfrak i(q^\veee)\in\Sigma$ belongs to the $\mathfrak p$-fibre over $\pi(q^\veee)$, we can define $\mathfrak s^\veee(q^\veee)$ through the equation
\[
\Pi(\mathfrak s^\veee(q^\veee))=\mathfrak i(q^\veee).
\]
This proves the existence of a cover as in the statement of the lemma and shows that the degree $k^\veee$ is less than or equal to the order of $e_\Z$. 
\end{proof}

If $\mathfrak p\in\mathfrak P(\Sigma)$ and $\{\Psi_r:\Sigma\to\Sigma\}$ is an isotopy with $\Psi_0=\id_\Sigma$, then $\{\mathfrak p_r:=\Psi_r^*\mathfrak p\}$ yields a path in $\mathfrak P(\Sigma)$. The next lemma shows that all paths in $\mathfrak P(\Sigma)$ arise in this way.
\begin{lem}\label{l:triv}
The following two statements hold:
\begin{enumerate}[(i)]
\item Every oriented $S^1$-bundle $\mathfrak p_0\in\mathfrak P(\Sigma)$ has a $C^1$-neighbourhood $\mathcal W$ and a continuous map $(\mathfrak p\in\mathcal W)\mapsto (\Psi_{\mathfrak p}\in\mathrm{Diff}(\Sigma))$ such that
\[
\bullet\ \ \Psi_{\mathfrak p_0}=\id_\Sigma,\qquad\qquad \bullet\ \ \Psi_{\mathfrak p}^*\mathfrak p=\mathfrak p_0,\quad \forall\,\mathfrak p\in\mathcal W.
\]
\item If $\{\mathfrak p_r:\Sigma\to M_r\}$ is a path in $\mathfrak P(\Sigma)$, there exists an isotopy $\{\Psi_r\}$ of $\Sigma$ such that
\[
\qquad\bullet\ \ \Psi_{0}=\id_\Sigma,\ \qquad\qquad \bullet\ \ \Psi_r^*\mathfrak p_r=\mathfrak p_0,\quad \forall\,r\in[0,1].
\]
\end{enumerate}
\end{lem}
\begin{proof}
Pick an arbitrary connection $\eta\in\mathcal K(\mathfrak p_0)$ and an arbitrary Riemannian metric on $M$ with injectivity radius $\rho_\inj$ and distance function $\dist:M\times M\to[0,\infty)$. We define the open subset of $M\times\Sigma$
\[W:=\Big\{(q,z)\in M\times\Sigma\ \Big|\ \dist(q,\mathfrak p_0(z))<\rho_\inj/2\Big\}\] 
and let $\mathfrak q_M:W\to M$ and $\mathfrak z_\Sigma:W\to\Sigma$ be the projections on the two factors. We define $\mathcal W$ as the space of all oriented $S^1$-bundles $\mathfrak p$ such that the following two properties hold: 
\begin{itemize}
\item For every $z_1,z_2$ in $\Sigma$ belonging to the same $\mathfrak p$-fibre, we have $\dist(\mathfrak p_0(z_1),\mathfrak p_0(z_2))<\rho_\inj/2$; 
\item For every $z\in \Sigma$, the disc $D_z\subset \Sigma$ is transverse to the orbits of $\mathfrak p$. Here, $D_z$ is the union of all the $\eta$-horizontal lifts through $z$ of the geodesic rays in $M$ emanating from $\mathfrak p_0(z)$ and with length $\rho_\inj$.
\end{itemize}
For every $\mathfrak p\in \mathcal W$, we construct a map of fibre bundles over $M$ given by
\[
S:W\to \ta M,\qquad S(q,z):=\big(q,v(q,z)\big), \quad v(q,z):=\int_{S^1}\exp_q^{-1}\circ\,\mathfrak p_0\circ\gamma_z(t)\di t,
\]
where $\gamma_z:S^1\to \Sigma$ is the oriented $\mathfrak p$-fibre passing through $z$. By definition, $S$ is constant along the $\mathfrak p$-fibres. Let $0_M\subset \ta M$ be the zero section and set $G:=S^{-1}(0_M)$. We claim that
\begin{itemize}
\item $S$ is transverse to $0_M$,
\item $\mathfrak p_M:=\mathfrak q_M|_{G}:G\to M$ is an $S^1$-bundle map,
\item $\Psi_\Sigma:=\mathfrak z_\Sigma|_{G}:G\to \Sigma$ is a diffeomorphism.
\end{itemize}
This is clear if $\mathfrak p=\mathfrak p_0$, since in this case $S(q,z)=\exp_q^{-1}(\mathfrak p_0(z))$ and $G=\{(q,z)\ |\ \mathfrak p_0(z)=q\}$. Therefore, up to shrinking $\mathcal W$, it also holds for $\mathfrak p\in \mathcal W$, as $S$ depends continuously on $\mathfrak p$ and being transverse, or a submersion or a diffeomorphism is a $C^1$-open condition \cite[Chapter 2]{Hir94}. Thus, we get an $S^1$-bundle $\mathfrak p':=\mathfrak p_M\circ \Psi_\Sigma^{-1}:\Sigma\to M$, which is equivalent to $\mathfrak p$. We now define a map $\Phi_{\mathfrak p}:\Sigma\to\Sigma$ such that $\Phi_{\mathfrak p}^*\mathfrak p_0=\mathfrak p$. For every $z\in\Sigma$, $\Phi_{\mathfrak p}(z)$ is the parallel transport with respect to $\eta$ along the radial geodesic on $M$ connecting $\mathfrak p_0(z)$ with $\mathfrak p'(z)$.  The dependence of $\Phi_{\mathfrak p}$ from $\mathfrak p$ is continuous and clearly $\Phi_{\mathfrak p_0}=\id_\Sigma$. As being a diffeomorphism is a $C^1$-open condition, we have $\Phi_{\mathfrak p}\in\mathrm{Diff}(\Sigma)$ and (i) is proved by setting $\Psi_{\mathfrak p}=\Phi_{\mathfrak p}^{-1}$.

For part (ii), we break the given path into short paths and apply the first part.
\end{proof}

\section{Weakly Zoll pairs and Zoll odd-symplectic forms}\label{sec:weak_Zoll}

\subsection{Definitions and first properties}\label{ss:weakzoll}
\begin{dfn}\label{d:weaklyzoll}
A couple $(\mathfrak p,c)$ is called a {\bf weakly Zoll pair} if $\mathfrak p:\Sigma\to M$ is an element of $\mathfrak P(\Sigma)$ and $c\in H^2_\dR(M)$. We say that $\Omega\in\Xi^2(\Sigma)$ is \textbf{associated} with $(\mathfrak p,c)$, if $\Omega=\mathfrak p^*\om$ for some $\om\in\Xi^2_c(M)$. We denote by $\mathfrak Z(\Sigma)$ the set of weakly Zoll pairs and by $\mathfrak Z_C(\Sigma)$ the subset of those $(\mathfrak p,c)\in \mathfrak Z(\Sigma)$ with $\mathfrak p^*c=C\in H^2_\dR(\Sigma)$.
\end{dfn}

Let $(\mathfrak p_0,c_0)\in\mathfrak Z_C(\Sigma)$ with $\mathfrak p_0:\Sigma\to M_0$ and take any $\om_0\in\Xi^2_{c_0}(M_0)$. We consider the volume functionals $\Vol:\Omega^1(\Sigma)\to\R$ and $\Fvol:\Xi^2_C(\Sigma)\to\R$ defined in Section \ref{c:vol} with respect to $\Omega_0:=\mathfrak p_0^*\om_0\in\Xi^2_C(\Sigma)$. We use $\Fvol$ to get a volume on $\mathfrak Z_C(\Sigma)$ denoted with the same name: 
\begin{equation*}\label{eq:vol_weakzoll}
\Fvol:\mathfrak Z_C(\Sigma)\to \R,\qquad \Fvol(\mathfrak p,c):=\Fvol(\mathfrak p^*\om),\quad \om\in\Xi_c^2(M).
\end{equation*}

\begin{lem}\label{l:weakzollvol}
The following three statements hold.
\begin{enumerate}[(i)]
\item The volume functional $\Fvol:\mathfrak Z_C(\Sigma)\to \R$ is well-defined.
\item For all $\zeta\in\Omega^1(M_0)$, $\mathfrak p_0^*\zeta$ is $\mathfrak p_0^*\om_0$-normalised, namely 
\begin{equation*}
\Vol(\mathfrak p_0^*\zeta)=0. 
\end{equation*}
\item The functions $\Fvol:\Xi^2_C(\Sigma)\to \R$, $\Fvol:\mathfrak Z_C(\Sigma)\to\R$ depend only on the pair $(\mathfrak p_0,c_0)$ and not on the chosen $\om_0\in\Xi^2_{c_0}(M_0)$.
\end{enumerate}
\end{lem}
\begin{proof}
All three items will stem out a preliminary result. Let $(\mathfrak p,c)\in\mathfrak Z_C(\Sigma)$ with $\om\in\Xi^2_c(M)$ and pick $\alpha_\om\in\Omega^1(\Sigma)$ such that $\mathfrak p_0^*\om_0+\di\alpha_\om=\mathfrak p^*\om$. 
We claim that
\begin{equation}\label{e:volaomega}
\Vol(\alpha_\om+\mathfrak p^*\zeta)=\Vol(\alpha_\om),\qquad\forall\,\zeta\in\Omega^1(M).
\end{equation}
Indeed, if we set $\alpha_r:=\alpha_\om+r\mathfrak p^*\zeta$ so that $\mathfrak p_0^*\om_0+\di\alpha_r=\mathfrak p^*(\om+r\di\zeta)$ and $\dot\alpha_r=\mathfrak p^*\zeta$, then
\[
\Vol(\alpha_\om+\mathfrak p^*\zeta)-\Vol(\alpha_\om)=\int_0^1\vol_{\alpha_r}\cdot \dot\alpha_r\di r=\int_0^1\Big(\int_\Sigma \mathfrak p^*\zeta\wedge\mathfrak p^*(\om+r\di\zeta)^n\Big)\di r=0.
\]
If $\om'$ is another form in $\Xi^2_c(M)$, then there is $\zeta\in\Omega^1(M)$ with $\om'-\om=\di\zeta$ and \eqref{e:volaomega} implies
\[
\Fvol(\mathfrak p^*\om')=\Fvol(\mathfrak p^*\om),
\]
which establishes item (i). Choosing $(\mathfrak p,c)=(\mathfrak p_0,c_0)$ and $\alpha=0$ in \eqref{e:volaomega}, we get item (ii). For item (iii), we take another form $\om_0'\in\Xi^2_{c_0}(M_0)$ and write $\om_0'-\om_0=\di\zeta_0$ for some $\zeta_0\in\Omega^1(M_0)$.  Applying Lemma \ref{l:volcomp} with $\Omega_0=\mathfrak p_0^*\om_0$ and $\Omega_0'=\mathfrak p^*_0\om_0'$ together with item (ii), we deduce
\begin{equation*}
\Vol(\alpha)=\Vol'(\alpha-\mathfrak p_0^*\zeta_0),\qquad \forall\,\alpha\in\Omega^1(\Sigma),
\end{equation*}
where $\Vol'$ is the volume functional associated with $\Omega_0'$. 
This implies $\Fvol(\Omega)=\Fvol'(\Omega)$ for all $\Omega\in\Xi^2_C(\Sigma)$.
\end{proof}
We have a canonical projection
\begin{equation*}\label{def:mathfrakP}
\mathfrak P:\mathfrak Z(\Sigma)\to\mathfrak P(\Sigma),\qquad \mathfrak P(\mathfrak p,c)=\mathfrak p.
\end{equation*}
Let us fix a class $C\in H^2_\dR(\Sigma)$ and a connected component \label{def:mathfrakP0}$\mathfrak P^0(\Sigma)$ of $\mathfrak P(\Sigma)$. We define
\begin{equation*}\label{e:mathfrakzoc}
\mathfrak Z_C^0(\Sigma):= \mathfrak P^{-1}(\mathfrak P^0(\Sigma))\cap\mathfrak Z_C(\Sigma).
\end{equation*}
We consider the restriction of $\mathfrak P$ on this set
\begin{equation*}\label{e:mathfrakP0C}
\mathfrak P^0_C:\mathfrak Z_C^0(\Sigma)\to\mathfrak P^0(\Sigma).
\end{equation*}
By \eqref{e:gysin}, for every $(\mathfrak p,c)\in\mathfrak Z_C^0(\Sigma)$, we have a surjective map
\begin{equation*}
\R\to (\mathfrak P^0_C)^{-1}(\mathfrak p),\qquad A\mapsto (\mathfrak p,Ae+c),
\end{equation*}
where $e$ is minus the real Euler class of $\mathfrak p$. Finally, we define the evaluation map
\begin{equation*}\label{eq:evaluation_map}
\ev:\mathfrak Z^0_C(\Sigma)\to\R,\qquad \ev(\mathfrak p,c):=\langle c^n,[M]\rangle.
\end{equation*}
\begin{dfn}\label{d:nondegenerate}
We say that $\mathfrak Z^0_C(\Sigma)$ is \textbf{non-degenerate}, if the map $\ev:\mathfrak Z^0_C(\Sigma)\to\R$ is non-zero, namely if there exists $(\mathfrak p,c)\in \mathfrak Z^0_C(\Sigma)$ with $\ev(\mathfrak p,c)\neq0$.
\end{dfn}
\begin{lem}\label{l:nondegenerate}
Let $\mathfrak p_0:\Sigma\to M_0$ be an element in $\mathfrak P^0(\Sigma)$, and let $e_0\in H^2_\dR(M_0)$ be minus the real Euler class of $\mathfrak p_0$. For every $\mathfrak p:\Sigma\to M$ in $\mathfrak P^0(\Sigma)$, there exists a diffeomorphism $\Psi:\Sigma\to\Sigma$ isotopic to $\id_\Sigma$ such that $\Psi^*\mathfrak p=\mathfrak p_0$. If $\psi:M_0\to M$ is the quotient map of $\Psi$ (see \eqref{eq:commutative_diagram2}), then, for every $(\mathfrak p,c)\in (\mathfrak P_C^0)^{-1}(\mathfrak p)$ and $A\in\R$, there holds 
\begin{align*}
(i)&\hspace{-60pt}&\psi^*(Ae+c)&=Ae_0+\psi^*c,\\
(ii)&\hspace{-60pt}&\ev(\mathfrak p,Ae+c)&=\ev(\mathfrak p_0,Ae_0+\psi^*c),\\
(iii)&\hspace{-60pt}& \Fvol(\mathfrak p,Ae+c)&=\Fvol(\mathfrak p_0,Ae_0+\psi^*c).
\end{align*}
\end{lem}
\begin{proof}
Since $\mathfrak P^0(\Sigma)$ is connected, the existence of a diffeomorphism $\Psi$ as in the statement follows from Lemma \ref{l:triv}.(ii). By Re, we see that $\psi^*e=e_0$, which immediately implies item (i). For item (ii), we observe that $[M]=\psi_*[M_0]$ since $\psi$ is an orientation-preserving diffeomorphism, and compute
\begin{equation*}
\ev(\mathfrak p,Ae+c)=\langle (Ae+c)^n,\psi_*[M_0]\rangle=\langle \psi^*(Ae+c)^n,[M_0]\rangle=\ev(\mathfrak p_0,Ae_0+\psi^*c). 
\end{equation*}
For the last relation, we take any $\om_A\in\Xi^2(M)$ such that $[\om_A]=Ae+c$. Then $\Omega_A:=\mathfrak p^*\om_A$ is associated with $(\mathfrak p,Ae+c)$. The form $\Psi^*\Omega_A\in \Xi^2_C(\Sigma)$ is associated with $(\mathfrak p_0,Ae_0+\psi^*c)$. Indeed, $\Psi^*\Omega_A=\Psi^*\mathfrak p^*\om_A=\mathfrak p_0^*\psi^*\om_A$ and $[\psi^*\om_A]=\psi^*(Ae+c)=Ae_0+\psi^*c$. Thus, from Proposition \ref{prp:vol}, we derive
\[
\Fvol(\mathfrak p,Ae+c)=\Fvol(\Omega_A)=\Fvol(\Psi^*\Omega_A)=\Fvol(\mathfrak p_0, Ae_0+\psi^*c).\qedhere
\]
\end{proof}
\begin{cor}\label{c:weaklyfour}
If $\mathfrak Z_C^0(\Sigma)$ is non-empty, the four statements below hold.
\begin{enumerate}[(i)]
\item The real Euler class of some bundle in $\mathfrak P^0(\Sigma)$ vanishes, if and only if the real Euler class of every element of $\mathfrak P^0(\Sigma)$ vanishes.
\item The set $\mathfrak Z^0_C(\Sigma)$ is non-degenerate, if and only if, for every $\mathfrak p\in\mathfrak P^0(\Sigma)$ there exists a pair $(\mathfrak p,c)\in\mathfrak Z^0_C(\Sigma)$ such that $\ev(\mathfrak p,c)\neq0$.
\item If the real Euler class of the bundles in $\mathfrak P^0(\Sigma)$ vanishes, then $\mathfrak P^0_C$ is a diffeomorphism and $\ev:\mathfrak Z^0_C(\Sigma)\to\R$ is a constant map.
\item If the real Euler class of the bundles in $\mathfrak P^0(\Sigma)$ does not vanish, then $\mathfrak P^0_C$ has the structure of an affine $\R$-bundle. The $\R$-action on some $(\mathfrak p,c)\in\mathfrak Z^0_C(\Sigma)$ is given by
\begin{equation*}
A\cdot(\mathfrak p,c)=(\mathfrak p,Ae+c),\qquad\forall\,A\in\R.
\end{equation*}
\end{enumerate}
\end{cor}
\begin{proof}
Items (i) and (ii) are direct consequences of items (i) and (ii) in Lemma \ref{l:nondegenerate}. 

To prove item (iii), let us assume that bundles in $\mathfrak P^0(\Sigma)$ have vanishing real Euler class. By \eqref{e:gysin}, the $\mathfrak P^0_C$-fibres are sets containing only one element. This shows that $\mathfrak P^0_C$ is a diffeomorphism. This together with item (ii) in Lemma \ref{l:nondegenerate} yields that $\ev$ is constant. 

We suppose that the real Euler class does not vanish and prove item (iv). By \eqref{e:gysin}, the $\mathfrak P^0_C$-fibres are in bijection with $\R$ through the maps $A\mapsto (\mathfrak p,Ae+c)$. We construct an explicit local trivialisation of the bundle structure.  Lemma \ref{l:triv}.(i) yields a neighbourhood $\mathcal W$ of  $\mathfrak p_0\in\mathfrak P^0(\Sigma)$ and a map $\mathfrak p\mapsto\Psi_{\mathfrak p}$ from $\mathcal W$ to a neighbourhood of $\id_\Sigma$ inside $\mathrm{Diff}(\Sigma)$ such that $\Psi_{\mathfrak p_0}=\id_\Sigma$ and $\Psi_{\mathfrak p}^*\mathfrak p=\mathfrak p_0$. We define $\Phi_{\mathfrak p}:=\Psi_{\mathfrak p}^{-1}$ and let $\phi_{\mathfrak p}:M\to M_0$ be the quotient map. If $(\mathfrak p_0,c_0)\in (\mathfrak P_C^0)^{-1}(\mathfrak p_0)$, then the map
\[
\mathcal W\times \R\to (\mathfrak P_C^0)^{-1}(\mathcal W),\qquad (\mathfrak p,A)=(\mathfrak p,Ae+\phi_{\mathfrak p}^*c_0)
\]
provides a local trivialisation around $\mathfrak p_0$.
\end{proof}
\begin{dfn}\label{d:zollodd}
A two-form $\Omega$ on $\Sigma$ is called {\bf Zoll}, if it is odd-symplectic and there exists $\mathfrak p_\Omega:\Sigma\to M_\Omega$ in $\mathfrak P(\Sigma)$ such that the oriented $\mathfrak p_\Omega$-fibres are positively tangent to $\ker\Omega$. In this case, we say that $\mathfrak p_\Omega$, which is determined up to equivalence, is the bundle \textbf{associated} with $\Omega$. We write $\ZZ(\Sigma)$ for the set of Zoll (odd-symplectic) forms on $\Sigma$ and $\ZZ_C(\Sigma)$ for the subset of those forms with class $C\in H^2_\dR(\Sigma)$. 
\end{dfn}
If $\Omega$ is Zoll, then the form $\Omega$ descends to the base manifold $M_\Omega$, since $\di\Omega=0$. Namely, there exists $\om_\Omega\in\Omega^2(M_\Omega)$ such that 
\[
\Omega=\mathfrak p_\Omega^*\om_\Omega.
\] 
The form $\om_\Omega$ is closed as well, since $\mathfrak{p}^*_\Omega$ is injective. Furthermore, $\om_\Omega$ is symplectic since $\Omega$ is odd-symplectic and the orientations induced by $\om_\Omega$ and by $\mathfrak p_\Omega$ on $M_\Omega$ coincide. We have a natural inclusion
\[
\ZZ(\Sigma)\to\mathfrak Z(\Sigma),\qquad \Omega\mapsto (\mathfrak p_\Omega,[\om_\Omega]).
\]
As $\om_\Omega$ is positive symplectic, we deduce that
\begin{equation*}
\ev(\mathfrak p_{\Omega},[\om_\Omega])>0.
\end{equation*}
\begin{rmk}\label{r:zollnondeg}
The inequality above implies that for $\Omega\in\mathcal Z(\Sigma)$ the component $\mathfrak Z^0_C(\Sigma)$ which $(\pm\mathfrak p_\Omega,[\om_\Omega])$ is belonging to, is non-degenerate.
\end{rmk} 
\begin{rmk}\label{r:ZolldiffZoll}
By \cite{BW58}, Zoll forms with vanishing cohomology class are just the exterior differentials of Zoll contact forms. Indeed, if $\Omega=\mathfrak p_\Omega^*\om_\Omega$ is exact, then the cohomology class of $\om_\Omega$ is non-zero, as $\om_\Omega$ is symplectic, and lies in the kernel of the map $\mathfrak p^*_\Omega$. In view of \eqref{e:gysin}, this means that there exists $T\in\R\setminus\{0\}$ such that $[\tfrac1T\omega_\Omega]=e$. Then, there exists a connection $\eta\in\mathcal K(\mathfrak p_\Omega)$ with $\kappa_{\eta}=\tfrac1T\omega_\Omega$ and $\alpha:= T\eta$ is a Zoll contact form with $\di\alpha=\Omega$. The orientations $\mathfrak o_\Sigma$ and $\mathfrak o_\alpha$ coincide exactly when $T>0$.
\end{rmk}

Before moving further, it is worthwhile to briefly discuss stability properties of Zoll forms. Let $\{\Omega_r\}$ be a path in $\ZZ_C(\Sigma)$ with the corresponding path of associated bundles $\{\mathfrak p_r\}$. If the real Euler class of the bundles vanishes, we aim at finding an isotopy $\{\Psi_r\}$ of $\Sigma$, such that
\[
\Psi_r^*\Omega_r=\Omega_0.
\]
If the real Euler class of the bundles does not vanish, we aim at finding an isotopy $\{\Psi_r\}$ of $\Sigma$, a path of real numbers $\{A_r\}$ with $A_0=0$, and $\eta_0\in\mathcal K(\mathfrak p_0)$ such that
\[
\Psi_r^*\Omega_r=\Omega_0+A_r\di\eta_0.
\]
In this last case, it seems unlikely that all paths admit such an expression. That this happens if $C=0$ is a result of Weinstein \cite{Wei74}. The stability for $e=0$ is reminiscent of \cite{Gin87}. 
\begin{prp}\label{p:wei}
Let $\{\Omega_r\}$ be a path in $\ZZ_C(\Sigma)$ with the path of associated bundles $\{\mathfrak p_r\}$.
\begin{itemize}
\item If $C=0$, there is $\eta_0\in\mathcal K(\mathfrak p_0)$, an isotopy $\{\Psi_r\}$ of $\Sigma$ and real numbers $\{T_r\}$ such that
\[
\Psi_r^*\Omega_r=T_r\di\eta_0=\Omega_0+(T_r-T_0)\di\eta_0.
\]
\item If $e=0$, there is an isotopy $\{\Psi_r\}$ of $\Sigma$ such that $\Psi_r^*\Omega_r=\Omega_0$.
\end{itemize}
\end{prp}
\begin{proof}
By Lemma \ref{l:triv}, we can suppose in both cases that $\{\Omega_r=\mathfrak p_0^*\om_r\}$, where $\{\om_r\}$ is a path of symplectic forms on $M$. If $C=0$, \cite{KN63} yield $\{\eta_r\}\subset\mathcal K(\mathfrak p_0)$ and non-zero real numbers $\{T_r\}$ such that $\Omega_r=\di(T_r\eta_r)$ and $\eta_r=\eta_0+\mathfrak p^*\zeta_r$, where $\{\zeta_r\}$ are one-forms on $M$. Setting $\alpha_r:=T_r\eta_r$ and applying the Gray stability theorem from Remark \ref{r:stabil}.(a), we see that there exists an isotopy $\{\Psi_r\}$ and a path $\{H_r:\Sigma\to\R\}$ with $H_0=0$ such that
\[
\Psi_r^*\alpha_r=(1+\tfrac{1}{T_0}H_r)\alpha_0.
\]
Since $\dot\alpha_r=\dot T_r\eta_r+T_r\mathfrak p^*\dot\zeta_r$, equation \eqref{e:stab2} implies $\dot H_r=\dot T_r$ and hence $H_r=T_r-T_0$, which readily implies the desired formula.

If $e=0$, then $\om_r=\om_0+\di\zeta_r$ due to \eqref{e:gysin}. By Remark \ref{r:stabil}.(b) and Proposition \ref{p:Moser} with $\alpha_0=0$, it follows that $\Psi_r^*\Omega_r=\Omega_0+\di(H_r\eta_0)$, for some paths $\{\Psi_r\}$ and $\{H_r\}$ as above. Again by equation \eqref{e:stab2}, we conclude that $H_r=0$ for all $r$.
\end{proof}
We are now ready to classify Zoll odd-symplectic forms on a three-dimensional manifold, as promised in the introduction.

\subsection{Classification of Zoll odd-symplectic forms in dimension three}
\subsubsection*{Proof of Proposition \ref{p:cla}}
Let $\Sigma$ have dimension three, let $b_\Sigma$ be the rank of the free part of $H_1(\Sigma;\Z)$, and let  $|H_1^\mathrm{tor}(\Sigma;\Z)|$ denote the cardinality of its torsion subgroup. Let $\Omega\in\ZZ_C(\Sigma)$ be a Zoll form with cohomology class $C$ and let $\mathfrak p=\mathfrak p_\Omega$ its associated bundle. If $C=0$, we know from Remark \ref{r:ZolldiffZoll} that $\Omega$ is the differential of a Zoll contact form. In particular, $\Sigma$ is the total space of a non-trivial oriented $S^1$-bundle, $b_\Sigma$ is even and we already treated this case in \cite[Proposition 1.2]{BK19a}. 

Suppose that $C\neq 0$. In this case, $e=0$ by equivalence (ii) in \eqref{e:cne} . Since $M$ is a surface this implies that $e_\Z=0$, and hence $\mathfrak p$ is trivial. Therefore, $\Sigma\cong M\times S^1$ and we see that
\begin{equation}\label{e:genus}
b_\Sigma=1+2\,\mathrm{genus}(M),\qquad |H_1^\mathrm{tor}(\Sigma;\Z)|=1.
\end{equation}
In particular, $b_\Sigma$ is odd. Let $\Omega'\in\ZZ_{C'}(\Sigma)$ be another Zoll form with class $C'$. Since $b_\Sigma$ is odd, then $C'\neq0$ and $\mathfrak p':\Sigma\to M'$ is the trivial bundle. We write $\Omega=\mathfrak{p}^*\om$ and $\Omega'=\mathfrak{p}'^*\om'$, where $\om$ and $\om'$ are symplectic forms on $M$ and $M'$ respectively. From \eqref{e:genus}, $M$ and $M'$ have the same genus, so that there is a diffeomorphism $\psi:M\to M'$. Since $\omega$ and $\omega'$ are symplectic on $M$ and $M'$ and $H^2_{\mathrm{dR}}(M;\R)\cong \R\cong H^2_{\mathrm{dR}}(M';\R)$, by Moser's trick, we can assume that 
\[
\psi^*\omega'=T \omega,\qquad \text{for some }T>0.
\] 
As both $\mathfrak p$ and $\mathfrak p'$ are trivial bundles, it is immediate to find a diffeomorphism
\begin{equation}\label{e:fsigmasigma}
\Psi:\Sigma\to\Sigma
\end{equation}
lifting $\psi$ and such that $\Psi^*\mathfrak p'=\mathfrak p$. Therefore, $f$ preserves the orientation of $\Sigma$ and $\Psi^*\Omega'=T\Omega$.
\medskip

We now want to describe the connected components of the space of Zoll forms on $\Sigma$ with fixed cohomology class. In order to do so, we first determine the connected components of the space of oriented $S^1$-bundles $\mathfrak P(\Sigma)$ with the help of classical results in low-dimensional topology. As a preliminary observation, we point out that if $\mathfrak p\in\mathfrak P(\Sigma)$, then the class $[\mathfrak p^{-1}(\pt)]_\Z\in H_1(\Sigma;\Z)$ is primitive, since its intersection number with a global section of $\mathfrak p$ is equal to $1$. We distinguish three cases: $M=S^2$, $M=\T^2$, and $\mathrm{genus}(M)\geq2$.
\medskip

\noindent{\bf Case 1:} $\Sigma\cong S^2\times S^1$. We regard $S^2$ as the unit sphere in $\R^3$. We recall that the mapping class group of orientation-preserving diffeomorphisms of $S^2\times S^1$ is given by
\[
\mathrm{MCG}(S^2\times S^1)\cong\tfrac{\Z}{2\Z}\,[\Psi_1]\oplus \tfrac{\Z}{2\Z}\,[\Psi_2].
\] 
Here, the generators $\Psi_1,\Psi_2:S^2\times S^1\to S^2\times S^1$ are given by
\[
\Psi_1(q,\phi):=(-q,-\phi),\qquad \Psi_2(q,\phi):=(\rho_\phi(q),\phi),\qquad \forall\, (q,\phi)\in S^2\times S^1,
\] 
where $\rho_\phi:S^2\to S^2$ is the rotation of angle $2\pi \phi$ around the north pole. We consider the standard projection $\mathfrak p_+:S^1\times S^2\to S^2$ along $S^1$ and we define $\mathfrak p_-:=\Psi_1^*\mathfrak p_+$. 

We claim that $\mathfrak P(S^2\times S^1)$ has four connected components containing respectively $\mathfrak p_+$, $\Psi_2^*\mathfrak p_+$, $\mathfrak p_-$ and $\Psi_2^*\mathfrak p_-$. To this purpose, we observe that $[\mathfrak p_+^{-1}(\pt)]_\Z=[(\Psi_2^*\mathfrak p_+)^{-1}(\pt)]_\Z$ and $[\mathfrak p_-^{-1}(\pt)]_\Z=[(\Psi_2^*\mathfrak p_-)^{-1}(\pt)]_\Z$ are distinct and yield the two primitive homology classes in $H_1(S^2\times S^1;\Z)\cong\Z$. Therefore, we just need to show that $\mathfrak p_+$, $\Psi_2^*\mathfrak p_+$ are not in the same connected component. If, by contradiction, there were a path in $\mathfrak P(S^2\times S^1)$ from $\mathfrak p_+$ to $\Psi_2^*\mathfrak p_+$, then by Lemma \ref{l:triv}.(ii) there would exist a diffeomorphism $\Psi_2':S^2\times S^1\to S^2\times S^1$ isotopic to $\Psi_2$ such that $\mathfrak p_+=(\Psi_2')^*\mathfrak p_+$. This forces $\Psi_2'$ to be of the form
\[
\Psi_2'(q,\phi)=(\psi(q),\phi'(x,\phi))
\] 
where $\phi\mapsto \phi'(q,\phi)$ is an orientation-preserving diffeomorphism of $S^1$, for all $q\in S^2$, and $\psi$ is an orientation-preserving diffeomorphism of $S^2$. However, every orientation-preserving diffeomorphisms of $S^2$ is isotopic to the identity and the set $\mathrm{Diff}_+(S^1)$ of orientation-preserving diffeomorphisms of $S^1$ is homotopy equivalent to $S^1$, so that $\pi_2(\mathrm{Diff}_+(S^1))=0$. Thus, the map $\Psi_2'$ would be isotopic to the identity, which is impossible, as it is isotopic to $\Psi_2$. 

Let $\mathfrak p$ be an arbitrary element in $\mathfrak P(S^2\times S^1)$. We have either $[\mathfrak p^{-1}(\mathrm{pt})]_\Z=[\mathfrak p^{-1}_+(\mathrm{pt})]_\Z$ or $[\mathfrak p^{-1}(\mathrm{pt})]_\Z=[\mathfrak p^{-1}_-(\mathrm{pt})]_\Z$. We have seen in \eqref{e:fsigmasigma} that there exists an isomorphism of oriented $S^1$-bundles $\Psi:S^2\times S^1\to S^2\times S^1$ preserving the orientation and such that $\Psi^*\mathfrak p_+=\mathfrak p$. Therefore, if $[\mathfrak p^{-1}(\mathrm{pt})]_\Z=[\mathfrak p^{-1}_+(\mathrm{pt})]_\Z$, then $\Psi$ is either isotopic to $\id_{S^2\times S^1}$ or to $\Psi_2$ and $\mathfrak p$ is either homotopic to $\mathfrak p_+$ or to $\Psi_2^*\mathfrak p_+$; if $[\mathfrak p^{-1}(\mathrm{pt})]_\Z=[\mathfrak p^{-1}_-(\mathrm{pt})]_\Z$, then $\Psi$ is either isotopic to $\Psi_1$ or to $\Psi_1\circ \Psi_2$ and $\mathfrak p$ is either homotopic to $\mathfrak p_-$ or to $\Psi_2^*\mathfrak p_-$. 
\medskip

\noindent{\bf Case 2:} $\Sigma\cong\T^2\times S^1=\T^3$. To every orientation-preserving diffeomorphism $\Psi$ on $\T^3$, we can associate the induced map in homology 
\[
H_1(\Psi):H_1(\T^3;\Z)\to H_1(\T^3;\Z).
\] 
As $H_1(\T^3;\Z)\cong\Z^3$, we can identify $H_1(\Psi)$ with an element of $\mathrm{SL}(3;\Z)$. Conversely, every $A\in\mathrm{SL}(3;\Z)$ gives an orientation-preserving diffeomorphism $\Psi_A:\T^3\to\T^3$ with $H_1(\Psi_A)=A$. Moreover, the mapping class group is computed explicitly through the isomorphism
\begin{equation}\label{eq:mcg_T3}
\mathrm{MCG}(\T^3)\to\mathrm{SL}(3;\Z),\qquad \Psi\mapsto H_1(\Psi).
\end{equation}
Let $\mathfrak p_0:\ta^3\cong \T^2\times S^1\to\T^2$ be the standard projection along $S^1$. For every primitive class ${\mathfrak h}\in H_1(\T^3;\Z)$, there exists $A_{\mathfrak h}\in\mathrm{SL}(3;\Z)$ with $A_{\mathfrak h}\cdot\mathfrak h=[\mathfrak p^{-1}(\pt)]_\Z$, so that the fibres of the oriented $S^1$-bundle 
\[
\mathfrak p_{\mathfrak h}:=\Psi_{A_{\mathfrak h}}^*\mathfrak p_0
\] 
lie in the homology class ${\mathfrak h}$. The map $A_{\mathfrak h}$ is not unique. However, if $A_{\mathfrak h}'$ is another such map, then $\Psi_{A_{\mathfrak h}'}^*\mathfrak p_0$ and $\Psi_{A_{\mathfrak h}}^*\mathfrak p_0$ are equivalent bundles, and up to equivalence, $\mathfrak p_{\mathfrak h}$ depends only on $\mathfrak h$. 

On the other hand, if $\mathfrak p\in\mathfrak P(\T^3)$, we claim that $\mathfrak p_{[\mathfrak p^{-1}(\pt)]_\Z}$ and $\mathfrak p$ are isotopic. Indeed, let $\Psi:\T^3\to\T^3$ be an orientation-preserving diffeomorphism with $\Psi^*\mathfrak p_0=\mathfrak p$, as in \eqref{e:fsigmasigma}. The equality $H_1(\Psi)([\mathfrak p^{-1}(\pt)]_\Z)=[\mathfrak p^{-1}_0(\pt)]_\Z$ implies that $H_1(\Psi)=A_{[\mathfrak p^{-1}(\pt)]_\Z}$ so that there holds $\mathfrak p_{[\mathfrak p^{-1}(\pt)]_\Z}=\Psi_{H_1(\Psi)}^*\mathfrak p_0$. Moreover, since $\Psi$ and $\Psi_{H_1(\Psi)}$ are isotopic due to \eqref{eq:mcg_T3}, $\Psi_{H_1(\Psi)}^*\mathfrak p_0$ is homotopic to $\mathfrak p$ and the claim is proven.
\medskip

\noindent{\bf Case 3:} $\Sigma\cong M\times S^1$ with $\mathrm{genus}(M)\geq2$. Let $\mathfrak p_+:M\times S^1\to M$ be the oriented $S^1$-bundles given by the standard projection and set $\mathfrak p_-:=-\mathfrak p_+$. The bundles $\mathfrak p_+$ and $\mathfrak p_-$ are not homotopic since the homotopy classes of their fibres are different. If $\mathfrak p$ is any element of $\mathfrak P(M\times S^1)$, by \cite[Satz (5.5)]{Wal67}, there exists a diffeomorphism $\Psi: M\times S^1\to M\times S^1$ isotopic to the identity such that either $\Psi^*\mathfrak p=\mathfrak p_+$ or $\Psi^*\mathfrak p=\mathfrak p_-$. Thus, $\mathfrak p$ is either isotopic to $\mathfrak p_-$ or to $\mathfrak p_+$. 
\medskip

This finishes the description of the connected components of $\mathfrak P(\Sigma)$ in the three cases. To determine the connected components of the space of Zoll forms with fixed cohomology class, we consider the map assigning to each Zoll odd-symplectic form its associated bundle
\[
\mathfrak P_{\ZZ}:\ZZ(\Sigma)\to\mathfrak P(\Sigma),\qquad \Omega\mapsto \mathfrak p_{\Omega}
\]
If $\mathfrak p=\mathfrak P_{\ZZ}(\Omega)$, then $\mathrm{PD}([\Omega])$ is a positive multiple of $[\mathfrak p^{-1}(\pt)]$ by \eqref{e:pd}. Moreover, the map $\mathfrak P_{\ZZ}$ is surjective, since the quotient manifold of any bundle is an orientable surface, and therefore, possesses a positive symplectic form. In particular, if $\mathfrak p\in\mathfrak P(\Sigma)$ and $C\in H^2_\dR(\Sigma)$ are such that $\mathrm{PD}(C)$ is a positive multiple of $[\mathfrak p^{-1}(\pt)]$, then there exists $\Omega\in\ZZ_C(\Sigma)\subset\ZZ(\Sigma)$ such that $\mathfrak P_{\ZZ}(\Omega)=\mathfrak p$. This shows that
\begin{itemize}
\item if $M= S^2$ or $\T^2$, then $\ZZ_C(\Sigma)$ is not empty, for all $C\neq 0$;
\item if $M$ has higher genus, then $\ZZ_C(\Sigma)$ is not empty if and only if $C$ is a non-zero element in $\mathfrak p_+^*\big(H^2_\dR(M)\big)$.
\end{itemize} 
We fix now some $C\in H^2_\dR(\Sigma)$ for which $\ZZ_C(\Sigma)$ is non-empty and consider $\mathfrak P_{\ZZ,C}:=\mathfrak P_{\ZZ}|_{\ZZ_C(\Sigma)}$. As its non-empty fibres are convex, the connected components of $\ZZ_C(\Sigma)$ correspond through $\mathfrak P_{\ZZ,C}$ to the connected components of $\mathfrak P_{\ZZ,C}(\ZZ_C(\Sigma))$. The latter set is the union of those connected components of $\mathfrak P(\Sigma)$, whose elements $\mathfrak p$ satisfy $\mathrm{PD}(C)=a[\mathfrak p^{-1}(\pt)]$ for some $a>0$. By the discussion above, this shows that:
\begin{itemize}
\item If $M$ is the two-sphere, then the set $\mathfrak P_{\ZZ,C}(\ZZ_C(\Sigma))$ has two connected components since $[\mathfrak p_+^{-1}(\pt)]=[(\Psi_2^*\mathfrak p_+)^{-1}(\pt)]$ and $[\mathfrak p_-^{-1}(\pt)]=[(\Psi_2^*\mathfrak p_-)^{-1}(\pt)]$;
\item If $M$ has positive genus, then the set $\mathfrak P_{\ZZ,C}$ is connected.
\end{itemize} 
The proof of the proposition is completed.\qed

\section{Action of closed two-forms}\label{s:acfun}
\subsection{The action form}\label{ss:defact}
\begin{dfn}
If $\Omega$ is a closed two-form on $\Sigma$, an embedded one-periodic curve $\gamma:S^1\to \Sigma$ with $\dot\gamma\in\ker\Omega$ is said to be a {\bf closed characteristics} of $\Omega$. We write $\mathcal X(\Omega)$ for the set of closed characteristics of $\Omega$, up to orientation-preserving reparametrisations of $S^1$.
\end{dfn}
Fix a free-homotopy class ${\mathfrak h}\in [S^1,\Sigma]$ and let $\Lambda_{\mathfrak h}(\Sigma)$ be the space of one-periodic curves in $\Sigma$ with  class $\mathfrak h$. 
In what follows, given a pair $(C,\mathfrak h)$, we shall study a variational principle on $\Lambda_{\mathfrak h}(\Sigma)$, which detects, for all $\Omega\in\Xi^2_C(\Sigma)$, the elements of $\mathcal X(\Omega)$ belonging to $\Lambda_{\mathfrak h}(\Sigma)$. To this purpose, we define the action form $\mathfrak{a}=\mathfrak{a}(\Omega) \in\Omega^1(\Lambda_{\mathfrak h}(\Sigma))$  by
\[
\mathfrak{a}_\gamma(\xi):=\int_{S^1}\Omega(\xi(t),\dot\gamma(t))\,\di t,\qquad\forall\,\gamma\in\Lambda_{\mathfrak h}(\Sigma),\ \forall\, \xi\in \mathrm T_\gamma\Lambda_{\mathfrak h}(\Sigma).
\]
Any $C^1$-path $\{\gamma_r\}\subset\Lambda_{\mathfrak h}(\Sigma)$ can also be interpreted as a cylinder $\Gamma:[0,1]\times S^1\to \Sigma$ with $\Gamma(r,t)=\gamma_r(t)$, so that
\begin{equation}\label{action_form_cylinder}
\int_0^1\{\gamma_r\}^*\mathfrak a(\Omega)=\int_{[0,1]\times S^1}\Gamma^*\Omega.
\end{equation}
Furthermore, we have the following classical observation.
\begin{lem}\label{l:var}
The action form $\mathfrak a(\Omega)$ is closed. An embedded periodic curve $\gamma\in\Lambda_{\mathfrak h}(\Sigma)$ is a closed characteristic of $\Omega$ if and only if $\mathfrak{a}_\gamma(\Omega)=0$.\qed
\end{lem}

If $\Omega$ and $\Omega'$ are forms in $\Xi^2_C(\Sigma)$ with $\Omega'-\Omega=\di\alpha$, then there holds
\[
\mathfrak a(\Omega')-\mathfrak a(\Omega)=\di \mathcal B_\alpha,\qquad \mathcal B_\alpha(\gamma):=\int_{S^1}\gamma^*\alpha.
\]
Thus, it makes sense to define $\mathfrak a(C):=[\mathfrak{a}(\Omega)]\in H^1_{\op{dR}}(\Lambda_{\mathfrak h}(\Sigma))$. We have a homomorphism 
\begin{equation}\label{e:C=0_tau=0}
H^2_{\dR}(\Sigma)\to H^1_{\op{dR}}(\Lambda_{\mathfrak h}(\Sigma)),\qquad C\mapsto \mathfrak{a}(C).
\end{equation}

If $\mathfrak{a}(\Omega)$ admits a primitive functional, then the zeros of $\mathfrak{a}(\Omega)$ are critical points of the primitive. In the next lemma, we give a criterion ensuring that $\mathfrak a(\Omega)$ is exact on $\Lambda_{\mathfrak h}(\Sigma)$, in the case that there exists an oriented $S^1$-bundle with fibres in ${\mathfrak h}$. Below, we regard classes in $H^2_{\dR}(M)$ as real homomorphisms on $\pi_2(M)$ through the canonical map $\pi_2(M)\to H_2(M;\Z)$.

\begin{lem}\label{l:exact}
Let $C\in H^2_\dR(\Sigma)$ and $(\mathfrak p,c)\in\mathfrak Z_C(\Sigma)$ such that $\mathfrak p:\Sigma\to M$ has minus the real Euler class $e\in H^2_\dR(M)$ and the oriented $\mathfrak p$-fibres have class ${\mathfrak h}\in[S^1,\Sigma]$. There holds
\[
\mathfrak{a}(C)=0\;\ \mathrm{in}\ H^1_{\mathrm{dR}}(\Lambda_{\mathfrak h}(\Sigma))\qquad \Longleftrightarrow\qquad \ker e|_{\pi_2(M)}\subseteq\ker c|_{\pi_2(M)}.
\]
\end{lem}

\begin{proof}
Let $\om$ be any element in $\Xi^2_c(M)$ and define $\Omega:=\mathfrak p^*\om$. The cohomology class $\mathfrak{a}(C)$ is trivial if and only if its integral over any one-periodic curve $\Gamma:S^1\to \Lambda_{\mathfrak h}(\Sigma)$ vanishes. Choosing any oriented fibre $\gamma:S^1\to \Sigma$ such that $[\gamma]=\mathfrak h$, we may assume that $\Gamma(0)=\Gamma(1)=\gamma$, up to homotopy, since $\mathfrak{a}$ is closed and therefore the integral of $\mathfrak{a}$ depends only on the homotopy class of $\Gamma$. In view of \eqref{action_form_cylinder}, $\mathfrak{a}(C)\in H^1_{\dR}(\Lambda_{\mathfrak h}(\Sigma))$ is trivial if and only if, for any such $\Gamma$,
\[
\int_{[0,1]\times S^1}\Gamma^*\Omega=0.
\]
As $\mathfrak p\circ\gamma$ is a constant curve, we think of $\bar\Gamma:=\mathfrak p\circ\Gamma$ as a map from $S^2$ into $M$ with homotopy class $[\bar \Gamma]\in\pi_2(M)$. We have
\[
\int_{[0,1]\times S^1}\Gamma^*\Omega=\int_{S^2}\bar\Gamma^*\om=\langle c,[\bar\Gamma]\rangle.
\]
Hence, the lemma follows if we show that the map $\Gamma\mapsto[\bar\Gamma]$ is onto $\ker e|_{\pi_2(M)}$. To see that the image $[\bar\Gamma]$ is indeed in $\ker e|_{\pi_2(M)}$, we compute
\[
\langle e,[\bar\Gamma]\rangle=\int_{S^2}\bar\Gamma^*\kappa_\eta=\int_{[0,1]\times S^1}\Gamma^*(\di\eta)=\int_{S^1}\gamma^*\eta-\int_{S^1}\gamma^*\eta=0,
\]
where $\eta\in\mathcal K(\mathfrak p)$ is any connection for $\mathfrak p$. 

Then, we show that for any $\upsilon:S^2\to M$ with $\langle e,[\upsilon]\rangle=0$, there is $\Gamma:[0,1]\times S^1\to \Sigma$ with $\Gamma(0,\cdot)=\Gamma(1,\cdot)=\gamma$ such that $\bar\Gamma=\mathfrak p\circ\Gamma=\upsilon$. By the naturality of the Euler class, the restriction of $\mathfrak p$ over $\upsilon$ admits a global section $\Upsilon:S^2\to \Sigma$. It satisfies $\mathfrak p\circ\Upsilon=\upsilon$. Using the quotient map $[0,1]\times S^1\to S^2$, we lift $\Upsilon$ to a map $\Upsilon':[0,1]\times S^1\to \Sigma$. Up to modifying $\upsilon$ within its homotopy class, we can assume that $\Upsilon'(0,\cdot)=\Upsilon'(1,\cdot)=\gamma(0)$. Finally, we define
\[
\Gamma:[0,1]\times S^1\to \Sigma,\qquad \Gamma(s,t)=\mathfrak{u}(t,\Upsilon'(s,t)),
\]
where $\mathfrak{u}\in \mathfrak{U}(\mathfrak p)$. It follows that $\Gamma(0,\cdot)=\Gamma(1,\cdot)=\gamma$ and $\bar\Gamma=\bar{\Upsilon}'=\bar\Upsilon=\upsilon$, as required. 
\end{proof}
\begin{rmk}
The condition $\mathfrak a(C)=0$ in $H^1_\dR(\Lambda_{\mathfrak h}(\Sigma))$ depends only on $C$ and $\mathfrak h$. Therefore, the condition $\ker e|_{\pi_2(M)}\subseteq\ker c|_{\pi_2(M)}$ is also independent of the chosen pair $(\mathfrak p,c)\in\mathfrak Z_C(\Sigma)$ with the property that the $\mathfrak p$-fibres are in the class $\mathfrak h$. This last statement can also be seen directly combining Remark \ref{r:psiee}, equation \eqref{e:gysin}, and item (i) in Lemma \ref{l:nondegenerate}.
\end{rmk}

\subsection{The action functional on a covering space}\label{ss:defactfun}
In view of Lemma \ref{l:exact} we cannot expect the existence of a primitive functional of the action form $\mathfrak{a}$ in general. One standard way to resolve this problem is to find a primitive functional on a suitable covering space of $\Lambda_{\mathfrak h}(\Sigma)$. We define a natural covering space in \eqref{e:lambdatilde} below. However, it will have the small disadvantage that in some cases the primitive functional depends on the choice of a one-form $\alpha$ such that $\Omega=\Omega_0+\di\alpha$, where $\Omega_0$ is a reference two-form in the same cohomology class of $\Omega$. We overcome this nuisance through the notion of non-degeneracy, which we introduced in Definition \ref{d:nondegenerate}.
\medskip

As before, let $C$ be a class in $H^2_\dR(\Sigma)$ and $\mathfrak P^0(\Sigma)$ be a connected component of $\mathfrak P(\Sigma)$. We write $\mathfrak h\in[S^1,\Sigma]$ for the class of the oriented fibres of any bundle in $\mathfrak P^0(\Sigma)$. For the rest of this section, we work under the assumption that
\begin{equation}\label{e:nd}
\mathfrak Z_C^0(\Sigma)=\mathfrak P^{-1}(\mathfrak P_0(\Sigma))\cap\mathfrak Z_C(\Sigma)\ \ \text{is non-empty and non-degenerate},
\end{equation}
where $\mathfrak P:\mathfrak Z(\Sigma)\to\mathfrak P(\Sigma)$ is the standard projection. We fix a {\bf reference pair} $(\mathfrak p_0,c_0)\in\mathfrak Z_C^0(\Sigma)$ with $\mathfrak p_0:\Sigma\to M_0$ and denote by $e_0\in H^2_{\dR}(M_0)$ minus the real Euler class of $\mathfrak p_0$. From \eqref{e:jmathpsi}, we have an embedding
\begin{equation}\label{eq:fiber_loop}
\jmath_{\mathfrak p_0}:\Sigma\to\Lambda_{\mathfrak h}(\Sigma),
\end{equation}
where $\jmath_{\mathfrak p_0}(z)$ is a parametrisation of the oriented  $\mathfrak p_0$-fibre  passing through $z\in\Sigma$. By definition, there holds $\jmath_{\mathfrak p_0}(z)\in\mathcal X(\Omega_0)$. We consider the following covering space of $\Lambda_{\mathfrak h}(\Sigma)$:
\begin{equation}\label{e:lambdatilde}
\widetilde{\Lambda}_{\mathfrak h}(\Sigma):=\Big\{\{\gamma_r\}\ \Big|\ \gamma_0\in\jmath_{\mathfrak p_0}(\Sigma),\ \gamma_r\in \Lambda_{\mathfrak h}(\Sigma),\ \forall\, r\in[0,1] \Big\}\, \Big/\!\sim\,,
\end{equation}
where $\{\gamma_r^0\}\sim\{\gamma_r^1\}$, if there is a homotopy $\{\gamma_r^s\}_{s\in[0,1]}$ such that
\begin{equation}\label{e:admhom}
\gamma_0^s\in \jmath_{\mathfrak p_0}(\Sigma),\quad \gamma_1^0=\gamma_1^s=\gamma_{1}^1,\quad\forall\, s\in[0,1].
\end{equation}
We denote the elements of $\widetilde{\Lambda}_{\mathfrak h}(\Sigma)$ as $[\gamma_r]$ so that the covering map is given by 
\[
\widetilde{\Lambda}_{\mathfrak h}(\Sigma)\to{\Lambda}_{\mathfrak h}(\Sigma),\qquad [\gamma_r]\mapsto\gamma_1.
\]

We further choose some $\omega_0\in\Xi^2_{c_0}(M_0)$  and set as reference form 
\[
\Omega_0:=\mathfrak p_0^*\om_0\in\Xi^2_C(\Sigma).
\] 
Let $\Vol:\Omega^1(\Sigma)\to\R$ and $\Fvol:\Xi^2_C(\Sigma)\to\R$ be the volume functionals associated with $\Omega_0$. As observed in Lemma \ref{l:weakzollvol}, $\Fvol$ depends only on $(\mathfrak p_0,\om_0)$ but not on $\om_0$.

Let $\alpha\in\Omega^1(\Sigma)$ and recall the notation $\Omega_\alpha=\Omega_0+\di\alpha\in\Xi^2_C(\Sigma)$. We define the action
\begin{equation}\label{e:actiontilde1}
\widetilde{\mathcal A}_\alpha:\widetilde \Lambda_{\mathfrak h}(\Sigma)\to\R\,,\qquad\widetilde{\mathcal A}_\alpha\big([\gamma_r]\big):=\int_{S^1}\gamma_0^*\alpha+\int_{[0,1]\x S^1}\Gamma^*\Omega_\alpha,
\end{equation}
where, as before, $\Gamma:[0,1]\times S^1\to \Sigma$ is the cylinder associated with $\{\gamma_r\}$. 

The action is well-defined as one sees by integrating $0=\di\Omega_\alpha$ over a homotopy satisfying \eqref{e:admhom} and then applying Stokes' Theorem. Decomposing $\Omega_\alpha=\Omega_0+\di\alpha$ in the second integrand above and using Stokes' Theorem, we can rewrite $\widetilde{\mathcal A}_\alpha$ as
\begin{equation}\label{e:actiontilde2}
\widetilde{\mathcal A}_\alpha\big([\gamma_r]\big)= \int_{S^1}\gamma_1^*\alpha +\int_{D^2}\bar\Gamma^*\omega_0,
\end{equation}
where $\bar\Gamma=\mathfrak p_0\circ\Gamma:D^2\to M_0$. A straightforward computation shows that
\[
\di_{[\gamma_r]}\widetilde{\mathcal A}_\alpha\big([\xi_r]\big)=\mathfrak{a}_{\gamma_1}\big(\xi_1\big),\qquad\forall\,[\gamma_r]\in\widetilde{\Lambda}_{\mathfrak h}(\Sigma),\ \forall\,[\xi_r]\in \ta_{[\gamma_r]}\widetilde{\Lambda}_{\mathfrak h}(\Sigma),
\]
where $\mathfrak a=\mathfrak a(\Omega_\alpha)$. Hence, Lemma \ref{l:var} can be rephrased as follows.

\begin{cor}\label{c:crit}
Let $[\gamma_r]\in \widetilde\Lambda_{\mathfrak h}(\Sigma)$ with $\gamma_1\in \Lambda_{\mathfrak h}(\Sigma)$ embedded. Then $[\gamma_r]$ is a critical point of $\widetilde{\mathcal A}_\alpha$ if and only if $\gamma_1\in\mathcal X(\Omega_\alpha)$.\qed
\end{cor}
If $\alpha$ and $\alpha'$ are one-forms such that $\Omega_\alpha=\Omega=\Omega_{\alpha'}$, then $\tau:=\alpha'-\alpha$ is closed and we have
\begin{equation}\label{eq:action_diff}
\widetilde{\mathcal A}_{\alpha'}=\widetilde{\mathcal A}_\alpha+\langle [\tau],[\mathfrak p^{-1}(\pt)]\rangle.\vspace{8pt}
\end{equation}

\noindent\textbf{Case 1:~$e_0\neq0$.}
\medskip

\noindent By Lemma \ref{l:ecc}, we have $[\mathfrak p^{-1}(\pt)]=0$. The action functional depends only on the two-form $\Omega\in\Xi^2_C(\Sigma)$ due to \eqref{eq:action_diff}. Therefore, for any $\alpha\in\Omega^1(\Sigma)$ with $\Omega=\Omega_\alpha$, we can set
\begin{equation*}\label{def:mathcalaomega}
\widetilde{\mathcal A}_\Omega:=\widetilde{\mathcal A}_\alpha.
\end{equation*}
In this situation, the non-degeneracy \eqref{e:nd} of $\mathfrak Z_C^0(\Sigma)$ is not needed to associate the action functional with elements in $\Xi^2_C(\Sigma)$, as opposed to the next case.
\vspace{10pt}

\noindent\textbf{Case 2:~$e_0=0$.}\label{page}
\medskip

\noindent Here, the action functionals of $\alpha$ and $\alpha'$ might be different. Nevertheless, as $\mathfrak Z_C^0(\Sigma)$ is non-degenerate, we have $\langle c_0^n,[M_0]\rangle=\ev(\mathfrak p_0,c_0)\neq 0$. Thus, by \eqref{e:wd2} and \eqref{e:pd}, we can write
\[
\langle [\tau],[\mathfrak p^{-1}(\pt)]\rangle=\frac{\langle [\tau]\cup C^{n},[\Sigma]\rangle}{\langle c^n_0,[M_0]\rangle}=\frac{\Vol(\alpha')-\Vol(\alpha)}{\langle c_0^n,[M_0]\rangle},
\]
so that if $\alpha$ and $\alpha'$ have the same volume, they also have the same action by \eqref{eq:action_diff}. We set
\begin{equation}\label{eq:action_normalized_form}
\widetilde{\mathcal A}_\Omega:=\widetilde{\mathcal A}_\alpha
\end{equation}
for a normalised $\alpha\in\Omega^1(\Sigma)$, i.e.~$\Vol(\alpha)=0$, with $\Omega=\Omega_\alpha$.

\begin{rmk}\label{r:ch2}
This remark is parallel to Remark \ref{rmk:contact_hamiltonian}.
For contact forms and Hamiltonian systems, the action functional recovers the following well-known formulae.
\begin{itemize}
\item(Contact forms) Let $\Omega_0=0$ and $\alpha\in\Omega^{1}(\Sigma)$ be a (possibly contact) one-form. As $C=0$, we have that $e_0\neq0$ by Lemma \ref{l:ecc}.(ii). Thus, this is a special instance of Case 1. In fact, $\mathfrak a(C)=0$ due to  \eqref{e:C=0_tau=0}, and the function $\mathcal A_\alpha:\Lambda_{\mathfrak h}(\Sigma)\to\R$, given by
\[
\mathcal A_{\alpha}(\gamma)=\int_{S^1}\gamma^*\alpha,
\]
is the unique primitive of $\mathfrak a$ such that $\widetilde{\mathcal A}_{\di \alpha}([\gamma_r])=\mathcal A_{\alpha}(\gamma_1)$, for every $[\gamma_r]\in\widetilde{\Lambda}_{\mathfrak h}(\Sigma)$. 
\item(Hamiltonian systems) Let $\mathfrak p_0$ be trivial, namely $\Sigma=M_0\x S^1$, so that this is a special instance of Case 2. Assume $\Omega_0=\mathfrak p_0^*\om_0$, where $\om_0$ is a symplectic form on $M_0$, and $\alpha=H\di t$, where $t$ is the angular coordinate on $S^1$. If $\gamma_1(t)=(q_1(t),t)$, then we get the classical Hamiltonian action functional
\[
\widetilde{\mathcal A}_{H\di t}([\gamma_r]):=\int_{S^1} H(q_1(t),t)\,\di t +\int_{D^2}\bar\Gamma^*\om_0
\]
on the space of contractible curves with capping disc. Furthermore, the condition $\ker e_0|_{\pi_2(M_0)}\subseteq\ker c_0|_{\pi_2(M_0)}$ in Lemma \ref{l:exact} means that $\om_0$ is symplectically aspherical.
\end{itemize}
\end{rmk}
Next, we study the relation between the actions with respect to two different reference weakly Zoll pairs.
\begin{prp}\label{p:concatenationaction}
Let $(\mathfrak p_0',c_0')$ be another element in $\mathfrak Z_C^0(\Sigma)$. We write $\widetilde\Lambda_{\mathfrak h}'(\Sigma)$ for the associated covering space of $\Lambda_{\mathfrak h}(\Sigma)$ and $\widetilde{\Lambda}_{\mathfrak h}(\mathfrak p_0,\mathfrak p_0')\subset \widetilde\Lambda_{\mathfrak h}(\Sigma)$ for the set of elements $[\delta_r]$ such that $\delta_1\in\jmath_{\mathfrak p_0'}(\Sigma)$. We pick $\om_0'\in \Xi^2_{c_0'}(M_0')$ and set $\Omega_0':=(\mathfrak p_0')^*\om_0'$. We choose $\alpha'\in\Omega^1(\Sigma)$ such that
\[
\Omega_0'=\Omega_0+\di\alpha',\qquad \alpha'\ \text{is $\Omega_0$-normalised}.
\]
For any $\Omega\in\Xi^2_C(\Sigma)$, we take $\alpha\in\Omega^1(\Sigma)$ such that
\[
\Omega=\Omega_0+\di\alpha,\qquad \ \  \alpha\ \text{is $\Omega_0$-normalised}.
\]
Then, there holds $\Omega=\Omega_0'+\di(\alpha-\alpha')$ and the one-form $\alpha-\alpha'$ is $\Omega_0'$-normalised. Moreover, if we denote by $\widetilde{\mathcal A}'_\Omega$ the action of $\Omega$ with respect to $\Omega_0'$, then
\begin{equation}\label{e:concid}
\widetilde{\mathcal A}_\Omega\big([\{\delta_r\}\#\{\gamma_r\}]\big)=\widetilde{\mathcal A}_{\Omega_0'}\big([\delta_r]\big)+\widetilde{\mathcal A}_{\Omega}'\big([\gamma_r]\big),
\end{equation}  
for every $[\delta_r]\in\widetilde\Lambda_{\mathfrak h}(\mathfrak p_0,\mathfrak p_0')$ and $[\gamma_r]\in\widetilde\Lambda_{\mathfrak h}'(\Sigma)$. Here, the concatenation is made by choosing any representative $\{\gamma_r\}$ of $[\gamma_r]$ with $\gamma_0=\delta_1$.
\end{prp}
\begin{proof}
The one-form $\alpha-\alpha'$ is $\Omega'_0$-normalised thanks to Lemma \ref{l:volcomp}. Let $\Gamma,\Delta:[0,1]\times S^1\to\Sigma$ be the cylinders traced by the paths $\{\gamma_r\}$ and $\{\delta_r\}$, respectively. Let us show equation \eqref{e:concid} using \eqref{e:actiontilde1} and \eqref{e:actiontilde2}:
\begin{align*}
\widetilde{\mathcal A}_{\Omega_0'}\big([\delta_r]\big)+\widetilde{\mathcal A}_{\Omega}'\big([\gamma_r]\big)&=\int_{S^1}\delta_1^*\alpha'+\int_{[0,1]\times S^1}\Delta^*\Omega_0+\int_{S^1}\gamma_0^*(\alpha-\alpha')+\int_{[0,1]\times S^1}\Gamma^*\Omega\\
&=\int_{[0,1]\times S^1}\Delta^*\Omega_0+\int_{S^1}(\gamma_1)^*\alpha+\int_{[0,1]\times S^1}\Gamma^*\Omega_0+\int_{[0,1]\times S^1}\Gamma^*(\di\alpha)\\
&=\int_{[0,1]\times S^1}(\Delta\#\Gamma)^*\Omega_0+\int_{S^1}\gamma_0^*\alpha+\int_{S^1}\gamma_1^*\alpha-\int_{S^1}\gamma_0^*\alpha\\
&=\widetilde {\mathcal A}_\Omega\big([\{\delta_r\}\#\{\gamma_r\}]\big).\qedhere
\end{align*}
\end{proof}
\begin{cor}\label{c:indepomega0}
The action $\widetilde{\mathcal A}_{\Omega}$ depends only on the fixed reference pair $(\mathfrak p_0,c_0)\in\mathfrak Z_C^0(\Sigma)$, and not on the specific choice of $\om_0\in\Xi^2_{c_0}(M_0)$.
\end{cor}
\begin{proof}
Let $\om_0'$ be another element in $\Xi^2_{c_0}(M_0)$. By \eqref{e:concid} with $(\mathfrak p_0',c_0')=(\mathfrak p_0,c_0)$ and $[\delta_r]=[\gamma_0]$, it suffices to show that $\widetilde{\mathcal A}_{\mathfrak p_0^*\om_0'}([\gamma_0])=0$. To this purpose, let $\zeta\in\Omega^1(M_0)$ be such that $\om_0'-\om_0=\di\zeta$. By Lemma \ref{l:weakzollvol}.(ii), $\mathfrak p_0^*\zeta$ is $\mathfrak p_0^*\om_0$-normalised and we can use it to compute $\widetilde{\mathcal A}_{\mathfrak p_0^*\om_0'}$:
\[
\widetilde{\mathcal A}_{\mathfrak p_0^*\om_0'}([\gamma_0])=\int_{S^1}\gamma_0^*(\mathfrak p^*_0\zeta)=0.\qedhere
\]
\end{proof}\subsection{The action is invariant under pull-back and isotopies}\label{ss:acinvariant}
In this subsection, we prove two invariance results for the action. For the first one, we consider an additional connected oriented closed manifold \label{dfn:veeaction}$\Sigma^\veee$ of dimension $2n+1$. We suppose that there are a bundle $\mathfrak p_0^\veee:\Sigma^\veee\to M_0^\veee$ in $\mathfrak P^0(\Sigma^\veee)$ and a bundle map $\Pi:\Sigma^\veee\to \Sigma$ with $\mathfrak p_0^\veee=\Pi^*\mathfrak p_0$. Let $\mathfrak P^0(\Sigma^\veee)$ be the connected component of $\mathfrak P(\Sigma^\veee)$ containing $\mathfrak p_0^\veee$. Let $\pi_0:M_0^\veee\to M_0$ be the map fitting into the first commutative diagram in \eqref{eq:commutative_diagram2}. If we set  
\[
c_0^\veee:=\pi_0^*c_0\in H^2_\dR(M_0^\veee),\qquad C^\veee:=\Pi^*C\in H^2_\dR(\Sigma^\veee),
\]
then $(\mathfrak p_0^\veee,c_0^\veee)\in\mathfrak Z^0_{C^\veee}(\Sigma^\veee):=(\mathfrak P^\veee)^{-1}(\mathfrak P^0(\Sigma^\veee))\cap \mathfrak Z_{C^\veee}(\Sigma^\veee)$. In particular, $\mathfrak Z_{C^\veee}^0(\Sigma^\veee)$ is non-empty. We write $\mathfrak P_{C^\veee}^0:\mathfrak Z_{C^\veee}^0(\Sigma^\veee)\to\mathfrak P^0(\Sigma^\veee)$ for the projection and $\ev^\veee:\mathfrak Z_{C^\veee}^0(\Sigma^\veee)\to\R$ for the evaluation. Furthermore, we abbreviate 
\[
\om_0^\veee:=\pi_0^*\om_0\in\Xi^2_{c_0^\veee}(M_0^\veee),\qquad \Omega_0^\veee:=(\mathfrak p_0^\veee)^*\om_0^\veee=\Pi^*\Omega_0\in\Xi^2_{C^\veee}(\Sigma^\veee),
\] 
so that $\Omega_0^\veee$ is associated with the weakly Zoll pair $(\mathfrak p_0^\veee,c_0^\veee)$. We note that
\begin{equation}\label{e:fcae}
\bullet\quad \pi_0^*e_0=e_0^\veee,\qquad\qquad \bullet\quad \pi_0^*(Ae_0+c_0)=Ae^\veee_0+c_0^\veee,\quad\forall\, A\in \R,
\end{equation}
where $e_0^\veee$ is minus the real Euler class of $\mathfrak p_0^\veee$. If $\mathfrak h^\veee$ denotes the free-homotopy class of oriented fibres of elements in $\mathfrak P^0(\Sigma^\veee)$, we write $\widetilde{\mathcal A}^\veee_{\alpha^\veee}:\widetilde \Lambda_{\mathfrak h^\veee}(\Sigma^\veee)\to\R$ for the action of $\alpha^\veee\in\Omega^1(\Sigma^\veee)$ with respect to $\Omega_0^\veee$. The bundle map $\Pi$ yields a map between the spaces 
\[
\Pi_*:\Lambda_{\mathfrak h^\veee}(\Sigma^\veee)\to\Lambda_{\mathfrak h}(\Sigma),\qquad \Pi_*\gamma^\veee:=\Pi\circ\gamma^\veee
\]
and between their covering spaces
\[
\widetilde \Pi_*:\widetilde{\Lambda}_{\mathfrak h^\veee}(\Sigma^\veee)\to\widetilde{\Lambda}_{\mathfrak h}(\Sigma),\qquad \widetilde \Pi_*[\gamma_r^\veee]:=[\Pi\circ \gamma_r^\veee].
\] 
\begin{prp}\label{p:pbaction}
The set $\mathfrak Z_{C^\veee}^0(\Sigma^\veee)$ is non-degenerate if and only if $\deg\pi_0\neq 0$. In this case, $\widetilde{\mathcal A}^\veee_{\Omega^\veee}:\widetilde \Lambda_{\mathfrak h^\veee}(\Sigma^\veee)\to\R$ is well-defined for all $\Omega^\veee\in\Xi^2_{C^\veee}(\Sigma^\veee)$ and there holds
\[
\widetilde{\mathcal A}^\veee_{\Pi^*\Omega}=\widetilde{\mathcal A}_{\Omega}\circ \widetilde \Pi_*,\qquad \forall\,\Omega\in\Xi^2_C(\Sigma).
\]
\end{prp}
\begin{proof}
The map $(\mathfrak P_C^0)^{-1}(\mathfrak p_0)\to(\mathfrak P_{C^\veee}^0)^{-1}(\mathfrak p_0^\veee)$ given by $(\mathfrak p_0,c)\mapsto(\mathfrak p_0^\veee,\pi_0^*c)$ is well-defined. By \eqref{e:gysin} and \eqref{e:fcae}, it is also surjective. Since $\mathfrak Z_{C}^0(\Sigma)$ is non-degenerate and there holds \begin{equation}\label{e:c'm0'}
\ev^\veee(\mathfrak p_0^\veee,\pi_0^*c)=\deg \pi_0\cdot\ev(\mathfrak p_0,c),\qquad \forall\,(\mathfrak p_0,c)\in (\mathfrak P_C^0)^{-1}(\mathfrak p_0),
\end{equation}
we see that $\mathfrak Z_{C^\veee}^0(\Sigma^\veee)$ is non-degenerate exactly when $\deg\pi_0$ is non-zero.

Let $\Omega=\Omega_0+\di \alpha\in\Xi^2_C(\Sigma)$, where $\alpha\in\Omega^1(\Sigma)$ is $\Omega_0$-normalised. Then, $\Pi^*\Omega=\Omega_0^\veee+\di(\Pi^*\alpha)$ and $\Pi^*\alpha$ is $\Omega_0^\veee$-normalised by Proposition \ref{p:volinvcov}. For all $[\gamma_r^\veee]\in\widetilde\Lambda_{\mathfrak h^\veee}(\Sigma)$, we have
\begin{align*}
\widetilde{\mathcal A}^\veee_{\Pi^*\Omega}\big([\gamma^\veee_r]\big)=\int_{S^1}(\gamma_0^\veee)^*\Pi^*\alpha+\int_{[0,1]\times S^1}(\Gamma^\veee)^*\Pi^*\Omega&=\int_{S^1}(\Pi\circ\gamma_0^\veee)^*\alpha+\int_{[0,1]\times S^1}(\Pi\circ\Gamma^\veee)^*\Omega\\
&=\widetilde{\mathcal A}_{\Omega}\big(\widetilde \Pi_*[\gamma^\veee_r]\big).\qedhere
\end{align*}
\end{proof}
For the second invariance result, we consider a diffeomorphism $\Psi:\Sigma\to \Sigma$ isotopic to $\id_\Sigma$. We define
\[
\Psi^*:\Lambda_{\mathfrak h}(\Sigma)\to\Lambda_{\mathfrak h}(\Sigma),\qquad \Psi^*\gamma:=(\Psi^{-1})_*\gamma=\Psi^{-1}\circ\gamma.
\] 
Let  $\{\Psi_r\}$ be an isotopy of diffeomorphisms of $\Sigma$ with $\Psi_0=\id_\Sigma$ and $\Psi_1=\Psi$. We denote by $[\Psi_r]$ the homotopy class with fixed end-points of $\{\Psi_r\}$ in the space of isotopies. Given such a homotopy class, we define
\[
[\Psi_r]^*:\widetilde{\Lambda}_{\mathfrak h}(\Sigma)\to\widetilde{\Lambda}_{\mathfrak h}(\Sigma),\qquad [\Psi_r]^*[\gamma_r]:=\big[\Psi_r^{-1}\circ\gamma_r\big].
\]

\begin{prp}\label{prp:action_inv_isotopy}
For every $\Omega\in\Xi^2_C(\Sigma)$ and every homotopy class $[\Psi_r]$ as above, there holds
\[
\widetilde{\mathcal A}_{\Psi^*_1\Omega}\circ [\Psi_r]^*=\widetilde{\mathcal A}_\Omega.
\]
\end{prp}

\begin{proof}
We observe preliminarily that if $X_r$ and $Y_r$ are the time-dependent vector fields generating $\Psi_r$ and $\Psi_r^{-1}$, we have the relation 
\[
-X_r=\di \Psi_r(Y_r)\circ \Psi_r^{-1}.
\]
Let $\alpha$ be a normalised one-form such that $\Omega=\Omega_0+\di\alpha$ and let $\{\theta_r\}$ be the path of normalised one-forms introduced in Section \ref{s:voliso} such that 
\[
\Psi_r^*\Omega=\Omega_0+\di(\theta_r+\Psi_r^*\alpha),\qquad \dot\theta_r=\Psi_r^*(\iota_{X_r}\Omega_0),\quad\theta_0=0.
\]
For $s\in[0,1]$, we define the truncation $\{\Psi_r^s:=\Psi_{rs}\}$ and we write $\mathfrak a^s:=\mathfrak a(\Psi_s^*\Omega)$. For every $[\gamma_r]\in\widetilde{\Lambda}_{\mathfrak h}(\Sigma)$, we compute
\[
\begin{split}
\frac{\di}{\di s}\Big(\widetilde{\mathcal{A}}_{\Psi_s^*\Omega}\big([\Psi_r^s]^*[\gamma_r]\big)\Big)&=\Big(\frac{\di}{\di s}\widetilde{\mathcal{A}}_{\Psi_s^*\Omega}\Big)\big([\Psi_r^s]^*[\gamma_r]\big)+\di_{[\Psi_{rs}^*\gamma_r]}\widetilde{\mathcal{A}}_{\Psi_s^*\Omega}\cdot\frac{\di}{\di s}[\Psi_{rs}^*\gamma_r]
\\
&=\int_{S^1}(\Psi_s^*\gamma_1)^*\Big(\tfrac{\di}{\di s}(\theta_s+\Psi_s^*\alpha)\Big)\,\di t+\mathfrak{a}^s_{\Psi_s^*\gamma_1}\cdot\frac{\di}{\di s}\Psi_s^*\gamma_1\\
&=\int_{S^1}\gamma_1^*\Big(\iota_{X_s}\Omega+\di\big(\alpha(X_s)\big)\Big)\,\di t+\int_{S^1}(\Psi_s^*\Omega)\Big(\tfrac{\p}{\p s}\Psi_s^*\gamma_1,\tfrac{\p}{\p t}\Psi_s^*\gamma_1\Big)\di t\\
&=\int_{S^1}\gamma_1^*(\iota_{X_s}\Omega)\,\di t+\int_{S^1}\Omega\Big(\di \Psi_s\cdot\tfrac{\p}{\p s}\big(\Psi_s^{-1}\circ\gamma_1\big),\di \Psi_s\cdot\tfrac{\p}{\p t}\big(\Psi_s^{-1}\circ\gamma_1\big)\Big)\di t\\
&=\int_{S^1}\Omega(X_s\circ\gamma_1,\dot\gamma_1)\,\di t+\int_{S^1}\Omega(-X_s\circ\gamma_1,\dot\gamma_1)\,\di t\\
&=0,
\end{split}
\]
where in the second equality we used \eqref{e:actiontilde2} to compute $\tfrac{\di}{\di s}\widetilde{\mathcal A}_{\Psi_s^*\Omega}$. Since $\widetilde{\mathcal A}_{\Psi_0^*\Omega}\circ [\Psi_0]^*=\widetilde{\mathcal A}_{\Omega}$, the proof is completed.
\end{proof}

\subsection{The action of weakly Zoll pairs} \label{ss:actionweak}

Based on the action functional we have studied, we define the action functional on the space of weakly Zoll pairs:
\begin{equation}\label{e:defaczoll}
\mathcal A:\mathfrak Z_C^0(\Sigma)\to \R,\qquad\mathcal A(\mathfrak p,c):=\widetilde{\mathcal A}_{\mathfrak p^*\om}\big([\jmath_{\mathfrak p_r}(z_0)]\big)
\end{equation}
where $z_0$ is a point in $\Sigma$, $\om\in\Xi^2_c(M)$, and $\{\mathfrak p_r\}$ is a path in $\mathfrak P^0(\Sigma)$ starting at the reference bundle $\mathfrak p_0$ and ending at $\mathfrak p_1=\mathfrak p$. It will turn out that this action is well-defined without the need to pass to a covering space, even when the condition in Lemma \ref{l:exact} is not met. This fact is a striking consequence of the non-degeneracy \eqref{e:nd} of $\mathfrak Z_C^0(\Sigma)$. A key role will be played by the following polynomial.
\begin{dfn}\label{d:polp}
Let $Q:\R\to\R$ be the auxiliary polynomial
\[
Q(A):=\ev(\mathfrak p_0,Ae_0+c_0)=\langle(Ae_0+c_0)^n,[M_0]\rangle,\qquad\forall\,A\in\R.
\]
The \textbf{Zoll polynomial} $P:\R\to\R$ of the pair $(\mathfrak p_0,c_0)$ is given through
\[
P(0)=0,\qquad \frac{\di P}{\di A} = Q.
\]
\end{dfn}
\begin{rmk}
The Zoll polynomial is non-constant by Corollary \ref{c:weaklyfour}.(ii) since we have assumed that $\mathfrak Z_C^0(\Sigma)$ is non-degenerate. Furthermore, we have the explicit formula
\begin{equation*}
P(A)=\int_0^AQ(A')\,\di A'=\sum_{j=0}^{n}\frac{1}{j+1}\binom{n}{j}\langle  e^j_0\cup c^{n-j}_0,[M_0]\rangle A^{j+1},
\end{equation*}
which reduces to
\begin{equation}\label{eq:P_dim3}
P(A)=\langle e_0,[M_0]\rangle \frac{A^2}{2}+\langle c_0,[M_0]\rangle A,\quad \text{for}\ \ n=1.
\end{equation}
\end{rmk}
\begin{rmk}\label{rmk:P_simple_form}
This remark is parallel to Remark \ref{rmk:contact_hamiltonian} and \ref{r:ch2}. The Zoll polynomial has a simple form when any of $c_0$ and $e_0$ vanishes.
\begin{itemize}
\item Let us assume that $C=0$ and take $c_0=0$, which is relevant to the study of Zoll contact forms. We have
\[
P(A)=\langle e_0^{n},[M_0]\rangle \frac{A^{n+1}}{n+1}.
\]
\item Let us assume that $e_0=0$, which is relevant to the study of Hamiltonian systems on $M$ (in which case $\mathfrak p_0$ is trivial and $\om_0$ is a symplectic form). We have
\[
P(A)=\langle c_0^{n},[M_0]\rangle A.
\]
\end{itemize}
\end{rmk}

\begin{thm}\label{thm:mainaction}
The functional $\mathcal A:\mathfrak Z_C^0(\Sigma)\to \R$ does not depend on any choice involved.
Moreover, for any $(\mathfrak p,c)\in \mathfrak Z_C^0(\Sigma)$, there holds
\begin{equation}\label{eq:vanishing_action}
{\mathcal A}(\mathfrak p,c)=0,\quad \text{if}\ \ e_0=0\ \text{ or } \ (\mathfrak p,c)=(\mathfrak p_0,c_0),
\end{equation}
and 
\begin{equation}\label{eq:zoll_identity}
\Fvol(\mathfrak p,c)=P\big(\mathcal A(\mathfrak p,c)\big).
\end{equation}
If $\Psi:\Sigma\to\Sigma$ is isotopic to $\id_\Sigma$ satisfying $\Psi^*\mathfrak p=\mathfrak p_0$ (which exists by Lemma \ref{l:triv}) and $\psi:M_0\to M$ is its quotient map, then 
\begin{equation}\label{eq:pullback_c1}
\psi^*c={\mathcal A}(\mathfrak p,c)e_0+c_0.
\end{equation}
\end{thm}
\begin{proof} 
To ease the notation, we write $(\mathfrak p_1,c_1)$ instead of $(\mathfrak p,c)$ to denote an arbitrary element of $\mathfrak Z_C^0(\Sigma)$ with $\mathfrak p_1:\Sigma\to M_1$ and $c_1\in H^2_\dR(M_1)$. Moreover, $\{\mathfrak p_r\}$ will indicate any path connecting the reference bundle $\mathfrak p_0$ with $\mathfrak p_1$. We divide the proof in five steps.
\bigskip

\noindent\underline{Step 1}. \textit{For any path $\{\mathfrak p_r\}$, the  action value $\widetilde{\mathcal A}_{\mathfrak p_1^*\om_1}\big([\jmath_{\mathfrak p_r}(z_0)]\big)$ does not depend on the choice of $z_0\in\Sigma$ and $\om_1\in\Xi^2_{c_1}(M_1)$.}
\medskip

Let $z_1$ be another point in $\Sigma$, and let $\{z_s\}$ be a path between $z_0$ and $z_1$ with $s\in[0,1]$. We claim that $\widetilde {\mathcal A}_{\mathfrak p_1^*\om_1}\big([\jmath_{\mathfrak p_r}(z_s)]\big)$ does not depend on $s$. Indeed, since $\jmath_{\mathfrak p_1}(z_s)\in\mathcal X({\mathfrak p_1^*\om_1})$ for all $s\in[0,1]$, we see that $\di_{[\jmath_{\mathfrak p_r}(z_s)]}\widetilde{\mathcal A}_{\mathfrak p_1^*\om_1}=0
$, by Corollary \ref{c:crit}, which in turn implies
\[
\frac{\di}{\di s}\widetilde {\mathcal A}_{\mathfrak p_1^*\om_1}\big([\jmath_{\mathfrak p_r}(z_s)]\big)=\di_{[\jmath_{\mathfrak p_r}(z_s)]} \widetilde {\mathcal A}_{\mathfrak p_1^*\om_1}\cdot\frac{\di}{\di s}[\jmath_{\mathfrak p_r}(z_s)]=0.
\]
This shows the independence of the action value from $z_0$.

Next, we take another $\om_1'\in\Xi_{c_1}^2(M_1)$ such that $\om_1'-\om_1=\di\zeta$ for some one-form $\zeta$ on $M_1$. By Lemma \ref{l:weakzollvol}.(ii), $\mathfrak p_1^*\zeta$ is normalised with respect to $\mathfrak p_1^*\om_1$. Applying equation \eqref{e:concid} with $\Omega=\mathfrak p_1^*\om_1'$, $\Omega_0'=\mathfrak p_1^*\om_1$, $\{\delta_r\}=\{\jmath_{\mathfrak p_r}(z_0)\}$ and $\{\gamma_r\}=\{\gamma_0\}$ a constant path, we find
\[
\widetilde{\mathcal A}_{\mathfrak p_1^*\om_1'}\big([\jmath_{\mathfrak p_r}(z_0)]\big)-\widetilde{\mathcal A}_{\mathfrak p_1^*\om_1}\big([\jmath_{\mathfrak p_r}(z_0)]\big)=\widetilde{\mathcal A}_{\mathfrak p_1^*\om_1'}'\big([\gamma_0]\big)=\int_{S^1}\gamma_0^*(\mathfrak p_1^*\zeta)=0.
\] 
Hence, the action depends on the cohomology class $c_1$, not on the representative.
\bigskip

\noindent\underline{Step 2}. \textit{Let $\{\mathfrak p_r\}$ be a path and $\{\Psi_r\}$ an isotopy of diffeomorphisms given by Lemma \ref{l:triv}.(ii) such that $\Psi_r^*\mathfrak p_r=\mathfrak p_0$. If $\psi_1:M_0\to M_1$ is the quotient map of $\Psi_1$, then 
\begin{equation*}
(i)\quad \psi_1^*c_1=\widetilde{\mathcal A}_{\mathfrak p_1^*\om_1}\big([\jmath_{\mathfrak p_r}(z_0)]\big)e_0+c_0,\qquad (ii)\ \ \widetilde{\mathcal A}_{\mathfrak p_1^*\om_1}\big([\jmath_{\mathfrak p_r}(z_0)]\big)=0,\quad \text{if}\ \ e_0=0.
\end{equation*}}
Setting $\gamma_1:=\jmath_{\mathfrak p_0}\circ \Psi_1^{-1}(z_0)$ and using Proposition \ref{prp:action_inv_isotopy}, we compute
\begin{equation}\label{eq:isotopy_action}
\begin{split}
\widetilde{\mathcal A}_{\mathfrak p_1^*\om_1}\big([\jmath_{\mathfrak p_r}(z_0)]\big)=\widetilde{\mathcal A}_{\mathfrak p_1^*\om_1}\big([\Psi_r^{-1}]^*[\jmath_{\mathfrak p_0}\circ \Psi_r^{-1}(z_0)]\big)&=\widetilde{\mathcal A}_{\Psi_1^*\mathfrak p_1^*\om_1}\big([\jmath_{\mathfrak p_0}\circ \Psi_r^{-1}(z_0)]\big)\\
&=\widetilde{\mathcal A}_{\Psi_1^*\mathfrak p_1^*\om_1}\big([\gamma_1]\big)
\end{split}
\end{equation}
where the last equality follows from $[\jmath_{\mathfrak p_0}\circ \Psi_r^{-1}(z_0)]=[\gamma_1]$ (see \eqref{e:lambdatilde}).  We observe that $\Psi_1^*\mathfrak p_1^*\om_1=\mathfrak p_0^*\psi_1^*\om_1$ and thus $(\mathfrak p_0,\psi_1^*c_1)\in \mathfrak Z_C^0(\Sigma)$. By items (iii) and (iv) in Corollary \ref{c:weaklyfour}, there exists $A_1\in\R$ such that 
\begin{equation}\label{eq:psi_1_pullbak}
\psi_1^*c_1=A_1e_0+c_0.
\end{equation}
If $e_0\neq0$, $A_1$ is uniquely defined by this property. If $e_0=0$, we simply impose $A_1:=0$. Equation \eqref{eq:psi_1_pullbak} implies that there are $\eta\in\mathcal K(\mathfrak p_0)$ and $\zeta\in\Omega^1(M_0)$ such that 
\begin{equation*}
\Psi_1^*\mathfrak p_1^*\om_1=\mathfrak p_0^*\om_0+\di(A_1\eta+\mathfrak p^*_0\zeta),\qquad \mathfrak p_0^*\kappa=\di\eta
\end{equation*}
for some $\kappa\in\Xi^2_{e_0}(M_0)$. Since $A_1=0$, if $e_0=0$, the one-form $A_1\eta+\mathfrak p_0^*\zeta$ is $\mathfrak p_0^*\om_0$-normalised by Lemma \ref{l:weakzollvol}.(ii) and we conclude with \eqref{eq:isotopy_action} that
\[
\widetilde{\mathcal A}_{\mathfrak p_1^*\om_1}\big([\jmath_{\mathfrak p_r}(z_0)]\big)=\widetilde{\mathcal A}_{\Psi_1^*\mathfrak p_1^*\om_1}\big([\gamma_1]\big)=\int_{S^1}\gamma_1^*(A_1\eta+\mathfrak p^*_0\zeta)=A_1.
\]
This proves both items in Step 2.
\bigskip

\noindent\underline{Step 3}. \textit{For any path $\{\mathfrak p_r\}$, there holds $\Fvol(\mathfrak p_1,c_1)=P\big(\widetilde{\mathcal A}_{\mathfrak p_1^*\om_1}\big([\jmath_{\mathfrak p_r}(z_0)]\big)\big)$.}
\medskip

From item (iii) in Lemma \ref{l:nondegenerate} with $A=0$, we have
\[
\Fvol(\mathfrak p_1,c_1)=\Fvol(\mathfrak p_0,\psi_1^*c_1)=\Vol(A_1\eta+\mathfrak p_0^*\zeta),
\]
where $A_1=\widetilde{\mathcal A}_{\mathfrak p_1^*\om_1}\big([\jmath_{\mathfrak p_r}(z_0)]\big)$, $\eta$ and $\zeta$ are from the previous step. 
Using the definition of the volume \eqref{e:defvol} and Fubini's Theorem, we compute
\begin{align*}
\Vol(A_1\eta+\mathfrak p_0^*\zeta)&=\int_0^1\Big(\int_\Sigma (A_1\eta+\mathfrak p_0^*\zeta)\wedge\mathfrak p_0^*\big(rA_1 \kappa+\om_0\big)^n\Big)\di r\\
&=\int_0^1\Big(\int_{M_0} \big((\mathfrak p_0)_*(A_1\eta+\mathfrak p_0^*\zeta)\big)\wedge\big(rA_1 \kappa+\om_0\big)^n\Big)\di r\\
&=\int_0^1A_1\big\langle (rA_1e_0+c_0)^n,[M_0]\big\rangle\di r\\
&=\int_0^1A_1Q(rA_1)\di r\\
&=P(A_1)-P(0)\\
&=P(A_1).
\end{align*}

\noindent\underline{Step 4}. \textit{If $(\mathfrak p_1,c_1)=(\mathfrak p_0,c_0)$, then, for any path $\{\mathfrak p_r\}$, we have $\widetilde{\mathcal A}_{\mathfrak p_1^*\om_1}\big([\jmath_{\mathfrak p_r}(z_0)]\big)=0$.}
\medskip

Let $\psi_1$ and $A_1=\widetilde{\mathcal A}_{\mathfrak p_1^*\om_1}\big([\jmath_{\mathfrak p_r}(z_0)]\big)$ as in Step 2. Since $\psi_1^*e_0=e_0$, applying $\psi_1$ iteratively to item (i) in Step 2, we obtain 
\[
(\psi^m_1)^*c_0=mA_1e_0+c_0,\qquad \forall\, m\in\Z.
\]
From item (ii) in Lemma \ref{l:nondegenerate} with $A=0$, we deduce 
\begin{align*}
Q(0)=\ev (\mathfrak p_0,c_0)=\ev (\mathfrak p_0,(\psi^m_1)^*c_0)= \ev(\mathfrak p_0,mA_1e_0+c_0)=Q(mA_1),\qquad\forall\, m\in\Z.
\end{align*}
We assume by contradiction that $A_1\neq0$. Since $Q$ is a polynomial, the above identity implies that $Q(A)\equiv Q_0$, where $Q_0\in \R$, so that $P(A)=Q_0A$ for all $A\in\R$. Since $\mathfrak Z_C^0(\Sigma)$ is non-degenerate, the coefficient $Q_0$ is non-zero. This contradicts Step 3:
\[
0=\Fvol(\mathfrak p_0,c_0)=P(A_1)=Q_0A_1.
\]

\noindent\underline{Step 5}. \textit{End of the proof.\\[-2ex]}

To show that the functional $\mathcal A:\mathfrak Z^0_C(\Sigma)\to\R$ is well-defined, it remains to see that the action value $\widetilde{\mathcal A}_{\mathfrak p_1^*\om_1}\big([\jmath_{\mathfrak p_r}(z_0)]\big)$ does not depend on the choice of path $\{\mathfrak p_r\}$. Indeed, if $\{\mathfrak p'_r\}$ is another path with $\mathfrak p'_0=\mathfrak p_0$ and $\mathfrak p_1'=\mathfrak p_1$, the concatenation $\{\mathfrak p'_r\}\#\{\mathfrak p_{1-r}\}$ forms a loop based at $\mathfrak p_0$. Applying Step 4 and the definition of the action \eqref{e:actiontilde2}, we conclude
\[
\widetilde{\mathcal A}_{\mathfrak p_1^*\om_1}\big([\jmath_{\mathfrak p'_r}(z_0)]\big)-\widetilde{\mathcal A}_{\mathfrak p_1^*\om_1}\big([\jmath_{\mathfrak p_r}(z_0)]\big)=\widetilde{\mathcal A}_{\mathfrak p_0^*\om_0}\big([\jmath_{\mathfrak p'_r\#\mathfrak p_{1-r}}(z_0)]\big)=0.
\]

Finally, the identities claimed in the statement of the theorem follow directly from what we have proven. Equation \eqref{eq:vanishing_action} is a consequence of item (ii) in Step 2 and of Step 4. Equation \eqref{eq:zoll_identity} follows from Step 3. Item (i) in Step 2 implies \eqref{eq:pullback_c1}.
\end{proof}

\begin{rmk}
If $e_0=0$, we could also have considered the space of weakly Zoll \textit{one}-forms $\bar{\mathfrak Z}_C^0(\Sigma)$, whose elements are pairs $(\mathfrak p,\alpha)$ with $(\mathfrak p,[\om_\alpha])\in \mathfrak Z_C^0(\Sigma)$, where $\mathfrak p^*\om_\alpha=\Omega_0+\di\alpha$, and $\alpha$ {\it not} necessarily normalised. We have an action functional 
\[
\bar {\mathcal A}:\bar{\mathfrak Z}_C^0(\Sigma)\to\R,\qquad \bar {\mathcal A}(\mathfrak p,\alpha):=\widetilde{\mathcal A}_{\alpha}([\jmath_{\mathfrak p_r}(z_0)]).
\] 
Since $\mathfrak Z_C^0(\Sigma)$ is non-degenerate, from Remark \ref{rmk:P_simple_form} and Theorem \ref{thm:mainaction}, we deduce
\begin{equation*}
\Vol(\alpha)=\langle c_0^{n},[M_0]\rangle \cdot \bar {\mathcal A}(\mathfrak p,\alpha),\qquad \forall\, (\mathfrak p,\alpha)\in \bar{\mathfrak Z}_C^0(\Sigma).
\end{equation*}
\end{rmk}
\begin{rmk}
Let $\Omega\in\ZZ_C(\Sigma)$ be a Zoll odd-symplectic form such that the associated bundle $\mathfrak p_\Omega$ belongs to $\mathfrak P^0(\Sigma)$. If we define
\[
\mathcal A(\Omega):=\mathcal A\big(\mathfrak p_\Omega,[\om_\Omega]\big),\qquad \Omega=\mathfrak p_\Omega^*\om_\Omega,
\]
then Theorem \ref{thm:mainaction} translates into
\[
\Fvol(\Omega)=P\big(\mathcal A(\Omega)\big),\qquad \forall\,\Omega\in\ZZ_C(\Sigma).
\]
\end{rmk}
As we did for the volume and the action, we study how the Zoll polynomial behaves under pull-back and change of pair $(\mathfrak p_0,c_0)$ within $\mathfrak Z_C^0(\Sigma)$.
\begin{prp}\label{p:pbinvact}
Let $\mathfrak p_0^\veee:\Sigma^\veee\to M_0^\veee$ be an oriented $S^1$-bundle with maps $\Pi:\Sigma^\veee\to\Sigma$ and $\pi_0:M_0^\veee\to M_0$ as described in \eqref{eq:commutative_diagram2}. If $P^\veee$ is the Zoll polynomial of the weakly Zoll pair $(\Pi^*\mathfrak p_0,\pi_0^*c_0)=(\mathfrak p_0^\veee,c_0^\veee)$, there holds
\begin{equation*}
P^\veee=(\deg \pi_0)\cdot P.
\end{equation*} 
\end{prp}
\begin{proof}
From equations \eqref{e:fcae} and \eqref{e:c'm0'}, we see that $Q^\veee=(\deg \pi_0)\cdot Q$, where $Q^\veee$ is the derivative of $P^\veee$. Since $P(0)=0=P^\veee(0)$, the statement follows.
\end{proof}

\begin{prp}\label{p:changepol}
Let $(\mathfrak p_0',c_0')$ be a pair in $\mathfrak Z_C^0(\Sigma)$ with $\mathfrak p_0':\Sigma\to M_0'$. There holds
\[
\langle (c_0')^n,[M_0']\rangle=\ev(\mathfrak p_0',c_0')=\frac{\di P}{\di A}\big(\mathcal A(\mathfrak p_0',c_0')\big).
\]
If $P'$ is the Zoll polynomial associated with $(\mathfrak p_0',c_0')$, then
\begin{equation*}
P'(A)=P\big(A+\mathcal A(\mathfrak p_0',c_0')\big)-P\big(\mathcal A(\mathfrak p_0',c_0')\big),\qquad\forall\,A\in\R.
\end{equation*}
\end{prp}
\begin{proof}
If $Q'$ denotes the derivative of $P'$, by Lemma \ref{l:nondegenerate}.(ii) and \eqref{eq:pullback_c1}, we have 
\[
Q'(A)=Q\big(A+\mathcal A(\mathfrak p_0',c_0')\big).
\] 
Setting $A=0$, we get the first part of the statement. For the second part, we integrate the above identity and use the normalization $P'(0)=0$:
\[
P'(A)=\int_0^AQ'(A')\di A'=\int_0^AQ\big(A'+\mathcal A(\mathfrak p_0',c_0')\big)\di A'=P\big(A+\mathcal A(\mathfrak p_0',c_0')\big)-P\big(\mathcal A(\mathfrak p_0',c_0')\big).\qedhere
\]
\end{proof}

Combining the result above with the transformation rules for the action and the volume under change of reference weakly Zoll pair, we conclude the following invariance property.

\begin{cor}\label{c:changeofbase}
Let $P'$, $\widetilde {\mathcal A}$, $\Fvol'$ be the Zoll polynomial, the action and the volume associated with another reference pair $(\mathfrak p_0',c_0')\in \mathfrak Z_C^0(\Sigma)$. For every $\Omega\in\Xi^2_C(\Sigma)$, there holds
\begin{equation*}
P'\circ\widetilde {\mathcal A}_\Omega'-\Fvol'(\Omega)=P\circ\widetilde {\mathcal A}_\Omega-\Fvol(\Omega)\qquad\text{on }\ \widetilde{\Lambda}_{\mathfrak h}(\Sigma).
\end{equation*} 
\end{cor}
\begin{proof}
We preliminarily observe that
\[
\widetilde {\mathcal A}_\Omega=\mathcal A(\mathfrak p_0',c_0')+\widetilde {\mathcal A}_\Omega'
\]
thanks to Proposition \ref{p:concatenationaction}. We also recall that Lemma \ref{l:volcomp} implies that
\[
\Fvol(\Omega)=\Fvol(\mathfrak p_0',c_0')+\Fvol'(\Omega).
\]
Using these identities together with Proposition \ref{p:changepol} and Theorem \ref{thm:mainaction}, we readily compute
\begin{align*}
P(\widetilde {\mathcal A}_\Omega)-\Fvol(\Omega)&=P(\widetilde {\mathcal A}_\Omega)-\Fvol(\mathfrak p_0',c_0')-\Fvol'(\Omega)\\
&=P\big(\widetilde {\mathcal A}_\Omega'+\mathcal A(\mathfrak p_0',c_0')\big)-\Fvol(\mathfrak p_0',c_0')-\Fvol'(\Omega)\\
&=P'(\widetilde {\mathcal A}_\Omega')+P\big(\mathcal A(\mathfrak p_0',c_0')\big)-\Fvol(\mathfrak p_0',c_0')-\Fvol'(\Omega)\\
&=P'(\widetilde {\mathcal A}_\Omega')-\Fvol'(\Omega).\qedhere
\end{align*}
\end{proof}
\section{A conjectural local systolic-diastolic inequality}\label{sec:conjsys}
The following is our setting to study a systolic-diastolic inequality for odd-symplectic forms. We assume that there exists a Zoll odd-symplectic form $\Omega_*\in\ZZ_C(\Sigma)$, for some $C\in H^2_\dR(\Sigma)$.  We denote the bundle associated with $\Omega_*$ by
\[
\mathfrak p_1:=\mathfrak p_{\Omega_*}:\Sigma\to M_1:=M_{\Omega_*}
\]
and let $\om_*\in\Xi^2(M_1)$ be the symplectic form such that $\Omega_*=\mathfrak p_1^*\om_*$. The space $\mathfrak P^0(\Sigma)$ is the connected component of $\mathfrak p_1$ in $\mathfrak P(\Sigma)$ and $\mathfrak Z_C^0(\Sigma)$ is the corresponding space of weakly Zoll pairs, which is non-empty and non-degenerate thanks to Remark \ref{r:zollnondeg}. Namely, condition \eqref{e:nd} is automatically satisfied and we can use all the results contained in Section \ref{s:acfun}.

We fix a reference weakly Zoll pair $(\mathfrak p_0,c_0)\in \mathfrak Z_C^0(\Sigma)$, where 
\[
\mathfrak p_0:\Sigma\to M_0
\]
and let $\om_0$ be some form in $\Xi^2_{c_0}(M_0)$. Let $\mathfrak h\in[S^1,\Sigma]$ denote the free-homotopy class of the $\mathfrak p_0$-fibres. We define the volume and the action functionals
\[
\Fvol:\Xi^2_C(\Sigma)\to\R,\qquad \mathcal A:\mathfrak Z_C^0(\Sigma)\to\R,\qquad \widetilde{\mathcal A}_\Omega:\widetilde{\Lambda}_{\mathfrak h}(\Sigma)\to\R
\] 
with respect to $(\mathfrak p_0,c_0)$. Due to Proposition \ref{p:changepol}, we have 
\begin{equation}\label{e:pmonotone}
\frac{\di P}{\di A}(A_*)=\langle [\om_*]^{n},[M_1]\rangle>0,\qquad 
A_*:=\mathcal A(\mathfrak p_1,\om_*).
\end{equation}

\begin{dfn}\label{dfn:admissible_cover}
We say that a finite cover $\{B_i\}$ of $M_1$ is \textbf{admissible}, if all the sets $B_i$ and their intersections $B_i\cap B_j$ are homeomorphic to the open Euclidean ball, if non-empty. We write $\{\Sigma_i:=\mathfrak p_1^{-1}(B_i)\}$ for the pull-back cover of $\Sigma$, so that $\Sigma_i\cap \Sigma_j$ retracts to a $\mathfrak p_1$-fibre, if non-empty. We readily see that admissible finite covers exist and we fix one of them throughout the present section.
\end{dfn}
We also consider another oriented $S^1$-bundle \label{dfn:finitecovering}$\mathfrak p_0^\veee:\Sigma^\veee\to M_0^\veee$ with bundle map $\Pi:\Sigma^\veee\to\Sigma$ such that $\mathfrak p_0^\veee=\Pi^*\mathfrak p_0$. We further assume that $\Pi$ is a \textbf{finite covering map}. This is equivalent to asking that the corresponding quotient map $\pi_0:M_0^\veee\to M_0$, which fits into the first commutative diagram in \eqref{eq:commutative_diagram2}, is a finite covering map. Thus, we endow $M_0^\veee$ and $\Sigma^\veee$ with the orientations pulled back by $\pi_0$ and $\Pi$, respectively.

The pull-back form $\Omega_*^\veee:=\Pi^*\Omega_*$ is Zoll odd-symplectic on $\Sigma^\veee$ and defines the oriented $S^1$-bundle $\mathfrak p_1^\veee:=\Pi^*\mathfrak p_1:\Sigma^\veee\to M_1^\veee$. Since $\mathfrak p_0$ and $\mathfrak p_1$ are connected by a path of oriented $S^1$-bundles, the quotient map $\pi_1:M_1^\veee\to M_1$ is also an orientation-preserving finite covering map, which fits into the first commutative diagram in \eqref{eq:commutative_diagram2}. Moreover, $\pi_1$ evenly covers all the sets in $\{B_i\}$, as they are contractible. Therefore, the family $\{B_j^\veee\}$, whose elements are the connected components of the pre-images $\pi_1^{-1}(B_i)$, yields a finite admissible cover of $M_1^\veee$. 

\subsection{A local primitive for the action form}\label{ss:locprim}

We have seen that the action form does not admit a global primitive on $\Lambda_{\mathfrak h}(\Sigma)$ in general. However, we can always find a local primitive on the space of periodic curves that are close to $\jmath_{\mathfrak p_1}(\Sigma)$. Let $\Lambda(\mathfrak p_1)$ be the subset of $\Lambda_{\mathfrak h}(\Sigma)$, whose elements $\gamma$ are contained in some $\Sigma_i$ and are homotopic within $\Sigma_i$ to some $\mathfrak p_1$-fibre. This means that $\gamma$ admits a {\bf short  homotopy} $\{\gamma^\mathrm{short}_r\}$. Namely, such a homotopy has the properties
\[
\gamma^\mathrm{short}_r\subset\Sigma_i,\ \ \forall\, r\in[0,1],\qquad \gamma^\mathrm{short}_0=\jmath_{\mathfrak p_1}(z_i),\ \ \gamma^\mathrm{short}_1=\gamma,
\]
for some $z_i\in\Sigma_i$.  We choose any path $\{\mathfrak p_r\}$ from $\mathfrak p_0$ and $\mathfrak p_1$ in $\mathfrak P^0(\Sigma)$, and consider the map 
\[
\Lambda(\mathfrak p_1)\to \widetilde{\Lambda}_{\mathfrak h}(\Sigma),\qquad \gamma\mapsto 
\big[\{\jmath_{\mathfrak p_r}(z_i)\}\#\{\gamma_r^\mathrm{short}\}\big]
\]
Since the intersection of two elements of the finite admissible cover is contractible, this map does not depend on the particular short homotopy nor on the set $\Sigma_i$. Composing this map with the  functional $\widetilde{\mathcal A}_\Omega$ on $\widetilde{\Lambda}_{\mathfrak h}(\Sigma)$, we define the action functional on $\Lambda(\mathfrak p_1)$:
\begin{equation}\label{e:aclocdef}
\mathcal A_\Omega:\Lambda(\mathfrak p_1)\to\R\,,\qquad \mathcal A_\Omega(\gamma):=\widetilde{\mathcal A}_\Omega\Big(\big[\{\jmath_{\mathfrak p_r}(z_i)\}\#\{\gamma_r^\mathrm{short}\}\big]\Big).
\end{equation}
We note that $\di\mathcal A_\Omega=\mathfrak a(\Omega)$ on $\Lambda(\mathfrak p_1)$. 

\begin{lem}
The action value $\mathcal A_\Omega(\gamma)$ does not depend on the choice of path $\{\mathfrak p_r\}$.
\end{lem}
\begin{proof}
If $\{\mathfrak p_r'\}$ is another path in $\mathfrak P^0(\Sigma)$ with $\mathfrak p'_0=\mathfrak p_0$ and $\mathfrak p'_1=\mathfrak p_1$, then, by \eqref{e:actiontilde2} and \eqref{eq:vanishing_action},
\[
\widetilde{\mathcal A}_\Omega\Big(\big[\{\jmath_{\mathfrak p_r}(z_i)\}\#\{\gamma_r^\mathrm{short}\}\big]\Big)-\widetilde{\mathcal A}_\Omega\Big(\big[\{\jmath_{\mathfrak p_r'}(z_i)\}\#\{\gamma_r^\mathrm{short}\}\big]\Big)={\mathcal A}(\mathfrak p_0,c_0)=0.\qedhere
\]
\end{proof}

We write 
$
\mathcal X(\Omega;\mathfrak p_1):=\mathcal X(\Omega)\cap\Lambda(\mathfrak p_1)
$
 for the set of closed characteristics of $\Omega$  in $\Lambda(\mathfrak p_1)$. The {\bf systole} over the set $\Lambda(\mathfrak p_1)$ is the functional
\begin{equation}\label{e:amindef}
\mathcal A_{\min}:\Xi^2_C(\Sigma)\to \R\cup\{-\infty,+\infty\},\qquad \mathcal A_{\min}(\Omega):=\inf_{\gamma\in\mathcal X(\Omega;\mathfrak p_1)}\mathcal A_\Omega(\gamma),
\end{equation}
while the {\bf diastole} over the set $\Lambda(\mathfrak p_1)$ is the functional 
\begin{equation}\label{e:amaxdef}
\mathcal A_{\max}:\Xi^2_C(\Sigma)\to \R\cup\{-\infty,+\infty\}\,,\qquad \mathcal A_{\max}(\Omega):=\sup_{\gamma\in\mathcal X(\Omega;\mathfrak p_1)}\mathcal A_\Omega(\gamma).
\end{equation}
\begin{prp}\label{p:pbinvactloc}
Let $\Pi:\Sigma^\veee\to\Sigma$ be a bundle map with $\mathfrak p_0^\veee=\Pi^*\mathfrak p_0$, which is also a finite covering map. Let us consider the oriented $S^1$-bundle $\mathfrak p_1^\veee=\Pi^*\mathfrak p_1:\Sigma^\veee\to M_1^\veee$. Then, $\Pi_*:(\gamma^\veee\in\Lambda(\mathfrak p_1^\veee))\mapsto (\Pi\circ\gamma^\veee\in \Lambda(\mathfrak p_1))$ is a (surjective) finite covering map and there holds
\begin{equation*}
\mathcal A^\veee_{\Pi^*\Omega}=\mathcal A_{\Omega}\circ \Pi_*,\qquad \forall\,\Omega\in\Xi^2_C(\Sigma),
\end{equation*}
where $\mathcal A^\veee_{\Pi^*\Omega}$ is the action functional on $\Lambda(\mathfrak p_1^\veee)$ with reference pair $(\mathfrak p_0^\veee,c_0^\veee)$. 
As a consequence, the restriction map $\Pi_*:\mathcal X(\Pi^*\Omega;\mathfrak p_1^\veee)\rightarrow\mathcal X(\Omega;\mathfrak p_1)$ is also a finite cover and we have
\[
\mathcal A^\veee_{\min}(\Pi^*\Omega)=\mathcal A_{\min}(\Omega),\qquad \mathcal A_{\max}^\veee(\Pi^*\Omega)=\mathcal A_{\max}(\Omega).
\]
\end{prp}
\begin{proof}
The statement follows from Proposition \ref{p:pbaction}, since $\Pi^*$ lifts a path $\{\mathfrak p_r\}$ in $\mathfrak P^0(\Sigma)$ to a path $\{\mathfrak p_r^\veee\}$ in $\mathfrak P^0(\Sigma^\veee)$ and $\Pi$ maps short homotopies for $\mathfrak p_1^\veee$ to short homotopies for $\mathfrak p_1$.
\end{proof}
We recall some results from \cite{Gin87}, ensuring that the systole and the diastole of $\Omega$ are finite, if $\Omega$ is an odd-symplectic form sufficiently close to $\Omega_*$ in the same cohomology class. As usual, we measure distances and norms by fixing Riemannian metrics on the relevant spaces.

\begin{prp}\label{p:ginsys}
There exists a $C^1$-neighbourhood $\mathcal V$ of $\Omega_*$ in $\mathcal S_C(\Sigma)$ such that, for all $\Omega\in\mathcal V$, the following properties hold:
\begin{enumerate}[(i)]
\item There is a constant $C>0$ depending on $\mathcal V$ but not on $\Omega$ such that
\begin{equation*}
\sup_{t\in\R}\Big(\dist\big(q_\gamma(0),q_\gamma(t)\big)\Big)\leq C\Vert\Omega-\Omega_*\Vert_{C^0},\qquad q_\gamma:=\mathfrak p_1\circ\gamma,\qquad \forall\, \gamma\in \mathcal X(\Omega;\mathfrak p_1).
\end{equation*} 
\item The set $\mathcal X(\Omega;\mathfrak p_1)$ is compact and non-empty and the restrictions $\mathcal A_{\min}|_{\mathcal V}:\mathcal V\to\R$, $\mathcal A_{\max}|_{\mathcal V}:\mathcal V\to\R$ are Lipschitz-continuous real functions. 
\item If $\Omega\in\mathcal V$ is a Zoll form, then all its closed characteristics (up to orientation) lie in $\Lambda(\mathfrak p_1)$, i.e.~$\mathcal X(\Omega)=\mathcal X(\Omega;\mathfrak p_1)$, and, recalling the notation $\mathcal A(\Omega)=\mathcal A(\mathfrak p_\Omega,[\om_\Omega])$,
\[
\mathcal A_{\min}(\Omega)=\mathcal A(\Omega)=\mathcal A_{\max}(\Omega).
\]
\end{enumerate}
\end{prp}
\begin{proof}
Let $\eta\in \mathcal K(\mathfrak p_1)$, and let $V$ be the vertical vector field with $\eta(V)=1$. If $\Omega$ is odd-symplectic, the equation 
\[
\iota_{X_\Omega}(\eta\wedge\Omega_*^{n})=\Omega^{n}
\]
determines a nowhere vanishing section $X_\Omega\in\ker\Omega$. The periodic orbits of the flow $\Phi^{X_\Omega}$ yield elements in $\mathcal X(\Omega)$. Since $V=X_{\Omega_*}$, we have
\begin{equation*}
\|X_\Omega-V\|_{C^1}\leq C_1\Vert \Omega-\Omega_*\Vert_{C^1}
\end{equation*}
for some constant $C_1>0$. Using this estimate, one can argue as in \cite[Lemma 3.3, Proposition 3.4]{BK19a} to show that there exists a $C^1$-neighbourhood $\mathcal V\subset\mathcal S_C(\Sigma)$ of $\Omega_*$ such that, for all $\Omega\in\mathcal V$ item (i) holds and $T(\gamma)\leq 2$, where $T(\gamma)$ is the period of $\gamma\in\mathcal X(\Omega;\mathfrak p_1)$. 

One can follow \cite[Section III]{Gin87} or \cite[Section 3.2]{APB14} to show that there exists a Lipschitz-continuous map $\mathcal Y:\mathcal V\times\Sigma\to \Lambda(\mathfrak p_1)$ with the following properties. If one defines, for every $\Omega\in\mathcal V$, the function $S(\Omega):\Sigma\to\R$ by $S(\Omega)(z):=\mathcal A_\Omega\big(\mathcal Y(\Omega,z)\big)$ for all $z\in\Sigma$, then one gets a Lipschitz function $S:\mathcal V\to C^1(\Sigma)$ such that
\[
\Crit S(\Omega) \longrightarrow\mathcal X(\Omega;\mathfrak p_1),\qquad z\mapsto \mathcal Y(\Omega,z)
\]
is a one-to-one correspondence. In particular,
\[
\min S(\Omega)=\mathcal A_{\min}(\Omega),\qquad\max S(\Omega)=\mathcal A_{\max}(\Omega).
\] 
Since the functions $\max,\min:C^1(\Sigma)\to\R$ are Lipschitz-continuous, this shows (ii).

To prove (iii), let $\rho_\mathrm{Leb}>0$ be a Lebesgue number for the open cover $\{B_i\}$ of $M_1$ and let $C$ be the constant given in (i). Up to shrinking $\mathcal V$, we can assume that
\begin{equation*}
C\Vert\Omega-\Omega_*\Vert_{C^1}\leq\rho_\mathrm{Leb},\qquad \forall\, \Omega\in\mathcal V.
\end{equation*}
Let now $\Omega\in\mathcal V$ be a Zoll form with associated bundle $\mathfrak p_\Omega$, and consider the set
\[
\Sigma(\Omega;\mathfrak p_1):=\big\{z\in\Sigma\ \big|\ \{t\mapsto\Phi_t^{X_\Omega}(z)\}\in \Lambda(\mathfrak p_1)\big\},
\]
which is non-empty since we proved that $\mathcal X(\Omega;\mathfrak p_1)$ is non-empty. This set is open, as the sets $B_i$ are open. Furthermore, from item (i) and the inequality $C\Vert\Omega-\Omega_*\Vert_{C^1}\leq\rho_\mathrm{Leb}$, it is also closed. Since $\Sigma$ is connected, we deduce $\Sigma(\Omega;\mathfrak p_1)=\Sigma$ as desired.

To prove the equality between the actions in (iii), we observe that, since $\mathfrak p_\Omega$ is $C^1$-close to $\mathfrak p_1$, there exist a path $\{\mathfrak q_r\}$ with $\mathfrak q_0=\mathfrak p_1$, $\mathfrak q_1=\mathfrak p_\Omega$ and some $z_i\in\Sigma_i$ such that all the loops $\jmath_{\mathfrak q_r}(z_i)$ are contained in $\Sigma_i$, and, hence, build a short path. Therefore,
\[
\mathcal A_\Omega(\jmath_{\mathfrak p_\Omega}(z_i))=\widetilde{\mathcal A}_\Omega\Big(\big[\{\jmath_{\mathfrak p_r}(z_i)\}\#\{\jmath_{\mathfrak q_r}(z_i)\}\big]\Big)=\widetilde{\mathcal A}_\Omega\Big(\big[\{\jmath_{(\mathfrak p\#\mathfrak q)_r}(z_i)\}\big]\Big)=\mathcal A(\Omega),
\]
where $\{(\mathfrak p\#\mathfrak q)_r\}=\{\mathfrak p_r\}\#\{\mathfrak q_r\}$. As far as the existence of $\{\mathfrak q_r\}$ is concerned, the first statement in Lemma \ref{l:triv} yields a diffeomorphism $\Psi:\Sigma\to\Sigma$, which is $C^1$-close to $\id_\Sigma$ and satisfies $\Psi^*\mathfrak p_\Omega=\mathfrak p_1$. As a consequence, there exists an isotopy $\{\Psi_r\}$ of diffeomorphisms $C^1$-close to $\id_\Sigma$ with $\Psi_1=\Psi$. The path $\mathfrak q_r:=(\Psi_r^{-1})^*\mathfrak p_0$ has the desired properties.
\end{proof}
It seems appropriate at this point to recall Conjecture \ref{con:lsios} formulated in the Introduction.
\begin{con*}
If $\Omega_*\in \ZZ_C(\Sigma)$ is a Zoll odd-symplectic form, there is a $C^{k-1}$-neighbourhood $\mathcal U$ of $\Omega_*$ in $\mathcal S_C(\Sigma)$, for some $k\geq 2$, such that
\begin{equation*}
P(\mathcal A_{\min}(\Omega))\leq\Fvol(\Omega)\leq P(\mathcal A_{\max}(\Omega)),\qquad \forall\,\Omega\in \mathcal U.
\end{equation*}
The equality holds in any of the two inequalities above if and only if $\Omega$ is Zoll.
\end{con*}

The order of the three terms in the inequality stems from the fact that $P$ is monotone increasing in a neighbourhood of $A_*$ by \eqref{e:pmonotone}.
Theorem \ref{thm:mainaction} and Proposition \ref{p:ginsys}.(iii) imply that if $\Omega$ is Zoll, then both equalities hold.
\begin{cor}
If $\mathcal V\subset \mathcal S_C(\Sigma)$ is the $C^1$-neighbourhood of a Zoll odd-symplectic form $\Omega_*$ given by Theorem \ref{p:ginsys}, then there holds
\begin{equation*}
P(\mathcal A_{\min}(\Omega))=\Fvol(\Omega)= P(\mathcal A_{\max}(\Omega)),\qquad \forall\,\Omega\in \mathcal V\cap\mathcal Z_C(\Sigma).\eqno\qed
\end{equation*}
\end{cor}
\begin{rmk}\label{r:conje0}
If the real Euler class vanishes, then the inequality in Conjecture 2 becomes
\[
\mathcal A_{\min}(\Omega)\leq 0\leq \mathcal A_{\max}(\Omega),\qquad \forall\,\Omega\in \mathcal U.
\]
Indeed, in this case $\Fvol\equiv 0$, $P(A)=\langle [\om_*]^n,M_1]\rangle A$ and $\langle [\om_*]^n,M_1]\rangle>0$ by \eqref{e:pmonotone}.
\end{rmk}

\begin{rmk}\label{r:indepconj}
The conjecture is independent of the chosen reference pair $(\mathfrak p_0,c_0)\in\mathfrak Z_C^0(\Sigma)$. Indeed, if $(\mathfrak p_0',c_0')\in\mathfrak Z_C^0(\Sigma)$ is another pair with corresponding objects $\Fvol'$, $P'$ and $\mathcal A_\Omega'$, then
\[
P'\circ\mathcal A_\Omega'-\Fvol'(\Omega)=P\circ \mathcal A_\Omega-\Fvol(\Omega)
\]
by Corollary \ref{c:changeofbase}.
\end{rmk}

In the next subsection, we will see that $P$ can be written as the integral over $M_1$ of a function $\mathcal P:M_1\times \R\to\R$ so that the monotonicity of $P$ corresponds to the monotonicity of $\mathcal P$ in the second variable for an interval $\mathcal I$ containing $A_*$. This will enable us to establish the conjecture in some simple cases. Before doing that, we end this subsection by observing that Conjecture \ref{con:lsios} is also consistent with the pull-back operation.
\begin{prp}\label{p:covercon}
Let $\Pi:\Sigma^\veee\to\Sigma$ be a bundle map with $\mathfrak p_0^\veee=\Pi^*\mathfrak p_0$, which is also a finite covering map. If Conjecture \ref{con:lsios} holds for $\Omega_*^\veee=\Pi^*\Omega_*\in\mathcal S_{C^\veee}(\Sigma^\veee)$, then it holds for $\Omega_*\in\mathcal S_C(\Sigma)$. 
\end{prp}
\begin{proof}
Let us suppose that the conjecture is true for a neighbourhood $\mathcal U^\veee\subset\mathcal V^\veee$ around $\Omega_*^\veee$, where $\mathcal V^\veee$ is given by Proposition \ref{p:ginsys}. Thus, for all $\Omega^\veee\in\mathcal U^\veee$, there holds
\begin{equation}\label{e:p'a}
P^\veee(\mathcal A_{\min}(\Omega^\veee))\leq \Fvol(\Omega^\veee)\leq P^\veee(\mathcal A_{\max}(\Omega^\veee))
\end{equation}
and any of the two equalities holds if and only if $\Omega^\veee$ is Zoll. We claim that the local systolic inequality holds in the neighbourhood $\mathcal U:=(\Pi^*)^{-1}(\mathcal U^\veee)$ of $\Omega_*$. Indeed, let us take an arbitrary $\Omega$ in $\mathcal U$. This means that $\Pi^*\Omega$ belongs to $\mathcal U^\veee$.  Applying Proposition \ref{p:pbinvactloc}, Proposition \ref{p:pbinvact} and Proposition \ref{p:volinvcov}, it follows that
\[
P(\mathcal A_{\min}(\Omega))=P(\mathcal A_{\min}(\Pi^*\Omega))=\frac{1}{\deg \pi_0}P^\veee(\mathcal A_{\min}(\Pi^*\Omega))\leq\frac{1}{\deg \Pi}\Fvol(\Pi^*\Omega)=\Fvol(\Omega),
\]
where we used  $\deg \Pi=\deg \pi_0>0$ and \eqref{e:p'a}. An analogous inequality holds for $\mathcal A_{\max}(\Omega)$.

Let us assume now, for instance, that $P(\mathcal A_{\min}(\Omega))=\Fvol(\Omega)$. From the computation above, this implies $P^\veee(\mathcal A_{\min}(\Pi^*\Omega))=\Fvol(\Pi^*\Omega)$. Since the conjecture is true for $\mathcal U^\veee$, we see that $\Pi^*\Omega$ is Zoll and we call $\mathfrak p^\veee$ the associated bundle. All the fibres of $\mathfrak p^\veee$ lie in some $\Sigma^\veee_{j}$, by Proposition \ref{p:ginsys}. Since $\Pi$ is injective on such sets, it follows that there exists a bundle $\mathfrak p:\Sigma\to M$ such that $\mathfrak p^\veee=\Pi^*\mathfrak p$, which implies that $\Omega$ is a Zoll form with bundle $\mathfrak p$.  
\end{proof}
\subsection{A second look at H-forms}
For the rest of this section, we assume that $\mathfrak p_0=\mathfrak p_1=\mathfrak p_{\Omega_*}$, and take $\{\mathfrak p_r\}$ as the constant path. This does not cause any loss in generality thanks to Remark \ref{r:indepconj}. We fix a free $S^1$-action $\mathfrak u\in\mathfrak U(\mathfrak p_0)$ with generating vector field $V$. Let us take $\eta\in\mathcal K(\mathfrak u)$ with curvature $\kappa\in\Xi^2(M_0)$. Namely, we have $[\kappa]=e_0$ and $\mathfrak p_0^*\kappa=\di\eta$. By Theorem \ref{thm:mainaction}, there is $\om_0\in\Xi^2_{c_0}(M_0)$ such that
\[
\om_*=A_*\kappa+\omega_0.
\]
Moreover, we can exploit our freedom in choosing $\eta$ and $c_0$ to get
\[
\bullet\ \ \om_*=\om_0,\quad \text{if $e_0=0$},\qquad\qquad \bullet\ \ \om_*=A_*\kappa,\quad\text{if $C=0$}.
\]
Indeed, if $e_0=0$, we just take $\eta$ to be a flat connection, i.e.~$\kappa=0$. If $C=0$, we take $c_0=0$ so that $A_*\neq0$ (as $[\om_*^n]\neq 0$) and pick $\eta\in\mathcal K(\mathfrak u)$ with the property that $\kappa=\tfrac{1}{A_*}\om_*$, which is equivalent to $\om_0=0$. We define the function
\[\label{eq:mathcal_Q}
\mathcal Q:M_0\times \R\to\R,\qquad \mathcal Q(\,\cdot\,,A)\om^{n}_*=(A\kappa+\om_0)^{n},\qquad\forall\, A\in \R.
\]  
The function $\mathcal P:M_0\times \R\to \R$ is determined by the properties
\[\label{eq:mathcal_P}
\bullet\ \ \mathcal P(q,0)=0,\quad\forall\, q\in M_0,\qquad\qquad \bullet\ \ \frac{\partial \mathcal P}{\partial A}(q,A)=\mathcal Q(q,A),\quad\forall\,(q,A)\in M_0\times\R.
\]
There holds
\[
 Q(A)=\int_{M_0}\mathcal Q(\,\cdot\,,A)\,\om^{n}_*,\qquad P(A)=\int_{M_0}\mathcal P(\,\cdot\,,A)\om^{n}_*.
\]
From the definition, $\mathcal Q(\,\cdot\,,A_*)\equiv1$ and we define $\mathcal I\subset \R$ to be the \textit{maximal interval} satisfying
\[\label{eq:maximal_interval}
\bullet\ \ A_*\in\mathcal I,\qquad\qquad \bullet\ \ \min_{q\in M_0}\mathcal Q(q,A)>0,\quad \forall\, A\in\mathcal I.
\]
Therefore, $\mathcal P$ is strictly increasing in the second coordinate on $M_0\times \mathcal I$. This fact will be crucially used in Proposition \ref{p:qa1}.
\begin{rmk}\label{rmk:maximal_interval}
If $C=0$, then $\mathcal Q(\cdot,A)=(A/A_*)^{n}$, $\mathcal P(\cdot,A)=\tfrac{A_*}{n+1}(A/A_*)^{n+1}$, which implies $\mathcal I=(0,+\infty)$ or $\mathcal I=(-\infty,0)$ depending on the sign of $A_*$. If $e_0=0$, then $\mathcal Q\equiv1$ and $\mathcal P(\,\cdot\,,A)\equiv A$, which implies $\mathcal I=\R$.
\end{rmk}

We now make use of the notion of H-form from Section \ref{s:odddef}. The H-forms we use below are associated with $\Omega_0$, $\alpha_0=0$ and $\sigma_0=\eta$, the given $S^1$-connection. Namely, we deal with forms of the type 
\[
\Omega_{H\eta}=\Omega_0+\di(H\eta)\in \Xi^2_C(\Sigma),\qquad H:\Sigma\to\R.
\]

\begin{rmk}\label{r:fiber}
Let us describe the fibres of the map $H\mapsto \Omega_{H\eta}$. To this end, we compute
\begin{equation}\label{e:dH2}
\di(H\eta)=\di^{h}H\wedge\eta+H\mathfrak p_0^*\kappa.
\end{equation}
Here $\di^hH$ is the horizontal part of $\di H$:
\[
\di^{h}H:=\di H-\di H(V)\eta.
\]
Let us suppose now that $\di(H\eta)=0$. From \eqref{e:dH2}, we deduce
\[
\di^hH=0,\qquad H\mathfrak p_0^*\kappa=0.
\]
The first relation tells us that $H$ is invariant along curves tangent to $\ker\eta$.

Thus, if $e_0\neq0$, then $H$ is constant, as the holonomy in this case is the whole $S^1$. Evaluating the second relation at a point $q\in M_0$, where $\kappa_q\neq0$, we conclude that $H=0$.

On the other hand, if $e_0=0$, then we can take the quotient $\Pi_k:\Sigma\to\Sigma_k$ by the holonomy of $\eta$, which is a finite subgroup of $S^1$ (see Lemma \ref{l:3state1}). The bundle $\Sigma_k\to M_0$ is trivial and admits an angular function $\phi_k:\Sigma_k\to S^1$. We define $\phi:\Sigma\to S^1$ as $\phi:=\phi_k\circ\Pi_k$. We conclude that for each $H$ with $\di ^h H=0$, there exists $\bar H:S^1\to \R$ such that $H=\bar H\circ\phi$. 
\end{rmk}
In our setting, the volume of an H-form can be computed in terms of the integration operator associated with the free $S^1$-action $\mathfrak u$:
\begin{equation*}\label{dfn:ustar}
\mathfrak u_*:C^0(\Sigma)\to C^0(M_0),\qquad \mathfrak u_*(K)(q):=(\mathfrak p_0)_*(K\eta)(q)=\int_{S^1}K\big(\Phi^V_t(z)\big)\,\di t, \quad\forall\, q\in M_0,
\end{equation*}
where $z$ is any point in $\mathfrak p_0^{-1}(q)$ and $(\mathfrak p_0)_*$ denotes the integration along the $\mathfrak p_0$-fibres. 
\begin{lem}\label{l:volqa}
For a function $H:\Sigma\to\R$, there holds
\[
\Vol(H\eta)=\int_{M_0}\mathfrak u_*\big(\mathcal P(\mathfrak p_0,H)\big)\,\omega_{*}^{n}.
\]
\end{lem}
\begin{proof}
Recalling the definition of $\Vol$ from \eqref{e:defvol} and using Fubini's Theorem, we compute
\begin{align*}
\Vol(H\eta)=\int_0^1\di r\int_\Sigma H\eta\wedge \Omega_{rH\eta}^{n}&=\int_0^1\di r\int_{\Sigma}H \eta\wedge\Big(\mathfrak p_0^*\om_0+rH\mathfrak p_0^*\kappa\Big)^{n}\\&=\int_0^1\di r\int_{\Sigma}H\eta\wedge \mathcal Q(\mathfrak p_0,rH) \big(\mathfrak p_0^*\om_*\big)^{n}\\
&=\int_\Sigma\Big(\int_0^1H\mathcal Q(\mathfrak p_0,rH)\,\di r\Big)\eta\wedge \big(\mathfrak p_0^*\om_*\big)^{n}\\
&=\int_\Sigma\mathcal P(\mathfrak p_0,H) \eta\wedge \big(\mathfrak p_0^*\om_*\big)^{n}\\
&=\int_{M_0} (\mathfrak p_0)_*\big(\mathcal P(\mathfrak p_0,H) \eta\big)\wedge \om^{n}_*\\
&=\int_{M_0}\mathfrak u_*\big(\mathcal P(\mathfrak p_0,H)\big)\,\om^{n}_*,
\end{align*}
where in the second equality, we have used that $H\eta\wedge\di H\wedge\eta=0$.
\end{proof}
\begin{rmk}
When $e_0=0$, defining the Calabi invariant as the volume of $H\eta$
\[
\CAL_{\om_*}(H):=\Vol(H\eta)=\int_{M_0}\mathfrak u_*(H)\,\om^{n}_*,
\]
recovers the classical definition of the Calabi invariant for the trivial bundle $M_0\x S^1\to M_0$. We say that $H$ is {\bf normalised} with respect to $\om_*$ if its Calabi invariant $\CAL_{\om_*}(H)$ vanishes, which is equivalent to requiring that the one-form $H\eta$ is normalised according to Definition \ref{d:normalized}. When $H=h\circ\phi$, as in Remark \ref{r:fiber}, then $H$ is normalised if and only if the integral of $h$ is zero.
\end{rmk}
\begin{lem}
If a function $H:\Sigma\to \R$ takes values in $\mathcal I$, the H-form $\Omega_{H\eta}$ is odd-symplectic.
\end{lem}
\begin{proof}
To prove the statement, we show that $\Omega_{H\eta}$ is non-degenerate on the hyperplane distribution $\ker\eta$. We compute preliminarily 
\begin{equation*}
\Omega_0+\di(H\eta)=\Omega_0+\di^h H\wedge\eta+H\mathfrak p_0^*\kappa=(\mathfrak p_0^*\om_0+H\mathfrak p_0^*\kappa)+\di^h H\wedge\eta.
\end{equation*}
Let $z\in\Sigma$ be arbitrary, and set $q:=\mathfrak p_0(z)$. We have an inverse $\mathcal L_z:\ta_qM_0\to (\ker\eta)_z$ for the projection $\di\mathfrak p_0:(\ker\eta)_z\to\ta_qM_0$, which we can use to pull back the restriction of forms on $\ker\eta$ to the tangent space of $M_0$:
\begin{align*}
\mathcal L_z^*\Big(\Omega_0+\di(H\eta)\Big)_{\!z}^{n}=\Big(\mathcal L_z^*\big(\Omega_0+\di(H\eta)\big)_z\Big)^{n}&=\Big(\mathcal L_z^*\big(\mathfrak p_0^*\om_0+H\mathfrak p_0^*\kappa\big)_z\Big)^{n}\\
&=\Big(\om_0+H(z)\kappa\Big)_{\!q}^{n}\\
&=\mathcal Q(q,H(z))(\om_{*}^{n})_q,
\end{align*}
where we used $\mathcal L_z^*(\di^h H\wedge\eta)=0$. The last form is non-degenerate since $\mathcal Q(q,H(z))>0$.
\end{proof}
Combining this lemma with the stability property proved in Proposition \ref{p:Moser}, we arrive at the following corollary.
\begin{cor}\label{c:Moser2}
Suppose that $\mathfrak p_0:\Sigma\to M_0$ is an oriented $S^1$-bundle with $e_0=0$ and choose $\eta$ to be a flat connection for $\mathfrak p_0$. Let $\Omega_0=\mathfrak p_0^*\om_0$ be given, where $\om_0$ is some symplectic form on $M_0$. For every sufficiently small $C^2$-neighbourhood $\mathcal U$ of $\Omega_0$ in $\mathcal S_C(\Sigma)$, there exist a $C^2$-neighbourhood $\mathcal D$ of $\id_\Sigma$ in $\mathrm{Diff}(\Sigma)$ and a $C^2$-neighbourhood $\mathcal H$ of the zero function in the space of normalised functions on $\Sigma$ with the following property:
\[
\forall\, \Omega\in\mathcal U, \ \exists\, \Psi\in\mathcal D, \ \exists\, H\in \mathcal H,\qquad \Psi^*\Omega=\Omega_0+\di(H\eta).
\]
\end{cor}
\begin{proof}
Let $\Omega\in\mathcal S_C(\Sigma)$ be $C^2$-close to $\Omega_0$. By standard elliptic arguments (see for instance \cite[Chapter 10]{Nic07}), we have $\Omega=\Omega_0+\di\alpha$ for a $C^2$-small normalised one-form $\alpha$. We apply Proposition \ref{p:Moser} with $\alpha_0=0$, $\sigma_r\equiv\eta$ and $\alpha_r=r\alpha$ so that $\Omega_r=\Omega_0+r\di\alpha=\Omega_0+r(\Omega-\Omega_0)$. Indeed, the hypothesis in Remark \ref{r:stabil}.(b) are met, as $\Omega$ is $C^0$-close to $\Omega_0$ and $\sigma_r\equiv\eta$ is closed. Thus, we have the existence of a diffeomorphism $\Psi:\Sigma\to\Sigma$ isotopic to $\id_\Sigma$ and of a normalised function $H:\Sigma\to \R$ such that $\Psi^*\Omega=\Omega_0+\di(H\eta)$. More explicitly, from \eqref{e:stab1} and \eqref{e:stab2}, we see that we have paths $\{\Psi_r\}$ and $\{H_r\}$ with $\Psi_0=\id_\Sigma$ and $H_0=0$ satisfying
\[
X_r\in\ker\eta,\qquad\big(\iota_{X_r}\big(\Omega_0+r(\Omega-\Omega_0)\big)+\alpha\big)\big|_{\ker\eta}=0,\qquad \dot H_r=\alpha(V)\circ \Psi_r,
\]
where $X_r$ is the time-dependent vector field generating $\Psi_r$ and $V$ is the vertical vector field of $\eta$. From these relations, we see that $\Psi_r$ is $C^2$-close to $\id_\Sigma$ and $H_r$ is $C^2$-small.
\end{proof}

\subsection{The local systolic-diastolic inequality for quasi-autonomous H-forms}
We introduce H-forms of special type for which we can prove the systolic-diastolic inequality.
\begin{dfn}\label{d:qa}
We say that an H-form $\Omega_{H\eta}$ is \textbf{quasi-autonomous}, if there are pairs $q_{\min},q_{\max}\in M_0$ and $H_{\min},H_{\max}:\Sigma\to \R$ such that
\begin{align*}
(i)&\quad \Omega_{H_{\min}\eta}=\Omega_{H\eta}=\Omega_{H_{\max}\eta}\\
(ii)&\quad \min_{\Sigma} H_{\min}= H_{\min}(z),\ \ \forall\, z\in\mathfrak p^{-1}_0(q_{\min}),\qquad \max_{\Sigma} H_{\max}= H_{\max}(z),\ \ \forall\,z\in\mathfrak p_0^{-1}(q_{\max}).
\end{align*}
We say that $H:\Sigma\to\R$ is quasi-autonomous, if the form $\Omega_{H\eta}$ is quasi-autonomous.  Thanks to Remark \ref{r:fiber}, we have $H_{\max}=H=H_{\max}$ if $e_0\neq0$. We recover the standard definition of quasi-autonomous Hamiltonians \cite[p.~186]{HZ11} if $\mathfrak p_0$ is trivial.
\end{dfn}

\begin{prp}\label{p:qa1}
Let $\Omega_*\in\ZZ_C(\Sigma)$ be a Zoll odd-symplectic form with bundle $\mathfrak p_{\Omega_*}=\mathfrak p_0$. When $e_0=0$, the reference connection one-form $\eta$ for $\mathfrak p_0$ is assumed to be flat . If $H:\Sigma\to\R$ is quasi-autonomous with values in $\mathcal I$ and satisfies $\mathcal A_{\min}(\Omega_{H\eta})\in\mathcal I$, $\mathcal A_{\max}(\Omega_{H\eta})\in \mathcal I$, then
\[
P(\mathcal A_{\min}(\Omega_{H\eta}))\leq \Fvol(\Omega_{H\eta})\leq P(\mathcal A_{\max}(\Omega_{H\eta})).
\]
Moreover, let us consider the following three conditions:
\begin{enumerate}[{\normalfont (a)}]
\item $H$ is constant when $e_0\neq0$, or factors through the map $\phi$ of Remark \ref{r:fiber} when $e_0=0$.
\item The form $\Omega_{H\eta}$ is Zoll.
\item Any of the two inequalities above is an equality.
\end{enumerate}
We have $(\mathrm a)\Rightarrow(\mathrm b)$ and $(\mathrm a)\Leftrightarrow(\mathrm c)$. The implication $(\mathrm b)\Rightarrow(\mathrm c)$ holds if $\Omega_{H\eta}$ belongs to the neighbourhood $\mathcal V$ given in Proposition \ref{p:ginsys} or if
\[
\mathcal A_{\max}(\Omega_{H\eta})-\mathcal A_{\min}(\Omega_{H\eta})\, <\, \inf\Big\{a>0\ \Big|\ a=\langle\mathfrak a(C),[\gamma_r]\rangle,\ \ \text{for some}\ \ [\gamma_r]\in\pi_1(\Lambda_{\mathfrak h}(\Sigma))\Big\}.
\]
\end{prp}
\begin{proof}
From \eqref{e:dH2}, we see that the oriented $\mathfrak p_0$-fibre over $q_{\min}$ is a closed characteristic for $\Omega_{H\eta}$ with action $\min H_{\min}$. Therefore,
\begin{equation}\label{e:chaindiseq}
\mathcal A_{\min}(\Omega_{H\eta})\leq \min H_{\min}\leq H_{\min}. 
\end{equation}
Since $\mathcal P$ is increasing in the second coordinate $A$ when $A\in\mathcal I$, we have by Lemma \ref{l:volqa}
\begin{equation}\label{e:pamin}
\begin{split}
P(\mathcal A_{\min}(\Omega_{H\eta}))\leq P(\min H_{\min})&= \int_{M_0}\mathfrak u_*\big(\mathcal P(\mathfrak p_0,\min H_{\min})\big)\,\om_*^{n}\\
&\leq \int_{M_0}\mathfrak u_*\big(\mathcal P(\mathfrak p_0,H_{\min})\big)\,\om_*^{n}\\[.5ex]
&= \Fvol(\Omega_{H\eta}).
\end{split}
\end{equation}
In the case of $e_0=0$, this is true since $\mathcal I=\R$ (see Remark \ref{rmk:maximal_interval}) and we can assume that $H_{\min}$ is normalised up to adding a constant. 
The inequality for $\mathcal A_{\max}$ is obtained similarly. 

Next, we show the implications between (a), (b), and (c). Assuming (a), we have
\[
\Omega_{H\eta}=\begin{cases}
\mathfrak p^*_0(\om_0+H\kappa),& \text{if}\ e_0\neq0\\[.5ex]
\mathfrak p^*_0\om,&\text{if}\ e_0=0,
\end{cases}
\]
which is Zoll since the value of $H$ is in $\mathcal I$. This shows  $(\mathrm a)\Rightarrow(\mathrm b)$. Moreover, if $e_0\neq0$, then $H=H_{\min}=H_{\max}$ is constant. Therefore, all the inequalities in \eqref{e:chaindiseq} and \eqref{e:pamin} are actually equalities, and (c) follows.  If $e_0=0$, then $H=\bar H\circ \phi$ and the normalised functions $H_{\min}$ and $H_{\max}$ are both equal to zero. By the same reason as before, (a) implies (c).

We now show that $(\mathrm c)$ implies $(\mathrm a)$. Assume for instance, that $P(\mathcal A_{\min}(\Omega_{H\eta}))= \Fvol(\Omega_{H\eta})$. From inequality \eqref{e:pamin} and the fact that $A\mapsto\mathcal P(\,\cdot\,,A)$ is \textit{strictly} increasing for $A\in\mathcal I$, we have that $H_{\min}(z)=\min H_{\min}$, for all $z\in \Sigma$. Therefore, $H_{\min}$ is constant. If $e_0\neq0$, then $H=H_{\min}$ and the conclusion follows. If $e_0=0$, then $H=H_{\min}+(H-H_{\min})$, and therefore, $H$ factors through the map $\phi$, as $H_{\min}$ is constant and $H-H_{\min}$ factors through $\phi$ by Remark \ref{r:fiber}, since $\Omega_{H\eta}=\Omega_{H_{\min}\eta}$.

Finally, let us assume (b), namely that $\Omega_{H\eta}$ is Zoll. If $\Omega_{H\eta}\in\mathcal V$, we know from Proposition \ref{p:ginsys}, that $\mathcal X(\Omega_{H\eta})\subset \Lambda(\mathfrak p_0)$. Therefore, $\mathcal A_{\min}(\Omega_{H\eta})=\mathcal A_{\max}(\Omega_{H\eta})$ due to Proposition \ref{p:ginsys}.(iii) and (c) follows. On the other hand, let 
\[
\epsilon:=\mathcal A_{\max}(\Omega_{H\eta})-\mathcal A_{\min}(\Omega_{H\eta})\geq 0
\]
and suppose that $\epsilon$ is strictly less than the positive generator of $\langle\mathfrak a(C),\pi_1(\Lambda_{\mathfrak h}(\Sigma))\rangle$. Let $\mathfrak p$ be the bundle associated with the Zoll form $\Omega_{H\eta}$, and let $z_0,z_1\in\Sigma$ be two points such that $\jmath_{\mathfrak p}(z_0),\jmath_{\mathfrak p}(z_1)\in\mathcal X(\Omega_{H\eta};\mathfrak p_0)$ and
\[
\mathcal A_{\Omega_{H\eta}}(\jmath_{\mathfrak p}(z_0))=\mathcal A_{\min}(\Omega_{H\eta}),\qquad \mathcal A_{\Omega_{H\eta}}(\jmath_{\mathfrak p}(z_1))=\mathcal A_{\max}(\Omega_{H\eta}).
\]
If $s\mapsto z_s$ is any path connecting $z_0$ to $z_1$, we obtain an element $[\gamma_r]\in\pi_1(\Lambda_{\mathfrak h}(\Sigma))$ by concatenating a short homotopy for $\jmath_{\mathfrak p}(z_0)$, the path $s\mapsto \jmath_{\mathfrak p}(z_s)$, and a reverse short homotopy for $\jmath_{\mathfrak p}(z_1)$. This loop has the property that
\[
\epsilon=\langle \mathfrak a(C),[\gamma_r]\rangle,
\]
which implies that $\epsilon=0$ and condition (c) follows.
\end{proof}
\begin{rmk}
By Proposition \ref{p:ginsys}, a quasi-autonomous function that is $C^2$-close to the constant $A_*=\mathcal A(\Omega_*)$ satisfies the hypotheses of Proposition \ref{p:qa1}. Moreover for such a function, conditions (a), (b), and (c) above are equivalent since the corresponding odd-symplectic form belongs to $\mathcal V$.
\end{rmk}
Combining the result we have just proven with the theory of generating functions on closed symplectic manifolds, we can give a proof of Theorem \ref{t:triv}, namely of the local systolic-diastolic inequality in the case of $e_0=0$.

\subsubsection*{Proof of Theorem \ref{t:triv}}\label{ss:triv}
Let $\Omega_*$ be a Zoll-odd symplectic form such that the associated $S^1$-bundle $\mathfrak p_{\Omega_*}$ has vanishing real Euler class. Due to Remark \ref{r:indepconj} we may assume that $\mathfrak p_{\Omega_*}=\mathfrak p_0:\Sigma\to M_0$ and $e_0=0$. Let $\om_*$ be the symplectic form on $M_0$ such that $\mathfrak p_0^*\om_*=\Omega_*$. By Lemma \ref{l:3state1}, there exists a trivial oriented $S^1$-bundle $\mathfrak p_0^\veee:\Sigma^\veee\to M_0^\veee$ and a bundle map $\Pi:\Sigma^\veee\to \Sigma$ such that $\mathfrak p_0^\veee=\Pi^*\mathfrak p_0$. Therefore, thanks to Proposition \ref{p:covercon}, it is enough to prove Theorem \ref{t:triv} for trivial bundles. Let $\mathfrak p_0:M_0\times S^1\to M_0$ with angular function $\phi$. Assume that $\Omega$ is $C^2$-close to $\Omega_*$. Due to Corollary \ref{c:Moser2}, there is a diffeomorphism $\Psi:\Sigma\to\Sigma$ such that $\Psi=\Psi_1$, where $\{\Psi_r\}$ is a $C^1$-small isotopy with $\Psi_0=\id_{M_0\x S^1}$, and
\[
\Psi^*\Omega=\Omega_{H\di\phi}
\] 
for a $C^2$-small normalised Hamiltonian $H:M_0\times S^1\to\R$. Thus, $\Fvol(\Omega)=0=\Fvol(\Omega_{H\di\phi})$, namely $\CAL_{\om_*}(H)=0$. Moreover, $\Omega$ and $\Omega_{H\di\phi}$ have the same systole and diastole due to Proposition \ref{prp:action_inv_isotopy}, as $[\Psi_r]^*$ preserves the set of short homotopies. The closed characteristics in $\mathcal X(\Omega_{H\di\phi};\mathfrak p_0)$ are curves $\gamma\in\Lambda(\mathfrak p_0)$ of the type $\gamma(t)=(q(t),t)$, for some loop $q:S^1\to M$ with small capping disc $\widehat q:D^2\to M$. Thanks to \eqref{e:actiontilde2}, the action of these curves recovers the classical Hamiltonian action
\[
\mathcal A_{H\di\phi}(\gamma)=\int_{D^2}\widehat q^{\hspace{2pt}*}\om_*+\int_{S^1}H(q(t),t)\di t
\]

Let $\{\varphi_{t}\}_{t\in[0,1]}$ be the Hamiltonian isotopy on $(M_0,\om_*)$ up to time one generated by $H$. The maps $\varphi_t$ are $C^1$-close to the identity, as $H$ is $C^2$-small. From the theory of generating functions on arbitrary symplectic manifolds (see \cite[Proposition 5.11]{LM95} and \cite[Proposition 9.31]{MS98}), $\{\varphi_{t}\}_{t\in[0,1]}$ is homotopic with fixed endpoints to a $C^1$-small Hamiltonian isotopy $\{\varphi'_t\}_{t\in[0,1]}$ generated by a quasi-autonomous normalised Hamiltonian $H':M_0\times S^1\to\R$, namely $\Fvol(\Omega_{H'\di\phi})=0=\CAL_{\om_*}(H')$. Adapting an argument in \cite{Sch00}, we have an action-preserving bijection
\[
\mathcal X(\Omega_{H\di\phi};\mathfrak p_0)\cong \Fix \varphi_1=\Fix \varphi_1'\cong\mathcal X(\Omega_{H'\di\phi};\mathfrak p_0).
\]
Therefore, $\Omega_{H\di\phi}$ and $\Omega_{H'\di\phi}$ have the same systole and diastole. The local systolic-diastolic inequality holds for $\Omega_{H'\di\phi}$ due to Proposition \ref{p:qa1}, and hence also for $\Omega_{H\di\phi}$ and $\Omega$.\qed

\bibliographystyle{amsalpha}
\bibliography{systolic_bib}
\end{document}